\documentclass[11pt,reqno]{amsart}
\usepackage{todonotes}

\title[The one-phase Muskat problem]{Global  well-posedness for the one-phase Muskat problem in 3D}

\author{Hongjie Dong}
\address{Division of Applied Mathematics, Brown University, Providence, RI 02912}
\email{hongjie$\_$dong@brown.edu }

\author{Francisco Gancedo}
\address{Departamento de An\'alisis Matem\'atico $\&$ IMUS, Universidad de Sevilla, Sevilla, Spain}
\email{fgancedo@us.es}

\author{Huy Q. Nguyen}
\address{Department of Mathematics, University of Maryland, College Park, MD 20742}
\email[H. Nguyen]{hnguye90@umd.edu}

\usepackage[margin=1in]{geometry}
\usepackage{amsmath, amsthm, amssymb, mathrsfs, stmaryrd}
\usepackage{times}
\usepackage{color}
\usepackage{hyperref}

\newcommand{\bq}{\begin{equation}}
\newcommand{\eq}{\end{equation}}
\newcommand{\bqa}{\begin{eqnarray*}}
\newcommand{\eqa}{\end{eqnarray*}}

\theoremstyle{plain}
\newtheorem{theo}{Theorem}[section]
\newtheorem{prop}[theo]{Proposition}
\newtheorem{lemm}[theo]{Lemma}
\newtheorem{coro}[theo]{Corollary}
\newtheorem{assu}[theo]{Assumption}
\newtheorem{defi}[theo]{Definition}
\theoremstyle{definition}
\newtheorem{rema}[theo]{Remark}
\newtheorem{nota}[theo]{Notation}

\DeclareMathOperator{\supp}{supp}

\DeclareMathOperator{\Lip}{Lip}

\DeclareSymbolFont{pletters}{OT1}{cmr}{m}{sl}
\DeclareMathSymbol{s}{\mathalpha}{pletters}{`s}

\def\tt{\theta}
\def\eps{\varepsilon}
\def\na{\nabla}

\def\les{\lesssim}

\def\mez{\frac{1}{2}}
\def\tdm{\frac{3}{2}}

\def\Rr{\mathbb{R}}
\def\T{\mathbb{T}}

\def\Zz{\mathbb{Z}}

\def\cF{\mathcal{F}}
\def\cK{\mathcal{K}}

\def\cS{\mathcal{S}}

\def\ld{\lambda}

\def\p{\partial}
\def\na{\nabla}
\def\wc{\rightharpoonup}
\def\ka{\kappa}

\def\wt{\widetilde}

\def\om{\omega}

\def\wp{\wedge}

\def\dv{\text{div}}
\def\sgn{\text{sgn}}

\numberwithin{equation}{section}

\begin{document}
\newcommand{\huy}[1]{{\color{orange} \textbf{H:} #1}}
\begin{abstract}
This paper is concerned with the long time dynamics of the free boundary of a Darcy fluid in three space dimensions, also known as the one-phase Muskat problem. The dynamics of the free boundary is governed by a nonlocal fully nonlinear  parabolic partial differential equation. It is proven that for any periodic Lipschitz graph given as initial data, the problem has a unique global-in-time solution which satisfies the equation in the strong sense. Moreover, all H\"older norms of the solution decay exponentially in time.  These results have been previously established in  two space dimensions. This paper addresses new challenges to extend the results to the more difficult three dimensional setting. The approach developed is critical in three space dimensions and crucially relies on Dahlberg-Kenig's $W^{1, 2+\varepsilon}$ optimal regularity for layer potentials  together with delicate structures of the Dirichlet-to-Neumann operator and layer potentials in Lipschitz domains of $\mathbb{T}^2\times \mathbb{R}$.
\end{abstract}


\noindent

\maketitle

\section{Introduction}
Long time dynamics of free boundary problems is a challenging issue in fluid mechanics. For prototypical problems such as  the free boundary Euler and Navier-Stokes equations, the Muskat problem and the Hele-Shaw problem, it has been proven that local-in-time smooth solutions may develop finite time singularities in certain scenarios. These include the self-intersection of the free boundary  at a single point (splash singularity) for the free boundary Euler \cite{CCFGG2012, CCFGG2013, CoutandShkoller2014} and Navier-Stokes equations \cite{CCFGG2019,  CoutandShkoller2019} and the Muskat problem \cite{CCFG2013}, and the self-intersection of free boundary along an arc  (splat singularity) for the free boundary Euler equations \cite{CCFGG2013, CoutandShkoller2014}. Another type of singularity is the breakdown of smoothness for the Muskat problem \cite{CCFGL2012, CCFG2013-2}. As for the two-dimensional horizontal Helle-Shaw problem, at any positive time the free boundary is either analytic or has a cusp-type singularity \cite{Sakai}.

In view of the above finite time singularities, it is natural to develop well-posedness (existence and uniqueness) theories which can describe the global-in-time dynamics of the solution past singularities. Such a theory is currently not available for the free boundary Euler and 3D Navier-Stokes equations. Regarding the free boundary 2D Navier-Stokes equations with two fluids, global well-posedness for large solutions was obtained in \cite{GG2018, DanchinMucha} for equal viscosities and in \cite{PaicuZhang, GG2022} for small viscosity difference. On the other hand, several approaches have been developed for the horizontal  (no gravity) Hele-Shaw problem. Elliot-Janovsk\'y \cite{Elliot}  formulated the  two-dimensional horizontal Hele-Shaw problem as an elliptic variational inequality for each time and proved  the existence and uniqueness  of global $H^1$ weak solutions. Kim \cite{Kim03} introduced and proved the existence and uniqueness of global continuous viscosity solutions (\`a la Crandall-Lions) to the horizontal Hele-Shaw problem in all dimensions. Although both approaches \cite{Elliot} and \cite{Kim03} do not impose any size restriction on the solution,  the unknowns are defined in the bulk of the fluid (e.g. the fluid pressure) and the obtained regularity is insufficient to make sense of the {\it pointwise dynamics of the free boundary}. The aforementioned result in \cite{Sakai} provides alternatives for regularity of the free boundary of Elliot's variational solutions. On the other hand, a consequence of the works \cite{CJK1, CJK2} concerning Kim's viscosity solutions is that if the initial free boundary is a Lipschitz graph with sufficiently {\it small Lipschitz constant} then the free boundary is a $C^1$ graph at any positive time  and hence its dynamics is satisfied pointwise.

In the present work we are interested in the Muskat problem and our goal is twofold. First, we aim to develop a global well-posedness theory that simultaneously accommodates solutions of arbitrary size and is capable of describing the pointwise dynamics of the free boundary. Second, we aim to utilize the developed theory to deduce long time behavior of large solutions. In order to achieve these goals we consider the one-phase Muskat problem for the graphical free boundary of a fluid in porous media. Without imposing any symmetry, the general problem is posed in three space dimensions for the fluid, hence its free boundary is a two dimensional surface. In our previous work \cite{DGN}, by imposing the two dimensional symmetry of the fluid, we proved that if the initial free boundary is the graph of a periodic Lipschitz function then the equation describing its dynamics has a unique viscosity solution which satisfies the equation in $L^\infty_t L^2_x$, hence almost everywhere. This provides the first global well-posedness result without any size restriction for the Muskat problem. After that, it was shown in \cite{HN2023} that the global well-posedness result in \cite{DGN} is adequate to deduce  long time behavior of solutions, namely, all H\"older norms of the free boundary decay exponentially in time. In the present paper we prove that  the results in \cite{DGN} and \cite{HN2023} can be extended to the  full three dimensional setting. To state the results, we assume that the fluid domain at time $t$ is the lower half space below the graph of the free boundary function $f(x, t): \T^2\times \Rr_+\to \Rr$, i.e.,
\bq
\Omega_{f(t)}=\{ (x, z)\in \T^2\times \Rr: z<f(x, t)\}.
\eq
Here, we assume without loss of generality that $\T^2=\Rr^2/\Zz^2$. In porous media, the fluid motion is modeled by the Darcy law
\bq
\mu u+\na_{x, z} p=-\rho\vec{e_z},\quad  \na_{x, z}\cdot u=0\quad\text{in } \Omega_f,
\eq
where the positive constants $\mu$ and $\rho$ are respectively the dynamic viscosity and density of the fluid, and $u$ and $p$ are respectively the fluid velocity and pressure.  We have normalized the permeability constant and the acceleration due to gravity to unity.
The dynamic boundary condition states that the pressure is continuous across the free boundary. Assuming that the pressure in the dry region above $\Omega_{f(t)}$ is constant $P$, we have
\bq
p(x, f(x, t), t)=P.
\eq
The free boundary evolves according to the kinematic boundary condition
\bq
\p_t f(x, t)=u(x, f(x, t), t)\cdot N(x, t),\quad N=(-\na_x f, 1).
\eq
Letting $G(f)$ denote the Dirichlet-Neumann operator associated to the domain $\Omega_f$, the free boundary $f$ obeys the equation
\bq\label{Muskat:intro}
\p_tf=-\ka G(f)f,\quad \ka=\frac{\rho}{\mu}.
\eq
See Proposition \ref{prop:reformDN} below.

Our first main result is the global well-posedness of \eqref{Muskat:intro} for any Lipschitz initial data.
\begin{theo}\label{theo:main}
For any $f_0\in W^{1, \infty}(\T^2)$, there exist $\delta_*=\delta_*(\| f_0\|_{W^{1, \infty}})>0$ and a function
\bq\label{regf:1}
f\in C(\T^2\times [0, \infty))\cap L^\infty((0, \infty); W^{1, \infty}(\T^2))
\eq
with
\bq\label{regf:2}
 \p_t f\in L^\infty((0, \infty); L^{2+\delta_*}(\T^2))
\eq
such that the following holds.
\begin{itemize}
\item  $f$ satisfies  \eqref{Muskat:intro} in the $L^\infty_tL^{2+\delta_*}_x$ sense.
\item $f$ is the  unique viscosity solution (\`a la Crandall-Lions - see Defitnion \ref{def:viscosity}) of \eqref{Muskat:intro} with initial data $f_0$. In particular, the solution $f$ is stable in $L^\infty(\T^2)$; that is, any two solutions $f_1$ and $f_2$ as above satisfy
\bq
\| f_1(\cdot, t)-f_2(\cdot, t)\|_{L^\infty(\T^2)}\le \| f_1(\cdot, 0)-f_2(\cdot, 0)\|_{L^\infty(\T^2)}\quad\forall t>0.
\eq
Moreover, $\na f$ satisfies the maximum principle
\bq\label{maxslope:intro}
 \| \na f(t)\|_{L^\infty(\T^2)}\le \| \na f_0\|_{L^\infty(\T^2)}\quad\forall t>0.
\eq
\end{itemize}
\end{theo}
The global well-posedness result in \cite{DGN} is a special case of Theorem \ref{theo:main} when $f_0(x_1, x_2)=f_0(x_1)$. It was proven in \cite{AgrawalPatelWu} that there exists a class of initial data $f_0(x_1)$ with isolated acute corners  that lead to unique  local-in-time solutions such that the angle of the corners are preserved for a short time. Therefore, there is no instant smoothing for general Lipschitz solutions in Theorem \ref{theo:main}.  Nevertheless, we can show as in \cite{HN2023}  that the regularity \eqref{regf:1} and \eqref{regf:2} suffice to deduce the following long time behavior of the solution.
\begin{prop}\label{prop:main}
Let $f_0 \in W^{1, \infty}(\T^2)$ have mean zero, and let $f$ be the unique global solution with initial data $f_0$, as given in Theorem \ref{theo:main}. For any $\alpha \in (0, 1)$, there exist positive constants $c_\alpha$ and $C_\alpha$ depending only on $\alpha$ such that
\bq\label{decay:Holder}
\| f(t)\|_{C^\alpha(\T^2)}\le C_\alpha\|  f_0\|_{W^{1, \infty}(\T^2)}\exp\left( -\frac{c_\alpha \ka t}{1+\| \na f_0\|_{L^\infty(\T^2)}}\right)\quad\forall t>0.
\eq
\begin{proof}
We first note that $f(t)$ has mean zero for all $t\ge 0$ because the Dirichlet-Neumann operator has mean zero.  Moreover, by \eqref{maxslope:intro} we have the uniform Lipschitz bound
\bq\label{decay:Lip}
\| \na f(t)\|_{L^\infty(\T^2)}\le \| \na f_0\|_{L^\infty(\T^2)}\quad\forall t>0.
\eq
Since $\p_tf=-\ka G(f)f\in L^\infty_t L^2_x$, we have
\bq\label{dtf:L2}
\mez \frac{d}{dt}\| f(t)\|_{L^2}^2=(\p_t f(t), f(t))_{L^2, L^2}=-\ka \left(G(f(t))f(t), f(t)\right)_{L^2, L^2}.
\eq
In order to lower bound the right-hand side, we appeal to the following coercive estimate for the Dirichlet-Neumann operator, proven in \cite[Proposition 2.2]{HN2023},
\bq\label{coercive:DN}
\left(G(f(t))f(t), f(t)\right)_{L^2, L^2}\ge \frac{c}{1+\| \na f(t)\|_{L^\infty(\T^2)}}\| f(t)\|_{\dot H^\mez(\T^d)}^2,
\eq
where $c>0$ is an absolute constant. Since $f(t)$ has mean zero, the $\dot H^\mez$ norm of $f$ controls the $L^2$ norm. Therefore, combining \eqref{dtf:L2}, \eqref{coercive:DN} and \eqref{decay:Lip} yields
\[
\mez \frac{d}{dt}\| f(t)\|_{L^2}^2\le  -\frac{c\ka}{1+\| \na f_0\|_{L^\infty(\T^2)}}\| f(t)\|_{L^2(\T^d)}^2.
\]
Then  Gr\"onwall's lemma implies
\bq\label{decay:L2}
\| f(t)\|_{L^2(\T^2)}\le  \| f_0\|_{L^2(\T^2)}\exp\left( -\frac{c\ka t}{1+\| \na f_0\|_{L^\infty(\T^2)}}\right)\quad\forall t>0,
\eq
 Finally, by  interpolating \eqref{decay:L2} and \eqref{decay:Lip}, we obtain the decay \eqref{decay:Holder} of the $C^\alpha$ norm for any $\alpha \in (0, 1)$.
\end{proof}
\end{prop}
Theorem \ref{theo:main} and Proposition \ref{prop:main} provide global existence, uniqueness and long time behavior of solutions to the one-phase Muskat problem with large data. On the other hand, long time dynamics of solutions with small data has been extensively investigated for the Muksat problem in various settings. See \cite{CCGS2013, CCGRS2016, CGS, GGPS, Matioc, Cameron1, Cameron2, CordobaLazar, GancedoLazar, Nguyen-Besov, AlazardNguyen} and the references therein.  We note in particular that for the 2D two-phase Muskat problem with equal viscosities, small Lipschitz initial data are studied in \cite{ChenNguyenXu} and \cite{GGHP}. It was proven in \cite{ChenNguyenXu} that there exists a unique local-in-time solution which is smooth in positive time.  If the initial data contains a finite set of small corners, then  \cite{GGHP} provides a description on how the corners desingularize and move at the same time.

\subsection{Challenges in three dimensions and main ideas}
The key motivation for the global well-posedness Theorem \ref{theo:main} is the fact that smooth solutions of \eqref{Muskat:intro} obey the comparison principle which in turn implies the maximum principle for  $\| f(\cdot, t)\|_{W^{1, \infty}}$.

 In order to obtain pointwise (almost everywhere) solutions we employ the direct method of vanishing viscosity, which consists of the regularized problem
\bq\label{eq:f:eps:intro}
\p_tf^\eps=-\ka G(f^\eps)f^\eps+\eps \Delta f^\eps,\quad f^\eps\vert_{t=0}=f^\eps_0,
\eq
where $f_0^\eps$ is a regularization of the initial data $f_0$.   The added dissipation $\eps \Delta f^\eps$ retains the comparison principle, hence smooth solutions of \eqref{eq:f:eps:intro} obey the Lipschitz maximum principle
\bq\label{Lip:intro}
\| f^\eps(\cdot, t)\|_{W^{1, \infty}}\le \| f^\eps_0\|_{W^{1, \infty}}.
\eq
The method consists of two steps: 1) establishing global regularity for \eqref{eq:f:eps:intro} and 2) proving that $f^\eps$ converges to a solution of the original Muskat equation \eqref{Muskat:intro} in the vanishing viscosity limit $\eps\to 0$. To achieve these tasks we represent the Dirichlet-Neumann operator $G(f)g$ using layer potentials. More precisely, we prove in Proposition \ref{prop:reformDN} that if $f\in W^{1, \infty}(\T^2)$ and $g\in H^1(\T^2)$ then
\bq\label{DN:int:intro}
(G(f)g)(x)=-p.v. \int_{\T^2}\left\{\nabla_{x,z} \Gamma(x-x',f(x)-f(x'))\wedge w(x')\right\}\cdot N(x)dx',
\eq
where  $w(x)=\partial_2\theta(x)T_1(x)-\partial_{1}\theta(x)T_2(x)$ and  $\tt$ satisfies
\bq\label{def:tt:intro}
(\mez I+K[f])\tt(x):= \mez \tt(x)+p.v.\int_{\T^2} \na_{x, z}\Gamma(x- x', f(x)-f(x'))\cdot N(x') \tt(x')dx'=g(x).
\eq
Here $N=(-\na_xf(x), 1)$ and $T_j=\p_{x_j}(x,  f(x))$ are normal and tangent vectors to the surface $\{y=f(x)\}$, and $\Gamma$ is the fundamental solution of the Laplace equation on $\T^2\times \Rr$. In contrast to the two dimensional problem on $\T\times \Rr$, $\Gamma$ is implicit in three dimensions, which introduces new challenges in both steps.

Regarding the global regularity, we obtain in Lemma \ref{lemA1} the representation
\bq\label{Gamma:intro}
\Gamma(x,z)=\frac{-\eta_1(x,z)}{4\pi \sqrt{|x|^2+z^2}}+\eta_2(x,z)+\eta_3(x,z)|z|,
\eq
where $\eta_j$ are smooth functions supported in specific regions (see \eqref{eq2.25}). From \cite{ABZ3, PoyferreNguyen2017, NguyenPausader2019} we know that $G(f)f$ is well behaved when $f$ belongs to subcritical Sobolev spaces $H^s(\T^2)$ (uniformly in time), $s>2$.  Therefore, key to the global regularity of \eqref{eq:f:eps:intro} are the $H^k$  estimates with $k\in \{1, 2, 3\}$. The $H^1$ estimate follows from the  $L^2$ bound
\bq
\| G(f)g\|_{L^2}\le C(\| f\|_{W^{1, \infty}})\| g\|_{H^1}.
\eq
As for the $H^2$ and $H^3$ estimates, we need to bound the first and second order derivatives of  $G(f)f$ in $L^2(\T^2)$. This is done through the representation \eqref{DN:int:intro}-\eqref{def:tt:intro} for $G(f)f$ using the intermediate variables $w$ and $\tt$. It was proven in \cite{Vorcheta} that for Lipschitz domains, i.e. $f\in W^{1, \infty}(\T^d)$, the operator $\mez I+K[f]$ is invertible on $H^1(\T^d)$, $d\ge 2$. For the 2D problem \cite{DGN}, this $H^1$ regularity of $\tt$ was sufficient to estimate $\p_xG(f)f$ in $L^2(\T)$. However, for the 3D problem the $H^1$ regularity of $\tt$ turns out to be critical. Fortunately, a result of Dahlberg-Kenig on optimal regularity of layer potentials in Lipschitz domains implies that there exists $\eps_*=\eps_*(\| f\|_{W^{1, \infty}(\T^2)})>0$ such that $\mez I+K[f]: W^{1, p}(\T^2)\to W^{1, p}(\T^2)$ is invertible for any $p\in (1, 2+\eps_*)$. Consequently, $\tt \in L^\infty_t W^{1, 2+\delta_*}_x$ for some $\delta_*>0$ depending only on $\| f\|_{W^{1, \infty}(\T^2)}$ which is bounded by virtue of the maximum principle \eqref{Lip:intro}. This yields an extra uniform-in-$\eps$ H\"older regularity of $\tt$ which in turn provides extra cancelations in the nonlinear singular integral operators in \eqref{DN:int:intro} and \eqref{def:tt:intro}. With this and careful decompositions of the kernels, we are able to close the $H^2$ and $H^3$ estimates for $f^\eps$, thereby concluding the global regularity of \eqref{eq:f:eps:intro} together with the uniform bounds
\bq\label{unibound:intro}
\| f^\eps\|_{L^\infty_t W^{1, \infty}_x}\le C,\quad \| \tt^\eps\|_{L^\infty_t W^{1, 2+\delta_*}_x}\le C.
\eq

Next, we discuss the vanishing viscosity limit $\eps\to 0$. Fix $T>0$. From the uniform bounds \eqref{unibound:intro} and the Aubin-Lions lemma, we deduce, upon extracting subsequences, that  $f^\eps \to f$ in $C(\T^2\times [0, T])$, $f^\eps \overset{\ast}{\rightharpoonup} f$ in $L^\infty_t W^{1, \infty}_x$, and $\tt^\eps \overset{\ast}{\rightharpoonup} \tt$ in $L^\infty_t W^{1, 2+\delta_*}_x$. It can be shown that $w_n\overset{\ast}{\rightharpoonup} w=\partial_2\theta T_1-\partial_1\theta T_2$ in  $L^\infty_t L^{2+\delta_*}_x$. The main difficulty is to pass to the limit in the integrals in \eqref{DN:int:intro} and \eqref{def:tt:intro}. Although the strong convergence of $f^\eps$ implies the pointwise convergence of $\na_{x, z}\Gamma(x-x', f^\eps(x)-f^\eps(x'))$, this term is multiplied by $w^\eps$ (or $\tt^\eps)$ and the nonlinear term $N=(-\na_xf^\eps, 1)$ and hence the weak convergences of $w$ and $\tt$ are not enough to conclude. For the 2D problem, by taking advantage of the explicit formula for $\Gamma$, we showed in \cite{DGN} that
\[
(G(f)f)(x)=\frac{1}{4\pi}p.v.\p_x\int_{\T}\ln\Big(\cosh(f(x)-f(x'))-\cos(x-x')\Big)\p_x\tt(x')dx',
\]
so that upon integrating by parts against a test function, the strong convergence of $f$ and the weak convergence of $\p_x\tt$ are enough to pass to the limit. For the 3D problem, although $\Gamma$ is implicit, in Propositions  \ref{prop:fundamentalsol}  and \ref{prop:GwN} we discover a remarkable structure of the products
\[
\na_{x, z}\Gamma(x- x', f(x)-f(x'))\cdot N(x')\quad\text{and}\quad \nabla_{x,z} \Gamma(x-x',f(x)-f(x'))\wedge N(x),
\]
which eventually allows us to take the limit in the integrals. We conclude that $f$ satisfies the Muskat equation \eqref{Muskat:intro} in $L^\infty_t L^{2+\delta_*}_x$. The vanishing viscosity method also allows us to prove  that $f$ is a viscosity solutions (see Definition \ref{def:viscosity}).

In order to obtain the uniqueness of viscosity solutions, we prove that they obey the comparison principle. By the sup and inf-convolution technique, this can be deduced from the consistency of viscosity solutions, that is, if a viscosity solution $f$ is $C^{1, 1}$ at a point $(x_0, t_0)$ then $f$ classically satisfies \eqref{Muskat:intro} at the same point. Let $\psi(x, t)$ be a paraboloid tangent to $f$ from above at $X_0=(x_0, t_0)$. We construct a sequence of smooth functions $\psi_r(x,t)$ lying between $f$ and $\psi$ and converging to $f$ uniformly as $r\to 0^+$; moreover, $\psi_r$ is uniformly bounded in $W^{1, \infty}_{x, t}$. Consequently, the $\psi_r(t_0)$'s are uniformly $C^{1, 1}$ at $x_0$.  Since $\psi_r$ is a legitimate test function for the viscosity subsolution $f$, we have
\[
\p_t f(X_0)=\p_t \psi_r(X_0)\le -\ka G(\psi_r)\psi_r(X_0).
\]
It then suffices to prove that $G(\psi_r)\psi_r(X_0)\to G(f)f(X_0)$. To this end, we let $\phi_r$ (resp. $\phi$) be the harmonic extension of $\psi_r(t_0)$ to the lower half space $\{y<\psi_r(x, t_0)\}$ (resp. $\{y<f(x, t_0)\}$) and consider a common interior ball $B\subset \Rr^3$ tangent to the graphs of $f(t_0)$ and $\psi_r(t_0)$ at $x_0$. Then the Dirichlet-Neumann operators $G(\psi_r)\psi_r(X_0)$ and $G(f)f(X_0)$ are the same as the corresponding Dirichlet-Neumann operators for the ball $B$ with Dirichlet data $\phi_r\vert_{\p B}$ and $\phi\vert_{\p B}$ respectively. We prove an explicit integral formula for the Dirichlet-Neumann operator for $B$ (see Appendix \ref{appendix:sphere}), which makes sense at a point when the boundary data is $C^{1, \alpha}$ at the same point. Therefore, to obtain the convergence $G(\psi_r)\psi_r(X_0)\to G(f)f(X_0)$, we need the uniform  $C^{1, \alpha}$ regularity of the harmonic extension $\phi_r$ at $(x_0, \psi_r(x_0))$ provided  $\psi_r(t_0)$ is uniformly $C^{1, 1}$ at $x_0$. This is a {\it pointwise $C^{1, \alpha}$ elliptic regularity} statement. In \cite{DGN} we proved such a regularity result for 2D domains and in the present paper we extend it to 3D domains by using Jersion-Kenig's $W^{1, 3+\eps}$ estimates for harmonic functions in 3D Lipschitz domains. In fact, we show  that it suffices to assume that the Dirichlet data is pointwise $C^{1, \beta}$ instead of pointwise $C^{1, 1}$. This result is of independent interest.


\subsection{Organization of the paper} In Section \ref{section:reformulation}, we reformulate the one-phase Muskat problem in terms of the Drichlet-Neumann operator and utilize layer potential theory for Lipschitz domains to represent the Dirichlet-Neumann operator; necessary boundedness and contraction estimates for the Dirichlet-Neumann operator in subcritical Sobolev spaces are recalled.  In Section \ref{section:viscosity}, we prove the pointwise $C^{1, \alpha}$ elliptic regularity for 3D domains and apply it to establish the consistency and comparison principle for viscosity solutions to the Muskat problem; the proof of the consistency result makes use of an integral formula for the Dirichlet-Neumann operator for a 3D ball, whose proof is given in Appendix \ref{appendix:sphere}. Sections \ref{section:H2est} and \ref{section:H3} are devoted to Sobolev  a priori estimates and hence global regularity  for the regularized Muskat problem \eqref{eq:f:eps:intro}. After that, we prove the main Theorem \ref{theo:main} in Section \ref{section:proof}; the passage from approximate solutions to the solution relies on delicate structures of layer potentials for Lipschitz domains in $\T^2\times \Rr$, which are proven in Appendix \ref{appendix:fundamentalsol}. Finally, since our fluid domain $\Omega_f$ is unbounded, for the sake of completeness we prove the invertibility of boundary layer potentials in Appendix \ref{appendix:inverses}. 
 \section{Reformulations}\label{section:reformulation}
\begin{nota}
For $f:\T^d\to \Rr$ we denote
\begin{align}
&\Omega_f=\{(x, z): x\in \T^d,~z<f(x)\},\quad \Sigma_f=\{(x, f(x)): x\in \T^d\},\\
& N(x)=(-\na_xf(x), 1),\quad n(x)=\frac{N(x)}{|N(x)|}.
\end{align}
\end{nota}
\begin{defi}
For $f:\T^d\to \Rr$, the Dirichlet-Neumann operator $G(f)$ is defined by
\bq\label{def:Gh}
\big(G(f)g\big)(x)=\p_N\phi(x, f(x)):=\lim_{h\to 0^-}\frac{1}{h}\big[\phi\big((x, f(x))+hN(x)\big)-\phi(x, f(x))\big],
\eq
where $\phi(x, y)$ solves the elliptic problem
\bq\label{elliptic:G}
\begin{cases}
\Delta_{x, z}\phi=0\quad\text{in}~\Omega_f,\\
\phi(x, f(x))=g(x),\quad \nabla_{x, z}\phi\in L^2(\Omega_f).
\end{cases}
\eq
\end{defi}
When $f$ and $g$ are time-dependent, we  write
\bq\label{nota:DNt}
\big(G(f)g\big)(x, t)\equiv \big(G\big(f(t)\big)g(t)\big)(x).
\eq
The one-phase Muskat problem has a compact reformulation in terms of the Dirichlet-Neumann operator.
\begin{prop}[\protect{\cite[Proposition 2.1]{NguyenPausader2019}}]\label{prop:reformDN}
If the fluid domain is given by $\Omega_f$ for some $f(x, t):\Rr^2\times (0, T)\to \Rr$, then $f$ satisfies
\bq\label{Muskat:DN}
\p_tf=-\ka G(f)f\quad\text{on } (0, T),
\eq
where $\ka={\rho}/{\mu}$.
\end{prop}
In order to establish fine properties of the Drichlet-Neumann operator for Lipschitz domains, we shall need a reformulation via layer potential theory.  Assume that $f\in W^{1, \infty}(\T^2)$ and $g\in H^1(\T^2)$. By potential theory for Lipschitz domains \cite{Vorcheta}, the solution $\phi$ of \eqref{elliptic:G} is given by
\bq\label{def:potental}
\phi(x, z)=\cK \tt,\quad \tt=(\mez I+K)^{-1}g,
\eq
where $\cK=\cK[f]$ is the double layer potential
\bq\label{def:cK}
\cK h(x, z)=\int_{\T^2} \na_{x, z}\Gamma(x- x', z-f(x'))\cdot N(x') h(x')dx',\quad z\ne f(x),
\eq
and $K=K[f]$ is the boundary layer potential
\bq\label{def:K}
K h(x)=p.v.\int_{\T^2} \na_{x, z}\Gamma(x- x', f(x)-f(x'))\cdot N(x') h(x')dx'.
\eq
We have denoted by $\Gamma$  the fundamental solution of the Laplace equation on $\T^2\times \Rr$, which is not explicit  unlike in $\T\times \Rr$ (see \cite{DGN}). Nevertheless, $\Gamma$  can be represented as in \eqref{eq2.25} in Appendix \ref{appendix:fundamentalsol}. Now \eqref {def:potental} means that $\tt$ can be determined by solving the equation
\bq\label{def:tt}
\mez \tt(x)+p.v.\int_{\T^2} \na_{x, z}\Gamma(x- x', f(x)-f(x'))\cdot N(x') \tt(x')dx'=g(x).
\eq
 We note that in \eqref{def:potental}, $\phi$ is defined on $(\T^2_x\times \Rr)\setminus \Sigma$, while
\bq\label{limit:tt}
\lim_{(x, z)\xrightarrow{\text{i}} (x_0, f(x_0))}\phi(x, z)=(\mez I+K)\tt(x_0),\quad \lim_{(x, z)\xrightarrow{\text{e}} (x_0, f(x_0))}\phi(x, z)=-(\mez I-K)\tt(x_0)
\eq
for almost every  $x_0\in \T^2$ (see Theorem 1.10 in \cite{Vorcheta}). The letter $\text{i}$ (resp. $\text{e}$) in the above limit means that the limit is taken in interior (resp. exterior) cones of $\Omega_f$ with vertex $(x_0, f(x_0))$.

For completeness, we also introduce the single layer potential  
\bq\label{singlelayer}
\cS h(x, z)=\int_{\T^2} \Gamma(x- x', z-f(x')) h(x')dx',\quad (x,z)\in \T^2\times \Rr,
\eq
and the boundary single layer potential 
\bq\label{singlelayer2}
S h(x)=\int_{\T^2} \Gamma(x- x', f(x)-f(x')) h(x')dx',\quad x\in \T^2.
\eq
Now, we set $v(x, z)=\na_{x, z}\phi(x, z)$ for $z\ne f(x)$ and
\bq\label{def:vjump}
w(x)=[v(x, f(x))-v^c(x, f(x)]\wedge N(x),
\eq
where we have used the notation
\bq\label{nota:vc}
v(x, f(x))=\lim_{\eps \to 0^+} v((x, f(x))-\eps N(x)),\quad v^c(x, f(x))=\lim_{\eps \to 0^+} v((x, f(x))+\eps N(x)).
\eq
Since $\p_N\phi=\p_N\cK\tt$ is continuous across $\Sigma$ (see (7.43) in \cite{MT}), we deduce that $\na \cdot v=0$. Next, the vorticity  $\om=\na\wedge v$ can be written in terms of $w$ as
\bq\label{om:w}
\om(x, z)=w(x)\delta(z-f(x)).
\eq
Since $\na \cdot v=0$,  the identity $\na\wedge (\na \wedge v)=\na (\na \cdot v)-\Delta v$ implies $-\Delta v=\na \wedge \om$. Then invoking \eqref{om:w} yields
\bq
\begin{aligned}
v(x, z)&= -\na\wedge\Delta^{-1}\om\\
&=-\na \wedge\int_{\T^2\times \Rr} \Gamma(x-x', z-z')\om(x', z')dx'dz'\\
&=-\na \wedge\int_{\T^2} \Gamma(x-x', z-f(x'))w(x')dx'\\
&=-\int_{\T^2} \na_{x, z}\Gamma(x-x', z-f(x'))\wedge w(x')dx'.
\end{aligned}
\eq
Taking the limit $(x, z)\xrightarrow{\text{i}} (x, f(x))$  we obtain the following.
\begin{prop} For almost every $x\in \T^2$, we have
\bq\label{uboundary}
\begin{aligned}
v(x, f(x))&=\lim_{(x, z)\xrightarrow{\text{i}} (x, f(x))}v(x, z)\\
&=-p.v.\int_{\T^2}\nabla_{x,z} \Gamma(x-x',f(x)-f(x'))\wedge w(x')dx'+\frac12\frac{N(x)\wedge w(x)}{|N(x)|^2}.
\end{aligned}
\eq
\end{prop}
\begin{proof}
We decompose $\nabla_{x,z} \Gamma(x-x', z-f(x'))$ into the normal and tangential parts:
 \[
 \nabla_{x,z} \Gamma(x-x', z-f(x'))=[\nabla_{x,z} \Gamma(x-x', z-f(x'))\cdot N(x')]\frac{N(x')}{|N(x')|^2}+A(x, z, x').
 \]
 We note that the normal part leads to the double layer potential
 \[
-p.v.\int_{\T^2}[\nabla_{x,z} \Gamma(x-x', z-f(x'))\cdot N(x')]\frac{N(x')\wedge w(x')}{|N(x')|^2}dx',
 \]
 which converges to $(\mez I+K)\frac{N\wedge w}{|N|^2}$  when $(x, z)\xrightarrow{\text{i}} (x, f(x))$ for almost every $x$. As for the tangential parts, applying Lemma 1.5 in \cite{Vorcheta}, we find
 \[
 \lim_{(x, z)\xrightarrow{\text{i}} (x, f(x))}-\int_{\T^2}A(x, z, x')\wedge w(x')dx'= -p.v.\int_{\T^2}A(x, f(x), x')\wedge w(x')dx'
 \]
 for almost every $x$. Summing the above two limits yields \eqref{uboundary}.
\end{proof}
Using \eqref{uboundary}, we rewrite  the Dirichlet-Neumann operator  as
\bq
G(f)g=v(x, f(x))\cdot N(x)=-p.v. \int_{\T^2}\left\{\nabla_{x,z} \Gamma(x-x',f(x)-f(x'))\wedge w(x')\right\}\cdot N(x)dx'.
\eq
Next we express $w$ in terms of $\tt$. We let
\bq
T_1(x)=(1,0,\partial_1f(x)),\quad T_2(x)= (0,1,\partial_2f(x))
\eq
be tangent vectors to $\{y=f(x)\}$. It follows from \eqref{limit:tt} that $\theta(x)=\phi(x,f(x))-\phi^c(x,f(x))$ (see the notation \eqref{nota:vc}), hence
\bq\label{p1tt}
\begin{aligned}
\partial_{1}\theta(x)&=(\nabla_{x,z}\phi(x,f(x))-\nabla_{x,z}\phi^c(x,f(x)))\cdot T_1(x)=(v(x, f(x))-v^c(x,f(x)))\cdot T_1(x)
\end{aligned}
\eq
and
\bq\label{p2tt}
\partial_2\theta(x)
=(v(x, f(x))-v^c(x, f(x))\cdot T_2(x).
\eq
Now in \eqref{def:vjump} we write $N=T_1\wedge T_2$  and use the vector identity $a\wedge(b\wedge c)=b(a\cdot c)-c(a\cdot b)$. Then invoking \eqref{p1tt} and \eqref{p2tt}, we can express $w$ in terms of $\tt$ as
\bq\label{weqftheta:0}
w(x)=\partial_2\theta(x)T_1(x)-
\partial_{1}\theta(x)T_2(x).
\eq
Taking derivatives of equation $\phi(x, f(x))=g(x)$ with respect to $x_1$ and $x_2$, we deduce
\[
\na\phi(x, f(x))\cdot T_1(x)=-\p_1g(x),\quad \na\phi(x, f(x))\cdot T_2(x)=-\p_2g(x).
\]
Then using \eqref{uboundary} for $v(x, f(x))=\na \phi(x, f(x))$ and using \eqref{weqftheta:0} for $w$, we find
\bq\label{d1tt}
\frac12\partial_1\theta(x)-p.v.\int_{\T^2}\nabla_{x,z} \Gamma(x-x',f(x)-f(x'))\wedge w(x')dx'\cdot T_1(x)=-\partial_1g(x)
\eq
and
\bq\label{d2tt}
\frac12\partial_2\theta(x)-p.v.\int_{\T^2}\nabla_{x,z} \Gamma(x-x',f(x)-f(x'))\wedge w(x')dx'\cdot T_2(x)=-\partial_2g(x).
\eq
We have proved the following.
\begin{prop}\label{prop:reformDN_b}
If $f\in W^{1, \infty}(\T^2)$ and $g\in H^1(\T^2)$, then for a.e. $x\in \T^2$, $(G(f)g)(x)$ is given by
\bq\label{DN:int}
(G(f)g)(x)=-p.v. \int_{\T^2}\left\{\nabla_{x,z} \Gamma(x-x',f(x)-f(x'))\wedge w(x')\right\}\cdot N(x)dx',
\eq
where
\bq\label{weqftheta}
w(x)=\partial_2\theta(x)T_1(x)-
\partial_{1}\theta(x)T_2(x)\quad\text{and}\quad \tt=(\mez I+K)^{-1}g.
\eq
Moreover,   $w$ and $\tt$  satisfy the relations \eqref{d1tt} and \eqref{d2tt}.
\end{prop}
A key technical element in various places in our proof of Theorem \ref{theo:main} is the fact that for $f\in W^{1, \infty}(\T^2)$, $\tt=(\mez I+K[f])^{-1}f$ belongs to $W^{1, p}$ for some $p>2$. This is a consequence of the optimal $W^{1, 2+\eps}$  regularity for layer potentials in bounded Lipschitz domains, a celebrated result in \cite{DK}. In our infinite-depth  setting,  although the boundary of  $\Omega_f$ is bounded Lipschitz, $\Omega_f$ is unbounded. Moreover, we will also need a quantitative bound for the inverse $(\mez I+K[f])^{-1}$. Therefore, we will present a proof of the following result in Theorem \ref{theo:Kinv}, Appendix \ref{appendix:inverses}. 
\begin{prop}\label{prop:W1p}
Let $f\in W^{1, \infty}(\T^2)$ and  $m$ be such that
\bq\label{Lipschitzf}
|f(x)-f(y)|\le m|x-y|\quad\forall x,~y\in \T^2.
\eq
There exists $\eps_*>0$ depending only on $m$ such that for any $p\in [2, 2+\eps_*)$,
\bq\label{invert:tt}
(\mez I+K): W^{1, p}(\T^2)\to W^{1, p}(\T^2)
\eq
is a bounded operator with a bounded inverse, where the operator norms depend only on $m$.
\end{prop}
Next, we deduce the boundedness of $G(f)$ from $W^{1, p}(\T^2)$ to $L^p(\T^2)$ for the range of $p$ given in Proposition \ref{prop:W1p}.
\begin{prop}\label{prop:DNLp} Let $f\in W^{1, \infty}(\T^2)$ satisfy \eqref{Lipschitzf} and let $\eps_*=\eps_*(m)$ be as in Proposition \ref{prop:W1p}. For any $p\in (1, 2+\eps_*)$ there exists $C:\Rr_+\to \Rr_+$  such that for all $g\in W^{1, p}(\T^2)$ we have
\bq\label{DN:Lp}
\| G(f)g\|_{L^p(\T^2)}\le C(\| f\|_{W^{1, \infty}(\T^2)})\| g\|_{W^{1, p}(\T^2)}.
\eq
\end{prop}
\begin{proof}
Using the identities $N=T_1\wedge T_2$, $(a\wedge b)\cdot c=(c\wedge a)\cdot b$ and $(a\wedge b)\wedge c=b(a\cdot c)-a(b\cdot c)$, we deduce from \eqref{DN:int} that
\bq
\begin{aligned}
(G(f)g)(x)&=-p.v. \int_{\T^2}\left\{\nabla_{x,z} \Gamma(x-x',f(x)-f(x'))\wedge w(x')\right\}\cdot N(x)dx'\\
&=-p.v. \int_{\T^2}\left\{N(x)\wedge\nabla_{x,z} \Gamma(x-x',f(x)-f(x'))\right\}\cdot w(x')dx'\\
&=-p.v. \int_{\T^2}T_2(x)\left\{T_1(x)\cdot\nabla_{x,z} \Gamma(x-x',f(x)-f(x'))\right\}\cdot w(x')dx'\\
&\quad +p.v. \int_{\T^2}T_1(x)\left\{T_2(x)\cdot\nabla_{x,z} \Gamma(x-x',f(x)-f(x'))\right\}\cdot w(x')dx',
\end{aligned}
\eq
where $w(x)=\partial_2\theta(x)T_1(x)-\partial_{1}\theta(x)T_2(x)$, $\tt=(\mez I+K)^{-1}g$. Then applying \cite[Theorem 1.6]{Vorcheta}, we obtain
\bq
(G(f)g)(x)=T_2(x)\cdot \p_{T_1}\cS w(x,f(x))-T_1(x)\cdot \p_{T_2}\cS w(x,f(x)),
\eq
where $\cS$ is the single layer potential \eqref{singlelayer}. The same theorem asserts that the $\p_{T_j}\cS$'s are bounded on $L^p(\T^2)$, $p\in (1, \infty)$. But $\tt\in W^{1, p}(\T^2)$ for $p\in (1, 2+\eps_*)$ by virtue of Proposition \ref{prop:W1p}, hence $G(f)g\in L^p(\T^2)$ and \eqref{DN:Lp} holds.
\end{proof}
Finally, we recall the following   continuity and contraction estimates for the Dirichlet-Neumann operator in subcritical Sobolev spaces $H^s(\T^d)$, $s>1+\frac{d}{2}$. They will be used to propagate high Sobolev regularity for the regularized Muskat problem \eqref{Muskate}.
\begin{prop}[\protect{\cite[Proposition 2.13]{PoyferreNguyen2017}}]\label{prop:estDN}
Let $s_0>1+\frac{d}{2}$ and  $\sigma\ge \mez$. Then there exists a nondecreasing function $\cF:\Rr^+\to \Rr^+$  such that
\bq\label{est:DN}
\| G(f)g\|_{H^{\sigma-1}(\T^d)}\le \cF(\| f\|_{H^{s_0}(\T^d)})\left(\| g\|_{H^\sigma(\T^d)}+\| f\|_{H^\sigma(\T^d)}\| g\|_{H^{s_0}(\T^d)}\right)
\eq	
for all $f,\, g\in H^{\max\{s_0, \sigma\}}(\T^d)$.
\end{prop}
\begin{prop}[\protect{\cite[Corollary 3.25]{NguyenPausader2019}}]\label{theo:contraDN}
Let $s_0>1+\frac{d}{2}$ and $\sigma\in [\mez, s_0]$. Then there exists a nondecreasing function $\cF:\Rr^+\to \Rr^+$  such that
\[
\| G(f_1)g-G(f_2)g\|_{H^{\sigma-1}(\T^d)}\le \cF(\| (f_1, f_2)\|_{H^{s_0}(\T^d)})\| f_1-f_2\|_{H^\sigma(\T^d)}\| g\|_{H^{s_0}(\T^d)}
\]
for all $f_j,\, g\in H^{s_0}(\T^d)$.
\end{prop}
\section{Pointwise \texorpdfstring{$C^{1, \alpha}$}{} elliptic estimates and viscosity solutions}\label{section:viscosity}
Throughout this section, we denote $x=(x_1,x_2,x_3)\in \Rr^3$.
\subsection{Pointwise \texorpdfstring{$C^{1, \alpha}$}{} elliptic estimates}
Let $ U$ be a Lipschitz domain in $ \Rr^3$. For any $x_0\in  \Rr^3$ and $r>0$, we denote $ U_r(x_0)=B_r(x_0)\cap  U$ and $ U_r= U_r(0)$. We also define the half ball as
$$
B^+_r(x_0)=\{x\in B_r(x_0):x_3>x_{03}\}.
$$
Without loss of generality, we assume that $0\in \partial U$. Suppose that there exist some $r_0>0$ and a coordinate system, such that $\partial U\cap B_{2r_0}$ can be represented by a Lipschitz graph with Lipschitz constant $L>0$.

\begin{assu}
                        \label{assump1c}
There exist constants $C_0,r_0>0$, $\beta\in (0,1)$, and a function $\psi$ in $(-r_0,r_0)^2$ such that in a coordinate system
\begin{equation}\label{interiorball}
\psi(0)=\psi'(0)=0,\quad  U_{r_0}=\{x\in B_{r_0}:\,x_3>\psi(x_1,x_2)\},\quad |\psi(x_1,x_2)|\le C_0(x_1^2+x_2^2)^{(1+\beta)/2}.
\end{equation}
\end{assu}
In other words, $\partial U$ is $C^{1,\beta}$ at $0\in \partial U$.

\begin{theo}\label{theo:pointwiseelliptic}
Let $u$ be a harmonic function in $U$, which vanishes on $\p U$.  Under the conditions above, $u$ is $C^{1,\alpha}$ at $0$, i.e., for any $x\in\Omega$ such that $|x|<r_0$, it holds that
\begin{equation}
                            \label{eq12.24}
|u(x)-x\cdot \nabla u(0)|\le C|x|^{1+\alpha}r_0^{-2-\alpha}\|u\|_{L_2(\Omega_{2r_0})},
\end{equation}
where $C>0$ is a constant depending only on $C_0r_0^\beta$ and $L$, and $\alpha=\alpha(L,\beta)\in (0,1)$ is a small constant.
\end{theo}

Before we prove the theorem, we make some reductions. First by scaling, we may assume that $r_0=1$, so that the new $C_0$ becomes $C_0r_0^\beta$, and $L$ remain the same. Moreover, we may also assume that $\|u\|_{L^2( U_{2})}=1$ upon dividing $u$ by a constant.

Because $\partial U$ satisfies \eqref{interiorball}, by using a barrier argument and the boundary $C^{1,\beta}$ regularity for harmonic functions in $C^{1,\beta}$ domains, it is easily seen that that $u$ is Lipschitz at $0$ and
\begin{equation}
                                \label{eq12.24c}
|u(x)|\le C|x| \quad\text{in}\,\, U_{2}.
\end{equation}
Now by the using Jerison-Kenig's $W^{1, 3+\varepsilon}$ estimate  (see \cite{JK95}) for harmonic functions in Lipschitz domains and \eqref{eq12.24c}, there exists $p_0=p_0(L)>3$ such that
\begin{equation}
                            \label{eq11.243}
\|\nabla u\|_{L^{p_0}( U_r)}\le Cr^{3/p_0-5/2}\|u\|_{L^2( U_{2r})}
\le Cr^{3/p_0}.
\end{equation}

Next we take a smooth domain $E$ such that $B_{2/3}^+\subset E\subset B_{3/4}^+$. For any $x_0\in  \Rr^3$ and $r>0$, denote
$$
E_r(x_0)=\{x\in  \Rr^2:r^{-1}(x-x_0)\in E\},\quad \Gamma_r(x_0)=\{x\in \partial E_r(x_0): x_3=x_{03}\}.
$$
Clearly, by taking $r$ sufficiently small, from \eqref{interiorball} we have $E_r(0,C_0r^{1+\beta})\subset  U_r$.

Let $\eta=\eta(s)$ be a smooth function on $ \Rr$ such that $\eta(s)=0$ in $(-\infty,1)$ and $\eta(s)=1$ in $(2,\infty)$. Denote $\eta_r(s)=\eta(s/(C_0r^{1+\beta}))$. Then $u(x)\eta_r(x_3)$ satisfies
\begin{equation*}
\Delta (u(x)\eta_r(x_3))=\partial_3(u\eta'_r)+\partial_3u\eta'_r\quad \text{in}\,\,E_r(0,C_0r^{1+\beta})
\end{equation*}
and $u(x)\eta_r(x_3)=0$ on $\Gamma_r(0,C_0r^{1+\beta})$. Observe that the right-hand side is supported in a narrow strip $\{x\in  U_r:C_0r^{1+\beta}<x_3<2C_0r^{1+\beta}\}$.

We decompose $u\eta_r$ 
as follows. Let $w=w_r$ be a weak solution to
\begin{equation*}
\Delta w=\partial_3(u\eta'_r)+\partial_3u\eta'_r\quad \text{in}\,\,E_r(0,C_0r^{1+\beta})
\end{equation*}
with the zero Dirichlet boundary condition on $\partial E_r(0,C_0r^{1+\beta})$.
Then $v=v_r=u\eta_r-w$ satisfies
\begin{equation*}
\Delta v=0\quad \text{in}\,\,E_r(0,C_0r^{1+\beta})
\end{equation*}
and $v=0$ on $\Gamma_r(0,C_0r^{1+\beta})$.

{\em Estimates of $w$.} Since $E_r(0,C_0r^{1+\beta})$ is smooth, by the classical $W^{1, p}$ estimate, we know that for any $p<\infty$, $w\in W^1_p(E_r(0,C_0r^{1+\beta}))$. Notice that $\eta_r'(x_3)\le C|x_3-\psi(x_1,x_2)|^{-1}$. By using Hardy's inequality as well as a duality argument (see the proof of \cite[Theorem 3.5]{DX15}), we have
\begin{align}
                    \label{eq7.52c}
\|\nabla w\|_{L^p(E_r(0,C_0r^{1+\beta}))}
&\le C\|\nabla u\|_{L^p( U_r\cap\{x_3<2C_0r^{1+\beta}\})}.
\end{align}
We fix $p=(3+p_0)/2\in (3,p_0)$, where $p_0$ is the exponent in \eqref{eq11.243}, and let $q>1$ be such that $1/q=1/p-1/p_0$. Using \eqref{eq7.52c}, H\"older's inequality, and \eqref{eq11.243}, we have
\begin{align}
                    \label{eq11.333}
\|\nabla w\|_{L^p(E_r(0,C_0r^{1+\beta}))}
&\le C\|\nabla u\|_{L^{p_0}( U_r\cap\{x_3<2C_0r^{1+\beta}\})}r^{(3+\beta)/q}\le Cr^{3/p_0+(3+\beta)/q}.
\end{align}
By the zero boundary condition, the Morrey embedding, and \eqref{eq11.333}, we obtain
\begin{align}
                    \label{eq7.553}
\|w\|_{L^\infty(E_r(0,C_0r^{1+\beta}))}
&\le Cr^{1-3/p}[w]_{C^{1-3/p}(E_r(0,C_0r^{1+\beta}))}\notag\\
&\le Cr^{1-3/p}\|\nabla w\|_{L^p(E_r(0,C_0r^{1+\beta}))}\notag\\
&\le Cr^{1-3/p+3/p_0+(3+\beta)/q}=Cr^{1+\beta/q}.
\end{align}

{\em Estimates of $v$.} Since $B^+_{r/2}(0,C_0r^{1+\beta})\subset E_r(0,C_0r^{1+\beta})$, we know that $v$ is harmonic in $B^+_{r/2}(0,C_0r^{1+\beta})$ and vanishes on the flat boundary. It then follows from the boundary estimate for harmonic functions that
\begin{align*}
\|\nabla  v\|_{L^\infty(B^+_{r/4}(0,C_0r^{1+\beta}))}\le Cr^{-1}\|v\|_{L^\infty(B^+_{r/2}(0,C_0r^{1+\beta}))},
\end{align*}
which together with \eqref{eq7.553} and the Lipschitz regularity of $u$ at $0$ implies that
\begin{equation}
\label{eq11.51c}
\|\nabla v\|_{L^\infty(B^+_{r/4}(0,C_0r^{1+\beta}))}\le C.
\end{equation}
Furthermore, by \cite[Lemma 2.5]{DEK18}, for any linear function $\ell$ of $x_3$, we have
\begin{equation*}
\|\nabla^2 v\|_{L^\infty(B^+_{r/4}(0,C_0r^{1+\beta}))}\le Cr^{-2}\|v-\ell\|_{L^\infty(B^+_{r/2}(0,C_0r^{1+\beta}))}.
\end{equation*}
Because $v(0,C_0r^{1+\beta})=\partial_1v(0,C_0r^{1+\beta})=0$, by the mean value theorem , for any $\kappa\in (0,1/4)$,
\begin{align}
                \label{eq12.253}
&\|v-(x_3-C_0r^{1+\beta})\partial_3v(0,C_0r^{1+\beta})\|_{L^\infty(B^+_{\kappa r}(0,C_0r^{1+\beta}))}\notag\\
&\le C\kappa^{2}\|v-\ell\|_{L^\infty(B^+_{r/2}(0,C_0r^{1+\beta}))},
\end{align}
where $\ell$ is any linear function of $x_3$.

{\em Estimates of $u$.}
Since $u\eta_r=w+v$ in $E_r(0,C_0r^{1+\beta})$,
we combine \eqref{eq7.553} and \eqref{eq12.253} to get
\begin{align}
                        \label{eq11.463}
&\|u\eta_r-(x_3-C_0r^{1+\beta})\partial_2v(0,C_0r^{1+\beta})\|_{L^\infty(B^+_{\kappa r}(0,C_0r^{1+\beta}))}\notag\\
&\le C\kappa^{2}\inf_{a,b\in  \Rr}\|u\eta_r-(a+bx_3)\|_{L^\infty(B^+_{r/2}(0,C_0r^{1+\beta}))}
+Cr^{1+\beta/q}.
\end{align}

Let $h$ be the solution to $\Delta h=0$ in $\{x\in B_{1}:\,x_3>-C_0(x_1^2+x_2^2)^{(1+\beta)/2}\}$ with the boundary conditions $h=0$ on $\{x\in B_{1}:\,x_3=-C_0(x_1^2+x_2^2)^{(1+\beta)/2}\}$ and $h=2C$ on $\{x\in \partial B_{1}:\,x_3>-C_0(x_1^2+x_2^2)^{(1+\beta)/2}\})$, where $C$ is the constant in \eqref{eq12.24c}. Then by the boundary and interior estimates for harmonic functions, we have
\begin{equation}
                                \label{eq11.48c}
\|\nabla h\|_{L^\infty(\{x\in B_{1/2}:\,x_3>-C_0(x_1^2+x_2^2)^{(1+\beta)/2}\})}\le C(d, \beta).
\end{equation}
By the maximum principle, we know that $h\ge 0$ on $B_1\cap \partial U$ and $u\le h$ in $ U_1$. Similarly, we have $u\ge -h$ in $ U_1$. Thus,
\begin{equation}
                                \label{eq12.19c}
|u|\le h\quad \text{in}\,\, U_1.
\end{equation}
Therefore, from \eqref{eq11.48c} and \eqref{eq12.19c} we have
\begin{align*}
\|u(1-\eta_r)\|_{L^\infty( U_r)}\le \sup_{ U_r\cap \{x_3<2C_0r^{1+\beta}\}}|u(x)|
\le \sup_{ U_r\cap \{x_3<2C_0r^{1+\beta}\}}h
\le C(d,r_0)r^{1+\beta}.
\end{align*}
The above inequality, \eqref{eq11.51c}, and \eqref{eq11.463} implies that for any $\kappa\in (0,1/8)$,
$$
\inf_{a,b\in  \Rr}\|u-(a+bx_3)\|_{L^\infty( U_{\kappa r})}
\le C\kappa^{2}\inf_{a,b\in  \Rr}\|u-(a+bx_3)\|_{L^\infty( U_{r})}
+Cr^{1+\alpha},
$$
where $\alpha=\beta/q$.
By a standard iteration argument, we then deduce
$$
\inf_{a,b\in  \Rr}\|u-(a+bx_3)\|_{L^\infty( U_{r})}\le Cr^{1+\alpha},
$$
i.e., for any $r\in (0,1)$, there exist constants $a_r$ and $b_r$ such that
\begin{equation}
                                \label{eq8.10c}
\|u-(a_r+b_rx_3)\|_{L^\infty( U_{r})}\le Cr^{1+\alpha}.
\end{equation}
\begin{lemm}
Under the condition Theorem \ref{theo:pointwiseelliptic}, we have
\begin{equation}
                        \label{eq8.11c}
|u(0)-a_r|\le Cr^{1+\alpha}
\end{equation}
and for any $r/2\le s\le r<1$,
\begin{equation}
                        \label{eq8.20c}
|b_{s}-b_r|\le Cr^\alpha.
\end{equation}
\end{lemm}
\begin{proof}
The inequality \eqref{eq8.11c} follows immediately from \eqref{eq8.10c}. Next, by the triangle inequality, for any $x\in  U_{s}$,
\begin{align*}
&|(a_{s}+b_{s}x_3)-(a_{r}+b_{r}x_3)|\le |u(x)-(a_{s}+b_{s}x_3)|+|u(x)-(a_{r}+b_{r}x_3)|\le Cr^{1+\alpha}.
\end{align*}
This together with \eqref{eq8.11c} implies \eqref{eq8.20c}. The lemma is proved.
\end{proof}
\begin{proof}[Proof of Theorem \ref{theo:pointwiseelliptic}]
From \eqref{eq8.20c} we see that $b_r\to b^*$ as $r\to 0$ for some constant $b^*$ and
\begin{equation}
                        \label{eq8.21c}
|b^*-b_r|\le Cr^\alpha.
\end{equation}
Combining \eqref{eq8.10c}, \eqref{eq8.11c}, and \eqref{eq8.21c}, we reach
$$
|u(x)-b^* x_3|\le C|x|^{1+\alpha},
$$
which gives \eqref{eq12.24} with $\partial_1 u(0)=\partial_2 u(0)=0$ and $\partial_3u(0)=b^*$.
The theorem is proved.
\end{proof}
We shall use the following consequence of Theorem \ref{theo:pointwiseelliptic}.
\begin{coro}\label{coro:elliptic}
Assume that $f\in W^{1, \infty}(\T^2)$ is $C^{1, \beta}$ at $x_0\in \T^2$, where $\beta\in (0,1)$. Then the harmonic extension $\phi$ of $f$ to $\Omega_f$ is $C^{1, \alpha}$ at $(x_0, f(x_0))$ in the following quantitative sense. Let $\gamma\in (0, 1)$ and $M_0>0$ satisfy
\[
|f(x)- f(x_0)-\na f(x_0)\cdot (x-x_0)|\le M_0|x-x_0|^{1+\beta}\quad \forall |x-x_0|<\gamma.
\]
Then there exists $\alpha\in (0, 1)$  depending only on $\beta$ and $\| f\|_{W^{1, \infty}(\T^2)}$, and there exists $M>0$ depending only on $\| f\|_{W^{1, \infty}(\T^2)}$ and $M_0\gamma$, such that
\[
|\phi(x, z)-f(x_0)-(x-x_0, z-f(x_0))\cdot \na\phi(x_0, f(x_0))|\le M\gamma^{-1-\alpha}\big|(x-x_0)^2+(z-f(x_0))^2\big|^{\frac{1+\alpha}{2}}
\]
for all $(x, z)\in \Omega_f\cap B_\gamma((x_0, f(x_0)))$.
\end{coro}
\begin{proof}
This follows by writing $\phi(x, z)=p(x, z)+z$, so that $p$ is harmonic and vanishes on $\p\Omega_f$ and hence Theorem \ref{theo:pointwiseelliptic} applies to $p$.
\end{proof}

\subsection{Viscosity solutions}
We first recall the following comparison principle for $G(f)f$.
\begin{prop}\cite[Proposition 2.15 and Remark 2.16]{DGN}\label{prop:comparisonDN}
Let $f_1,\, f_2\in W^{1, \infty}(\T^2)$ be $C^{1, 1}$ at $x_0$. If $f_1(x)\le f_2(x)$ for all $x\in \T^2$ and $f_1(x_0)=f_2(x_0)$, then  $G(f_j)f_j$ are classically well-defined at $x_0$ and
\bq\label{compare:DN}
\big(G(f_1)f_1\big)(x_0)\ge \big(G(f_2)f_2\big)(x_0).
\eq
\end{prop}
In view of Proposition \ref{prop:comparisonDN}, it is natural to study  viscosity solutions to \eqref{Muskat:DN}.
\begin{defi}[{\bf Viscosity solutions}]\label{def:viscosity}
 A function $f: \T^2\times [0, T]$ is called a viscosity subsolution (resp. supersolution) of \eqref{Muskat:DN} on $(0, T)$  provided that

\noindent (i) $f$   is upper semicontinuous (resp. lower semicontinuous) on $\T^2\times [0, T]$, and \\
(ii) for every $\psi:\T^2\times (0, T)\to \Rr$ with
$$
\p_t\psi\in  C(\T^2\times (0, T))\quad \text{and}\quad
\psi\in  C((0, T); C^{1, 1}(\T^2)),
$$
if $f-\psi$ attains a global maximum (resp. minimum) over $ \T\times [t_0-r, t_0]$ at $(x_0, t_0)\in \T^2\times (0, T)$ for some $r>0$, then
\bq\label{ineq:defviscosity}
\p_t \psi(x_0, t_0)\le -\ka \big(G(\psi)\psi\big)(x_0, t_0)\quad(\text{resp. }\ge).
\eq
A  viscosity solution is both a viscosity subsolution and viscosity supersolution.
\end{defi}
By virtue of Proposition \ref{prop:comparisonDN}, it is readily seen that smooth solutions of \eqref{Muskat:DN} are viscosity solutions. Next, we prove the consistency of viscosity solutions, a fact that will be used later to deduce the comparison principle and hence uniqueness for viscosity solutions.
\begin{prop}[{\bf Consistency}]\label{prop:roughtest}
Let $f$ be a viscosity subsolution of \eqref{Muskat:DN} on $(0, T)$. Assume that $f\in W^{1, \infty}(\T^2\times (0, T))$ and $f$ is $C^{1, 1}$ at $(x_0, t_0)\in \T^2\times (0, T)$.  Then $G(f)f$ is classically well-defined   at $(x_0, t_0)$ and
\bq\label{sub:roughtest}
\p_t f(x_0, t_0)\le -\ka \big(G(f)f\big)(x_0, t_0).
\eq
The corresponding statement  for viscosity supersolutions holds true.
\end{prop}
\begin{proof}
This result was proven in \cite{DGN} for the 2D Muskat problem. The same strategy carries over to higher dimensions provided that we have 1) the pointwise $C^{1, \alpha}$ elliptic regularity in Theorem \ref{theo:pointwiseelliptic}, and 2) the integral representation in Proposition \ref{prop:DNsphere} for the Dirichlet-Neumann operator for the sphere. For the sake of completeness, we will outline the key steps. Denote $X=(x, t)$ and $X_0=(x_0, t_0)$. Since $f$ is $C^{1, 1}$ at $X_0$, there exists a parabola
\[
\overline{\psi}(X)=f(X_0)+\na_Xf(X_0)\cdot (X-X_0)+\frac{C}{2}|X-X_0|^2
\]
tangent to the graph of $f$ at $X_0$ and lies above the graph of $f$ for $|X-X_0|<r_0$. Let $\psi$ be the parabola with double opening,
\[
\psi(X)=f(X_0)+\na_Xf(X_0)\cdot (X-X_0)+C|X-X_0|^2.
\]
We approximate $f$ by a sequence of smooth functions $\{\psi_r\}_{r\in (0, r_0)}\subset C^\infty(\T^2\times \Rr)$ that lies between $f$ and $\psi$, namely,
\bq\label{properties:psir}
\begin{aligned}
&\psi_r\le \psi\quad\text{in } B_{r_0}(X_0),\\
& \psi_r\ge f\quad\text{on } \T^2\times [0, T],\\
&\forall \delta\in (0, r_0^2),\quad\psi_r\to f\quad\text{in}\,\, C(\T^2\times [\delta, T-\delta]),\\
&\forall \delta\in (0, r_0^2),\quad \sup_{r\in (0, \sqrt{\delta})}\| \psi_r\|_{W^{1, \infty}(\T^2\times [\delta, T-\delta])}<\infty.
\end{aligned}
\eq
In particular, the $\psi_r(t_0)$'s are uniformly $C^{1, 1}$ at $x_0$, i.e., there exists $M>0$ such that
 \bq\label{psir:uniC11}
 |\psi_r(x_0+x, t_0)+\psi_r(x_0-x, t_0)-2\psi_r(x_0, t_0)|\le M|x|^2\quad\forall |x|<r_0.
 \eq
 Since each $\psi_r$ is a valid test function for the viscosity subsolution $f$, we have
 \[
 \p_t \psi_r(x_0, t_0)\le -\ka \big(G(\psi_r)\psi_r\big)(x_0, t_0).
 \]
Noting that $\p_tf(x_0, t_0)=\p_t\psi_r(x_0, t_0)$ at the local maximum $t_0$, \eqref{sub:roughtest} would follow from the preceding inequality provided that
\bq\label{conv:DN}
\lim_{r\to 0} \big(G(\psi_r)\psi_r\big)(x_0, t_0)=\big(G(f)f\big)(x_0, t_0).
\eq
To prove \eqref{conv:DN} we let $\phi_r:\Omega_r\equiv\Omega_{\psi_r}\to \Rr$ be the harmonic function defining $G(\psi_r)\psi_r$, and $\phi:\Omega_f\to \Rr$ the corresponding harmonic function for $G(f)f$. By virtue of Corollary \ref{coro:elliptic} and the uniform $C^{1, 1}$ regularity \eqref{psir:uniC11},  the $\phi_r(t_0)$'s and $\phi(t_0)$ are uniformly $C^{1, \alpha}$ at $(x_0, f(x_0))$.  By a variational argument and the Arzel\`a-Ascoli theorem, it can be proven as in \cite{DGN} that
\bq\label{phirtophi}
\lim_{r\to 0}\phi_r(x, z)=\phi(x, z)\quad \forall (x, z)\in \Omega_f\subset\Omega_r.
\eq
There exists a common interior ball $B$ tangent to the graphs of $f(t_0)$ and  $\psi_r(t_0)$ at $(x_0, f(x_0))$.  Assume without loss of generality that $B=B_1(0)$ and $(x_0, f(x_0))=x_*:=(1, 0, \dots, 0)$ (see Remark \ref{rema:rotation}). We note that $g_r:=\phi_r\vert_{\p B}$ and $g:=\phi\vert_{\p B}$ are uniformly $C^{1, \alpha}$ at $(x_0, f(x_0))$, and by Proposition \ref{prop:DNsphere},
\begin{align*}
&\big(G(\psi_r)\psi_r\big)(x_*)=\frac{-1}{3\alpha(3)}\int_{\p B_1(0)}\frac{g_r(y)+g_r(\wt y)-2g_r(x_*)}{|y-x_*|^3}d\sigma(y),\\
&\big(G(f)f\big)(x_*)=\frac{-1}{3\alpha(3)}\int_{\p B_1(0)}\frac{g(y)+g(\wt y)-2g(x_*)}{|y-x_*|^3}d\sigma(y),
\end{align*}
where $\alpha(3)$ is the volume of $B_1$. The convergence \eqref{phirtophi} implies that $g_r\to g$ on $\p B\setminus \{x_*\}$. Finally, the uniform $C^{1, \alpha}$ regularity of $g_r$ at $x_*$ allows us to apply the dominated convergence theorem to conclude \eqref{conv:DN}.
\end{proof}
With the consistency in hand, one can use the sup/inf-convolution technique to establish the following comparison principle for viscosity solutions. We refer to the proof of  \cite[Theorem 6.3]{DGN} for details.
\begin{theo}[{\bf Comparison principle}]\label{theo:comparison:viscosity}
Assume that $f,~g:\T^2\times [0, T]\to \Rr$ are respectively a bounded viscosity subsolution  and supersolution of \eqref{Muskat:DN} on $(0, T)$. If $f(x, 0)\le g(x, 0)$ for all $x\in \T^2$, then $f(x, t)\le g(x, t)$ for all $(x, t)\in \T^2\times [0, T]$.
\end{theo}
The comparison principle together with the translation invariance of \eqref{Muskat:DN} implies the following maximum principles.
\begin{coro}[{\bf Maximum principles}]\label{prop:maxslop}
Let  $f$ be a  viscosity solution of \eqref{Muskat:DN}. If  $f(\cdot, 0)$ has the modulus of continuity $\mu:\Rr^+\to \Rr^+$ then so does $f(\cdot, t)$ for all $t\in [0, T]$. That is,
\bq\label{max:moduli}
|f(x, t)-f(y, t)|\le \mu(|x-y|)\quad\forall (x, t)\in \T^2\times [0, T].
\eq
In particular, for $\mu(z)=z\| \na f(\cdot, 0)\|_{L^\infty(\T^2)}$ we have
\bq\label{max:slope}
\| \na f(\cdot, t)\|_{L^\infty(\T^2)}\le \| \na f(\cdot, 0)\|_{L^\infty(\T^2)}\quad\forall t\in [0, T].
\eq
Finally, if $f_1$ and $f_2$ are viscosity solutions  of \eqref{Muskat:DN} on $[0, T]$, then
\bq
\label{max:amp}
\| f_1(\cdot, t)-f_2(\cdot, t)\|_{L^\infty(\T^2)}\le \| f_1(\cdot, 0)-f_2(\cdot, 0)\|_{L^\infty(\T^2)}\quad\forall t\in [0, T].
\eq
\end{coro}
In order to construct the viscosity solution of \eqref{Muskat:DN} and proving that it satisfies the equation pointwise, we employ the vanishing viscosity limit approach. Precisely, we prove  global regularity for the following   regularized problem
\bq\label{Muskate}
\p_t f=-\ka G(f)f+\eps\Delta f,\quad \eps\in (0, 1),\quad f(x, 0)=f_0(x).
\eq
We shall refer to \eqref{Muskate} as the $\eps$-Muskat equation.
\begin{theo}\label{theo:Muskate}
For any $f_0\in H^s(\T^2)$ with $s\ge 3$, there exists a unique global solution $f$ to \eqref{Muskate} with
\bq
f\in C([0, T]; H^s)\cap L^2([0, T]; H^{s+1})\quad\forall T>0.
\eq
Moreover, $f$ satisfies the maximum principles
\bq\label{max:Muskate}
\| f(t)\|_{L^\infty(\T^2)}\le \| f_0\|_{L^\infty(\T^2)},\quad \| f(t)\|_{W^{1, \infty}(\T^2)}\le \| f_0\|_{W^{1, \infty}(\T^2)}\quad\forall t>0.
\eq
\end{theo}
We observe that, on one hand, the added Laplacian $\eps \Delta$ retains the maximum principle \eqref{max:Muskate} satisfied by smooth solutions of \eqref{Muskat:DN}. On the other hand, at the $L^2$ level it only furnishes the $L^2_t H^1_x$ estimate which is weaker than $L^\infty_tW^{1, \infty}_x$ given by \eqref{max:Muskate}.  Therefore, the Laplacian  can only yield better control at higher Sobolev regularity which  in turn requires higher Sobolev estimates  for $G(f)f$. This is a delicate task since the available  Lipschitz control is critical for the operator.  We shall derive global $H^2$ estimates in Section \ref{section:H2est}, then global $H^3$ and higher $H^s$ estimates in Section \ref{section:H3}. The existence and uniqueness follow from a priori estimates and standard arguments and hence shall be omitted.
\section{\texorpdfstring{$H^2$}{} estimate for \texorpdfstring{$\eps$}{}-Muskat}\label{section:H2est}
 We first apply  Proposition \ref{prop:reformDN_b} with $g=\ka f$ to express  the nonlinearity $\ka G(f)f$ as
\bq\label{kG(f)f}
\ka G(f)f=G(f)(\ka f)=\int_{\T^2}\na_{x,z}\Gamma(x-x', f(x)-f(x'))\wp w(x')\cdot N(x)dx',
\eq
where
\bq\label{Fw}
w(x)=\p_2\tt(x)T_1(x)-\p_1\tt(x)T_2(x),\quad\quad T_j(x)=\partial_j(x,f(x,t)),
\eq
and $\theta=(\mez I+K)^{-1}(\ka f)$, i.e.
\bq\label{thetaeq}
\frac12\theta(x)+\, p.v.\int_{\T^2}\nabla_{x,z} \Gamma(x-x',f(x)-f(x'))\cdot N(x')\theta(x')dx'=-\ka f(x).
\eq
Moreover, $\tt$ and $w$ satisfy the relations
\bq\label{d1theta}
\frac12\partial_1\theta(x)-\int_{\T^2}\nabla_{x,z} \Gamma(x-x',f(x)-f(x'))\wedge w(x')dx'\cdot T_1(x)=-\ka\partial_1f(x)
\eq
and
\bq\label{d2theta}
\frac12\partial_2\theta(x)-\int_{\T^2}\nabla_{x,z} \Gamma(x-x',f(x)-f(x'))\wedge w(x')dx'\cdot T_2(x)=-\ka\partial_2f(x).
\eq
We shall make use of the following representation of the fundamental solution $\Gamma$ of the Laplace equation on $\T^2\times\Rr$:
\begin{equation}
	\label{eq2.25:0}
	\Gamma(x,z)=\frac{-\eta_1(x,z)}{4\pi \sqrt{|x|^2+z^2}}+\eta_2(x,z)+\eta_3(x,z)|z|,
\end{equation}
where the $\eta_j$'s are smooth functions satisfying
\begin{itemize}
	\item $\eta_1$ is supported in $B_1$ and equals to $1$ in $B_{1/2}$;
	\item $\eta_2$ decays exponentially along with its derivatives as $|z|\to \infty$;
	\item  $\eta_3$ is supported in $\{|z|\ge 1\}$ and equals to $1/2$ when $|z|\ge 10$.
\end{itemize}
The preceding representation is proven in Lemma \ref{lemA1}. In this and the next section, to derive Sobolev a priori estimates, we shall assume that $f$ is a smooth  global solution satisfying the maximum principles in \eqref{max:Muskate}. Then by virtue of Proposition \ref{prop:W1p} we can choose a number  $r>2$  depending only on $ \| f_0\|_{W^{1, \infty}(\T^2)}$ such that
\bq\label{NablaThetaLr}
\|\theta\|_{W^{1,p}}\leq C(\|f\|_{W^{1,\infty}}) \|f\|_{W^{1,p}}\quad \forall p\in (1, r].
\eq
Without loss of generality we choose $r$ to be close to $2$.  By Morrey's inequality, \eqref{NablaThetaLr} allows us to control the H\"older norm of $\tt$,
\bq\label{tt:Holder}
\|\theta\|_{C^{1-\frac{2}{p}}}\leq C(\|f\|_{W^{1,\infty}}) \|f\|_{W^{1,p}}\quad\forall p\in (2, r].
\eq
Moreover, in view of \eqref{Fw},  \eqref{NablaThetaLr} implies the $L^p$ bound for $w$
\bq\label{w:Lp}
\| w\|_{L^p}\le C(\|f\|_{W^{1,\infty}})\| f\|_{W^{1, p}}\quad \forall p\in (1, r].
\eq
A priori estimates for the $\eps$-Muskat problem \eqref{Muskate} starts with the following $L^2$ estimate, which follows at once from the positivity of $G(f)$,
\bq\label{fL2bound}
\|f(t)\|^2_{L^2}+2\eps\int_0^t\|\nabla f(\tau)\|^2_{L^2}d\tau\leq \|f_0\|^2_{L^2}.
\eq
For the $H^1$ estimate, we note that
\begin{align*}
	\mez \frac{d}{dt}\|\na f(t)\|_{L^2}^2
	&=\ka\int\Delta f G(f)f dx-\eps\|\Delta f\|^2_{L^2}\\
	&\leq \ka\|\Delta f\|_{L^2}\|G(f)f\|_{L^2}-\eps\|\na^2 f\|^2_{L^2}
\end{align*}
and  recall from \eqref{DN:Lp} that
$$
\|G(f)f\|_{L^2}\leq C(\|f\|_{W^{1,\infty}})\|\nabla f\|_{L^2}.
$$
Then using Young's inequality, we find
\begin{align*}
	\mez \frac{d}{dt}\|\na f(t)\|_{L^2}^2\le  \frac1\eps C(\|f\|_{W^{1,\infty}})\|\nabla f\|^2_{L^2}-\frac{\eps}2\|\nabla^2 f\|^2_{L^2},
\end{align*}
where $C$ depends only on $\ka$. Integrating in time and invoking the $L^2$ estimate \eqref{fL2bound}, we obtain the $H^1$ estimate
\bq\label{NablafL2bound}
\|\nabla f(t)\|^2_{L^2}+\eps\int_0^t\|\nabla^2 f(\tau)\|^2_{L^2}d\tau\leq \|\nabla f_0\|^2_{L^2}+\eps^{-2}C(\|f_0\|_{W^{1,\infty}})\|f_0\|_{L^2}^2,
\eq
where we have used the maximum principle \eqref{max:Muskate}.

The remainder of this section is devoted to the $H^2$ estimate for \eqref{Muskate}. Since
\begin{align}
	\begin{split}\label{evolH2}
	\mez\frac{d}{dt} \| \Delta f\|_{L^2}^2&=-\int \Delta f\Delta (\ka G(f)f)dx+\eps\int \Delta f\Delta ^2f dx\\
	&\leq \frac1{2\eps}\|\nabla(\ka G(f)f)\|_{L^2}^2-\frac\eps2\|\nabla^3 f\|^2_{L^2},
\end{split}
\end{align}
it amounts to control the $L^2$ norm of $\na(G(f)f)$. To this end we  use  the formulas \eqref{kG(f)f}  and \eqref{eq2.25:0} for $G(f)f$  and split
\bq\label{IandR1}
\|\nabla(\ka G(f)f)\|_{L^2}=I+R_1,
\eq
where $I$ is  most singular term
$$
I=\Big\|\nabla_x\int\frac{(x',\delta_{x'}f(x))\wedge w(x-x')}{4\pi[|x'|^2+(\delta_{x'}f(x))^2]^{3/2}}\eta_1(x',\delta_{x'}f(x))\cdot N(x)dx'\Big\|_{L^2_x},
$$
 while $R_1$  gathers the remaining terms involving $\nabla\eta_1$, $\nabla^2\eta_1$,  $\eta_2$, $\nabla\eta_2$, $\nabla^2\eta_2$, $\eta_3$, $\nabla\eta_3$ and $\nabla^2\eta_3$. We have introduced the notation
 \bq
 \delta_{x'}f(x)=f(x)-f(x-x').
 \eq
 With the properties of the $\eta_j$'s, the integrals in $R_1$ are not singular and can be controlled by
  \bq\label{R1b}
R_1\leq C(\|f\|_{W^{1,\infty}})\|w\|_{L^2}\|\Delta f\|_{L^2}\leq C(\|f\|_{W^{1,\infty}})\|\Delta f\|_{L^2},
\eq
where we have invoked \eqref{w:Lp}. To control the singular term $I$, we further split
\bq\label{Isplit}
I\leq I_1+I_2+I_3+I_4+I_5,
\eq
where
$$
I_1=\Big\|\int\frac{(x',\delta_{x'}f(x))\wedge w(x-x')}{4\pi[|x'|^2+(\delta_{x'}f(x))^2]^{3/2}}\partial_z\eta_1(x',\delta_{x'}f(x))\delta_{x'}\nabla f(x)\cdot N(x)dx'\Big\|_{L^2_x},
$$
$$
I_2=\Big\|\int\frac{(x',\delta_{x'}f(x))\wedge w(x-x')}{4\pi[|x'|^2+(\delta_{x'}f(x))^2]^{3/2}}\eta_1(x',\delta_{x'}f(x))\cdot \nabla N(x)dx'\Big\|_{L^2_x},
$$
$$
I_3=\Big\|\int\frac{(0,\delta_{x'}\nabla f(x))\wedge w(x-x')}{4\pi[|x'|^2+(\delta_{x'}f(x))^2]^{3/2}}\eta_1(x',\delta_{x'}f(x))\cdot N(x)dx'\Big\|_{L^2_x},
$$
$$
I_4=3\Big\|\int\frac{(x',\delta_{x'}f(x))\wedge w(x-x')}{4\pi[|x'|^2+(\delta_{x'}f(x))^2]^{5/2}}\delta_{x'}f(x)\delta_{x'}\nabla f(x)\eta_1(x',\delta_{x'}f(x))\cdot N(x)dx'\Big\|_{L^2_x},
$$
$$
I_5=\Big\|\int\frac{(x',\delta_{x'}f(x))\wedge \nabla w(x-x')}{4\pi[|x'|^2+(\delta_{x'}f(x))^2]^{3/2}}\eta_1(x',\delta_{x'}f(x))\cdot N(x)dx'\Big\|_{L^2_x}.
$$
Since $\partial_z\eta_1$ is supported away from the origin, the integral in $I_1$ is not singular and using \eqref{w:Lp} gives
\begin{equation}\label{I1B}
	I_1\leq C(\|f\|_{W^{1,\infty}})\|w\|_{L^2}\|\nabla f\|_{L^2}\leq C(\|f\|_{W^{1,\infty}})\|\Delta f\|_{L^2}.
\end{equation}
The terms $I_j$ with $j\in\{2, 3, 4, 5\}$ involve $\eta_1(x',\delta_{x'}f(x))$ which shall be approximated by an even cutoff function $\eta_4(x')$, where
\bq\label{eta4def}
0\leq \eta_4(x)\in C^{\infty}_c(\Rr^2),\quad \mathrm{supp}(\eta_4)\subset B_2,\quad \eta_4=1 ~\mbox{on}~ B_{1}.
\eq
\textbf{Estimates for $I_2$.} Using $\eta_4$ we split
\bq\label{I2split}
I_2\leq I_{2,1}+I_{2,2}+I_{2,3}+I_{2,4},
\eq
where
$$
I_{2,1}=\Big\|\int\frac{(x',\delta_{x'}f(x))\wedge w(x-x')}{4\pi[|x'|^2+(\delta_{x'}f(x))^2]^{3/2}}(\eta_1(x',\delta_{x'}f(x))-\eta_4(x'))\cdot \nabla N(x)dx'\Big\|_{L^2_x},
$$
$$
I_{2,2}=\Big\|\int\frac{(0,\delta_{x'}f(x)-\nabla f(x)\cdot x')\wedge w(x-x')}{4\pi[|x'|^2+(\delta_{x'}f(x))^2]^{3/2}}\eta_4(x')\cdot \nabla N(x)dx'\Big\|_{L^2_x},
$$
$$
I_{2,3}=\Big\|\int(x',\nabla f(x)\cdot x')\wedge w(x\!-\!x')\eta_4(x')\cdot\nabla N(x)A(x,x') dx'\Big\|_{L^2_x}
$$
with
\bq\label{Adef}
A(x,x')=(4\pi)^{-1}[|x'|^2+(\delta_{x'}f(x))^2]^{-\frac32}-(4\pi)^{-1}[|x'|^2+(\nabla f(x)\cdot x')^2]^{-\frac32},
\eq
and
$$
I_{2,4}=\Big\|p.v.\int\frac{(x',\nabla f(x)\cdot x')\wedge w(x-x')}{4\pi[|x'|^2\!+\!(\nabla f(x)\cdot x')^2]^{3/2}}\eta_4(x')\cdot \nabla N(x)dx'\Big\|_{L^2_x}.
$$
Since $\eta_1(x',\delta_{x'}f(x))-\eta_4(x')=0$ for $x'$ near $0$, the integral in $I_{2,1}$ is not singular and is bounded by
$$
I_{2,1}\leq C(\|f\|_{W^{1,\infty}})\|\Delta f\|_{L^2}
$$
as in $I_1$. Using the formula \eqref{Fw} for $w$ to calculate the cross product in $I_{2,2}$, we find
$$
I_{2,2}\leq I_{2,2}^1+I_{2,2}^2,
$$
where
$$
I_{2,2}^1=\Big\|\int\frac{(0,\delta_{x'}f(x)-\nabla f(x)\cdot x',0) \partial_2\theta(x-x')}{4\pi[|x'|^2+(\delta_{x'}f(x))^2]^{3/2}}\eta_4(x')\cdot \nabla N(x)dx'\Big\|_{L^2_x},
$$
$$
I_{2,2}^2=\Big\|\int\frac{(\delta_{x'}f(x)-\nabla f(x)\cdot x',0) \partial_1\theta(x-x')}{4\pi[|x'|^2+(\delta_{x'}f(x))^2]^{3/2}}\eta_4(x')\cdot \nabla N(x)dx'\Big\|_{L^2_x}.
$$
The identity $\partial_{2}\theta(x-x')=\partial_{x'_2}(\delta_{x'}\theta(x))$ allows us to integrate by parts in $I_{2, 2}^1$ and obtain
\bq\label{I221}
I_{2,2}^1\leq I_{2,2}^{1,1}+I_{2,2}^{1,2}+I_{2,2}^{1,3},
\eq
where
$$
I_{2,2}^{1,1}=\Big\|\int\delta_{x'}\theta(x)\frac{(0,\delta_{x'}f(x)-\nabla f(x)\cdot x',0) }{4\pi[|x'|^2+(\delta_{x'}f(x))^2]^{3/2}}\partial_2\eta_4(x')\cdot \nabla N(x)dx'\Big\|_{L^2_x},
$$
$$
I_{2,2}^{1,2}=\Big\|\int\delta_{x'}\theta(x)\frac{(0,\delta_{x'}\partial_2f(x),0) }{4\pi[|x'|^2+(\delta_{x'}f(x))^2]^{3/2}}\eta_4(x')\cdot \nabla N(x)dx'\Big\|_{L^2_x},
$$
$$
I_{2,2}^{1,3}=3\Big\|\!\int\!\delta_{x'}\theta(x)\frac{(0,\delta_{x'}f(x)\!-\!\nabla f(x)\cdot x',0)(x'_2\!+\!\delta_{x'}f(x)\partial_2 f(x\!-\!x')) }{4\pi[|x'|^2+(\delta_{x'}f(x))^2]^{5/2}}\eta_4(x')\!\cdot\!\nabla N(x)dx'\Big\|_{L^2_x}.
$$
Since $\partial_2\eta_4$ is supported away from the origin,  \eqref{tt:Holder} provides the $L^\infty$ bound for $\tt$, and hence
$$
I_{2, 2}^{1,1}\leq C(\|f\|_{W^{1,\infty}})\|\theta\|_{L^\infty}\|\Delta f\|_{L^2}\leq C(\|f\|_{W^{1,\infty}})\|\Delta f\|_{L^2}.
$$
In order to control $I_{2,2}^{1,2}$ we distinguish  $|x'|<\delta$ versus $|x'|>\delta$. For $|x'|<\delta$ we use   the H\"older estimate \eqref{tt:Holder} and the estimate
\bq\label{TVMo2}
|\delta_{x'}f(x)-\nabla f(x)\cdot x'|\leq |x'|^2\int_0^1ds|\nabla^2f(x-sx')|.
\eq
It follows that
\begin{align}
	\begin{split}\label{I2212approach}
	I_{2,2}^{1,2}\leq& \|\theta\|_{\dot{C}^{1-\frac2r}}\Big\|\!\int_0^1ds\int_{|x'|<\delta}\!\frac{
		|\nabla^2 f(x-sx')|}{4\pi|x'|^{1+\frac2r}}dx'|\nabla^2f(x)|\Big\|_{L^2_x}\\
	&	+
	C\|\theta\|_{L^\infty}\|f\|_{W^{1,\infty}}\Big\|\!\int_0^1ds\int_{|x'|>\delta}\!\frac{dx'}{|x'|^{3}}|\nabla^2f(x)|\Big\|_{L^2_x}\\
	\leq&C(\|f\|_{W^{1,\infty}})\delta^{1-\frac2r}\|\nabla^2f\|_{L^4}^2+C(\|f\|_{W^{1,\infty}})\delta^{-1}\|\Delta f\|_{L^2}.
	\end{split}
\end{align}
We recall the Gagliardo-Nirenberg inequality
\bq\label{GNL2Linf}
\|\nabla u\|_{L^p}\leq C_p\|\nabla^2u\|^{1-\frac2p}_{L^2}\|u\|^{\frac2p}_{L^\infty},\quad 4\leq p<\infty.
\eq
Applying \eqref{GNL2Linf} with $u=\na f$ and $p=4$, then taking $\delta^{1-\frac2r}=\eps/(2^{10}C(\|f\|_{W^{1,\infty}}))$, we obtain
$$
I_{2,2}^{1,2}\leq \frac{\eps}{2^{10}}\|\nabla^3f\|_{L^2}+C(\|f\|_{W^{1,\infty}},\eps)\|\Delta f\|_{L^2}.
$$
Next, in view of \eqref{TVMo2},
we see that the above argument for $I_{2,2}^{1,2}$ works for $I_{2,2}^{1,3}$, giving
$$
I_{2,2}^{1,3}\leq \frac{\eps}{2^{10}}\|\nabla^3f\|_{L^2}+C(\|f\|_{W^{1,\infty}},\eps)\|\Delta f\|_{L^2}.
$$
Putting together the above estimates for $I_{2, 2}^{1, j}$ yields
$$
I_{2,2}^{1}\leq \frac{\eps}{2^{9}}\|\nabla^3f\|_{L^2}+C(\|f\|_{W^{1,\infty}},\eps)\|\Delta f\|_{L^2}.
$$
Replacing $\p_2\tt$ by $\p_1\tt$ in $I^1_{2, 2}$ we obtain the same bound for $I_{2,2}^2$, and hence
$$
I_{2,2}\leq \frac{\eps}{2^{8}}\|\nabla^3f\|_{L^2}+C(\|f\|_{W^{1,\infty}},\eps)\|\Delta f\|_{L^2}.
$$
Moving to the term $I_{2,3}$ we use  \eqref{Fw} and the identity $\partial_2T_1-\partial_1T_2=0$ to have
\bq\label{ibp:w}
\begin{aligned}
w(x-x')&=\partial_{x'_2}(\delta_{x'}\theta(x))T_1(x-x')-\partial_{x'_1}(\delta_{x'}\theta(x))T_2(x-x')\\
&=\partial_{x'_2}\left(\delta_{x'}\theta(x)T_1(x-x')\right)-\partial_{x'_1}\left(\delta_{x'}\theta(x)T_2(x-x')\right).
\end{aligned}
\eq
This allows us to integrate by parts in  $I_{2, 3}$. When no $x'$-derivative  hits $A(x,x')$ in \eqref{Adef}, we use the estimate
\bq\label{Abound}
|A(x,x')|\leq C\frac{|\delta_{x'}f(x)-\nabla f(x)\cdot x'|}{|x'|^4}\leq \frac{C}{|x'|^2}\int_0^1ds|\nabla^2f(x-sx')|
\eq
to bound the term as in $I_{2,2}^{1,2}$  \eqref{I221} with the approach in \eqref{I2212approach}. Denoting
$$
a(x,x')=[|x'|^2+(\nabla f(x)\cdot x')^2]^{1/2}\quad \mbox{and}\quad b(x,x')=[|x'|^2+(\delta_{x'}f(x))^2]^{1/2},
$$
 we write  $A(x,x')$ as
$$
A(x,x')=\frac{a^4(x,x')+a^2(x,x')b^2(x,x')+b^4(x,x')}{4\pi a^3(x,x')b^3(x,x')[a^3(x,x')+b^3(x,x')]}(\nabla f(x)\cdot x'+\delta_{x'}f(x))(\nabla f(x)\cdot x'-\delta_{x'}f(x))
$$
to obtain
\begin{align*}
|\nabla_{x'}A(x,x')|\leq& C(\|\nabla f\|_{L^\infty})\frac{|\nabla f(x)\cdot x'-\delta_{x'}f(x)|+|x'||\nabla f(x)-\nabla f(x-x')|}{|x'|^5}\\
\leq& \frac{C(\|\nabla f\|_{L^\infty})}{|x'|^3}\int_0^1ds|\nabla^2f(x-sx')|.
\end{align*}
This allows us to handle the term in which the $x'$-derivatives hit $A(x,x')$ similarly to $I_{2,2}^{1,2}$. We conclude that
$$
I_{2,3}\leq \frac{\eps}{2^{8}}\|\nabla^3f\|_{L^2}+C(\|f\|_{W^{1,\infty}},\eps)\|\Delta f\|_{L^2}.
$$
The term $I_{2,4}$ can be rewritten as
\bq\label{I24sioForm}
I_{2,4}=\|S(w)\cdot \nabla N\|_{L^2},
\eq
where $S$ a singular integral operator with odd kernel. Thus the method of rotation \cite{Duo, CordobaGancedo2007} yields
\bq\label{bound:Sop}
\|S\|_{L^p\to L^p}\leq C_p \sup_{x',x\in\T^2}\Big|\frac{(\frac{x'}{|x'|},\nabla f(x)\cdot \frac{x'}{|x'|})}{4\pi[1+(\nabla f(x)\cdot \frac{x'}{|x'|})^2]^{3/2}}\Big|\leq \frac{C_p}{4\pi},\quad 1<p<\infty.
\eq
Since $r>2$ and $r$ is close to $2$, we can choose $q>4$ such that  $\frac1r+\frac1{q}=\frac12$, $4\ll q <\infty$. Then we can apply \eqref{bound:Sop} with $p=r$ and apply \eqref{GNL2Linf} with $p=q$ to have
\bq\label{I24b}
I_{2,4}\leq C\|w\|_{L^r}\|\nabla^2 f\|_{L^q}\leq C(\|f\|_{W^{1,\infty}})\|\nabla^3 f\|^{2/r}_{L^2}.
\eq
Since $r>2$, Young's inequality gives
$$
I_{2,4}\le  \frac{\eps}{2^{8}}\|\nabla^3f\|_{L^2}+C(\|f\|_{W^{1,\infty}},\eps).
$$
We note that the $W^{1, 2+\eps}$ estimate \eqref{NablaThetaLr} is crucial in the preceding argument. Gathering the above estimates for the $I_{2,j}$ terms, we arrive that the following estimate for $I_2$,
\bq\label{I2b}
I_{2}\leq  \frac{\eps}{2^{6}}\|\nabla^3f\|_{L^2}+C(\|f\|_{W^{1,\infty}},\eps)(\|\Delta f\|_{L^2}+1).
\eq

\textbf{Estimates for $I_3$ and $I_4$.}  We approximate the terms $\delta_{x'}\nabla f(x)$ and $\eta_1(x',\delta_{x'}f(x))$ by $\nabla^2f(x)\cdot x'$ and $\eta_4(x')$ respectively, so that
\bq\label{I3split}
I_3\leq I_{3,1}+I_{3,2}+I_{3,3},
\eq
where
$$
I_{3,1}=\Big\|\int\frac{(0,\delta_{x'}\nabla f(x)-\nabla^2f(x)\cdot x')\wedge w(x-x')}{4\pi[|x'|^2+(\delta_{x'}f(x))^2]^{3/2}}(\eta_1(x',\delta_{x'}f(x))-\eta_4(x'))\cdot N(x)dx'\Big\|_{L^2_x},
$$
$$
I_{3,2}=\Big\|\int\frac{(0,\delta_{x'}\nabla f(x)-\nabla^2f(x)\cdot x')\wedge w(x-x')}{4\pi[|x'|^2+(\delta_{x'}f(x))^2]^{3/2}}\eta_4(x')\cdot N(x)dx'\Big\|_{L^2_x},
$$
$$
I_{3,3}=\Big\|\int\frac{(0,\nabla^2f(x)\cdot x')\wedge w(x-x')}{4\pi[|x'|^2+(\delta_{x'}f(x))^2]^{3/2}}\eta_1(x',\delta_{x'}f(x))\cdot N(x)dx'\Big\|_{L^2_x}.
$$
The non-singular term $I_{3,1}$ can be bounded as before by
\begin{equation*}\label{I31b}
I_{3,1}\leq  C(\|f\|_{W^{1,\infty}})\|\Delta f\|_{L^2}.
\end{equation*}
We integrate by parts using \eqref{ibp:w} to have
$$
I_{3,2}\leq I_{3,2}^{1}+I_{3,2}^{2}+I_{3,2}^{3}
$$
with
$$
I_{3,2}^{1}\!=\!\Big\|\!\int\!\delta_{x'}\theta(x)\frac{(\delta_{x'}\nabla f(x)\!-\!\nabla^2f(x)\!\cdot\! x')(\nabla\eta_4(x'),0)}{4\pi[|x'|^2+(\nabla f(x)\cdot x')^2]^{3/2}}\!\cdot\! N(x)dx'\Big\|_{L^2_x},
$$
$$
I_{3,2}^{2}\!=\!\Big\|\!\int\!\delta_{x'}\theta(x)\frac{(\delta_{x'}\nabla^2 f(x),0)\eta_4(x')}{4\pi[|x'|^2+(\delta_{x'}f(x))^2]^{3/2}}\!\cdot\! N(x)dx'\Big\|_{L^2_x},
$$
$$
I_{3,2}^{3}\!=3\Big\|\!\int\!\delta_{x'}\theta(x)\frac{(\delta_{x'}\nabla f(x)\!-\!\nabla^2f(x)\!\cdot\! x')(x'+\delta_{x'}f(x)\nabla f(x\!-\!x'),0)\eta_4(x')}{4\pi[|x'|^2+(\delta_{x'}f(x))^2]^{5/2}}\!\cdot\! N(x)dx'\Big\|_{L^2_x}.
$$
Since $\na \eta_4$ is supported away from the origin, $I_{3,2}^{1}$ can be easily controlled by
\begin{equation*}
	I_{3,2}^{1}\leq  C(\|f\|_{W^{1,\infty}})\|\Delta f\|_{L^2}.
\end{equation*}
Regarding $I_{3,2}^{2}$ we employ an approach similar to \eqref{I2212approach},
\bq\label{I32:small-large}
\begin{aligned}
	I_{3,2}^{2}\leq& C(\|f\|_{W^{1,\infty}})\|\theta\|_{\dot{C}^{1-\frac2r}}\Big\|\int_0^1ds\int_{|x'|<\delta}\!\frac{
		|\nabla^3 f(x-sx')|}{|x'|^{1+\frac2r}}dx'\Big\|_{L^2_x}\\
&	+
	\|f\|_{W^{1,\infty}}\|\theta\|_{L^\infty}\Big(\Big\|\int_{|x'|>\delta}\!\frac{
		|\nabla^2 f(x-x')|}{|x'|^{3}}dx'\Big\|_{L^2_x}+\Big\|\!\int_{|x'|>\delta}\!\frac{
		|\nabla^2 f(x)|}{|x'|^{3}}dx'\Big\|_{L^2_x}\Big)\\
	\leq&C(\|f\|_{W^{1,\infty}})(\delta^{1-\frac2r}\|\nabla^3f\|_{L^2}+\delta^{-1}\|\Delta f\|_{L^2}).
\end{aligned}
\eq
Choosing $\delta^{1-\frac2r}=\eps/(2^{10}C(\|f\|_{W^{1,\infty}}))$ yields
$$
I_{3,2}^{2}\leq \frac{\eps}{2^{10}}\|\nabla^3f\|_{L^2}+C(\|f\|_{W^{1,\infty}},\eps)\|\Delta f\|_{L^2}.
$$
By an analogous argument, $I_{3,2}^{3}$ is controlled by the same right-hand side, so that
$$
I_{3,2}\leq \frac{\eps}{2^{9}}\|\nabla^3f\|_{L^2}+C(\|f\|_{W^{1,\infty}},\eps)\|\Delta f\|_{L^2}.
$$
The term $I_{3,3}$ can be treated similarly to $I_{2}$ in (\ref{Isplit}, \ref{I2b}), giving
$$
I_{3,3}\leq  \frac{\eps}{2^{9}}\|\nabla^3f\|_{L^2}+C(\|f\|_{W^{1,\infty}},\eps)(\|\Delta f\|_{L^2}+1).
$$
Gathering the above estimates for $I_{3,j}$, we obtain
\bq\label{I3b}
I_{3}\leq \frac{\eps}{2^{7}}\|\nabla^3f\|_{L^2}+C(\|f\|_{W^{1,\infty}},\eps)(\|\Delta f\|_{L^2}+1).
\eq
The next term $I_4$ in \eqref{Isplit} can be treated similarly to $I_3$, whence
\bq\label{I4b}
I_{4}\leq \frac{\eps}{2^{7}}\|\nabla^3f\|_{L^2}+C(\|f\|_{W^{1,\infty}},\eps)(\|\Delta f\|_{L^2}+1).
\eq
\textbf{Estimates for $I_5$.} We use the same splitting scheme as in $I_2$ to have
\begin{equation*}
	I_5\leq I_{5,1}+I_{5,2}+I_{5,3}+I_{5,4},
\end{equation*}
where
$$
I_{5,1}=\Big\|\int\frac{(x',\delta_{x'}f(x))\wedge \nabla w(x-x')}{4\pi[|x'|^2+(\delta_{x'}f(x))^2]^{3/2}}(\eta_1(x',\delta_{x'}f(x))-\eta_4(x'))\cdot N(x)dx'\Big\|_{L^2_x},
$$
$$
I_{5,2}=\Big\|\int\frac{(0,\delta_{x'}f(x)-\nabla f(x)\cdot x')\wedge \nabla w(x-x')}{4\pi[|x'|^2+(\delta_{x'}f(x))^2]^{3/2}}\eta_4(x')\cdot N(x)dx'\Big\|_{L^2_x},
$$
$$
I_{5,3}=\Big\|\int(x',\nabla f(x)\cdot x')\wedge \nabla w(x-x')A(x,x')\eta_4(x')\cdot N(x)dx'\Big\|_{L^2_x},
$$
$$
I_{5,4}=\Big\|\int\frac{(x',\nabla f(x)\cdot x')\wedge \nabla w(x-x')}{4\pi[|x'|^2+(\nabla f(x)\cdot x')^2]^{3/2}}\eta_4(x')\cdot N(x)dx'\Big\|_{L^2_x}.
$$
Since $\nabla w(x-x')=-\nabla_{x'}(w(x-x'))$, we can integrate  by parts in $I_{5,1}$, $I_{5,2}$ and $I_{5,3}$, then use an approach similar to the one for $I_{3}$ in \eqref{Isplit}. It yields analogously that
$$
I_{5,1}+I_{5,2}+I_{5,3}\leq \frac{\eps}{2^{7}}\|\nabla^3f\|_{L^2}+C(\|f\|_{W^{1,\infty}},\eps)(\|\Delta f\|_{L^2}+1).
$$
The integral in $I_{5,4}$ is  a singular integral operator as in \eqref{bound:Sop}, hence
\begin{align}
	\begin{split}\label{nablaL2bclaim}
		I_{5,4}&\leq C(\|f\|_{W^{1,\infty}})\|\nabla w\|_{L^2}\\
		&\leq C(\|f\|_{W^{1,\infty}})(\||\nabla\theta||\nabla^2f|\|_{L^2}+\|\Delta\theta\|_{L^2})\\
		&\leq C(\|f\|_{W^{1,\infty}})(\|\nabla\theta\|_{L^r}\|\nabla^2f\|_{L^q}+\|\Delta\theta\|_{L^2})\\
		&\leq C(\|f\|_{W^{1,\infty}})(\|\nabla^3f\|^{2/r}_{L^2}+\|\Delta\theta\|_{L^2}),
	\end{split}
\end{align}
where $\frac1r+\frac1{q}=\frac12$ and we have employed \eqref{NablaThetaLr} (with $p=r$) and \eqref{GNL2Linf} (with $p=q$). We claim that
\bq\label{tt:H2:claim}
\|\Delta\theta\|_{L^2}\leq \delta\|\nabla^3f\|_{L^2}+ C(\|f\|_{W^{1,\infty}},\delta)(\|\Delta f\|_{L^2}+1)\quad\forall \delta\in (0, 1),
\eq
whose proof is postponed to Lemma \ref{DeltathetaL2} below. With this we deduce that
$$
I_{5,4}\leq \frac{\eps}{2^{7}}\|\nabla^3f\|_{L^2}+C(\|f\|_{W^{1,\infty}},\eps)(\|\Delta f\|_{L^2}+1),
$$
which in turn implies
\bq\label{I5b}
I_{5}\leq \frac{\eps}{2^{6}}\|\nabla^3f\|_{L^2}+C(\|f\|_{W^{1,\infty}},\eps)(\|\Delta f\|_{L^2}+1).
\eq
Putting together the estimates \eqref{I1B}, \eqref{I2b}, \eqref{I3b}, \eqref{I4b}, \eqref{I5b} for $I_j$ and  \eqref{R1b} for $R_1$, we arrive at
$$
I+R_1\leq \frac{\eps}{2^{4}}\|\nabla^3f\|_{L^2}+C(\|f\|_{W^{1,\infty}},\eps)(\|\Delta f\|_{L^2}+1).
$$
By interpolating $\|\Delta f\|_{L^2}$ between $ \|\nabla^3f\|_{L^2}$ and $\|\nabla f\|_{L^2}\le C\|\nabla f\|_{L^\infty}$, it follows that
\bq
I+R_1\leq \frac{\eps}{2^{3}}\|\nabla^3f\|_{L^2}+C(\|f\|_{W^{1,\infty}},\eps)(\|\na f\|_{L^2}+1).
\eq
Consequently, \eqref{IandR1} implies
\bq
\frac1{2\eps}\|\nabla(\ka G(f)f)\|_{L^2}^2\le \frac{\eps}{2^{6}}\|\nabla^3f\|_{L^2}^2+C(\|f\|_{W^{1,\infty}},\eps)(\|\na f\|^2_{L^2}+1).
\eq
Then we  insert this into \eqref{evolH2}, integrate in time, use \eqref{fL2bound}  to bound $\int_0^t\|\na f(\tau)\|^2_{L^2}d\tau$, and use the maximum principle \eqref{max:Muskate} for $\| f(t)\|_{W^{1, \infty}}$. We  arrive at   the $H^2$ bound
\bq\label{DeltafL2b}
\|\Delta f(t)\|^2_{L^2}+\frac\eps2\int_0^t\|\nabla^3 f(\tau)\|^2_{L^2}d\tau\leq\|\Delta f_0\|^2_{L^2}+C(\|f_0\|_{W^{1,\infty}},\eps)(\|f_0\|^2_{L^2}+t).
\eq
Finally, we prove the claim \eqref{tt:H2:claim} in the following lemma.
\begin{lemm}\label{DeltathetaL2}
	Consider $\theta$ and $w$	given by \eqref{thetaeq} and  \eqref{Fw} respectively with $f\in H^3$. Then the following estimates hold for any $\delta>0$,
	\bq\label{DeltathetaL2b}
	\|\Delta\theta\|_{L^2}\leq \delta\|\nabla^3f\|_{L^2} +C(\|f\|_{W^{1,\infty}},\delta)(\|\Delta f\|_{L^2}+1),
	\eq
	\bq\label{nablawL2b}
	\|\nabla w\|_{L^2}\leq \delta\|\nabla^3f\|_{L^2} +C(\|f\|_{W^{1,\infty}},\delta)(\|\Delta f\|_{L^2}+1).
	\eq	
\end{lemm}
\begin{proof} For $\frac{1}{r}+\frac{1}{q}=\mez$, it follows from the $L^r$ estimate \eqref{NablaThetaLr} for $\na \tt$ and the interpolation inequality  \eqref{GNL2Linf} that
$$
\|\nabla w\|_{L^2}\leq \|\nabla\theta\|_{L^r}\|\nabla^2f\|_{L^q}+\|f\|_{W^{1,\infty}}\|\Delta\theta\|_{L^2}
\leq C(\|f\|_{W^{1,\infty}})(\|\nabla^3f\|^{2/r}_{L^2}+\|\Delta\theta\|_{L^2}).
$$
Since $r>2$,  \eqref{DeltathetaL2b} implies \eqref{nablawL2b}.
	
For the proof of \eqref{DeltathetaL2b} we use the formula \eqref{d1theta} to get an expression for $\p_j^2\tt$. Without loss of generality we shall only consider  $j=1$. Since
$$
\mez\|\partial_1^2\theta\|_{L^2}\leq \Big\|\partial_{x_1}\Big(\int\nabla_{x,z} \Gamma(x-x',f(x)-f(x'))\wedge w(x')dx'\cdot T_1(x)\Big)\Big\|_{L^2_x}+\ka\|\partial_1^2f\|_{L^2},
$$
it remains to estimate the first term on the right-hand side of the preceding inequality. We write
$$
\Big\|\partial_{x_1}\Big(\int\nabla_{x,z} \Gamma(x-x',f(x)-f(x'))\wedge w(x')dx'\cdot T_1(x)\Big)\Big\|_{L^2_x}\le J+R_2,
$$
where
$$
J=\Big\|\partial_{x_1}\Big(\int\frac{(x-x',f(x)-f(x'))\wedge w(x')}{4\pi[|x-x'|^2+(f(x)-f(x'))^2]^{3/2}}\eta_1(x-x',f(x)-f(x'))dx'\cdot T_1(x)\Big)\Big\|_{L^2_x}
$$
and  $R_2$ contains terms with $\eta_3$, $\nabla\eta_3$, $\nabla^2\eta_3$, $\eta_2$,  $\nabla\eta_2$, $\nabla^2\eta_2$, $\nabla\eta_1$ and $\nabla^2\eta_1$. With the properties of the $\eta_j$'s, the integrals in $R_2$  are not singular and can be controlled by
\bq\label{r2b}
R_2\leq C(\|f\|_{W^{1,\infty}})\|w\|_{L^2}\|\Delta f\|_{L^2}\leq C(\|f\|_{W^{1,\infty}})\|\Delta f\|_{L^2}.
\eq
Using the formula \eqref{Fw} for $w$ to calculate the triple scalar product in $J$, we find
\bq\label{Jsplit}
J\leq J_1+J_2+J_3,
\eq
where
$$
J_1=\Big\|\partial_{x_1}\Big(\int\frac{(x_2-x_2')(\partial_1f(x)-\partial_1f(x'))\partial_2\theta(x')}{4\pi[|x-x'|^2+(f(x)-f(x'))^2]^{3/2}}\eta_1(x-x',f(x)-f(x'))dx'\Big)\Big\|_{L^2_x},
$$
$$
J_2=\Big\|\partial_{x_1}\Big(\int\frac{(x_2-x_2')(\partial_2f(x)-\partial_2f(x'))\partial_1\theta(x')}{4\pi[|x-x'|^2+(f(x)-f(x'))^2]^{3/2}}\eta_1(x-x',f(x)-f(x'))dx'\Big)\Big\|_{L^2_x},
$$
$$
J_3=\Big\|\partial_{x_1}\Big(\int\frac{(f(x)-f(x')-\nabla f(x)\cdot(x-x'))\partial_1\theta(x')}{4\pi[|x-x'|^2+(f(x)-f(x'))^2]^{3/2}}\eta_1(x-x',f(x)-f(x'))dx'\Big)\Big\|_{L^2_x}.
$$
Regarding $J_1$, we take $\p_1$ and make the change the variables $x'\mapsto x-x'$ to obtain
\bq\label{J1split}
J_1\leq J_{1,1}+J_{1,2}+J_{1,3},
\eq
where
$$
J_{1,1}=\Big\|\int\frac{x_2'\delta_{x'}\partial_1f(x)\partial_2\theta(x-x')}{4\pi[|x'|^2+(\delta_{x'}f(x))^2]^{3/2}}\nabla\eta_1(x',\delta_{x'}f(x))\cdot T_1(x)dx'\Big\|_{L^2_x},
$$
$$
J_{1,2}=\Big\|\partial_1^2f(x)\int\frac{x_2'\partial_2\theta(x-x')}{4\pi[|x'|^2+(\delta_{x'}f(x))^2]^{3/2}}\eta_1(x',\delta_{x'}f(x))dx'\Big\|_{L^2_x},
$$
$$
J_{1,3}=3\Big\|\int\frac{x_2'\delta_{x'}\partial_1f(x)\partial_2\theta(x-x')}{4\pi[|x'|^2+(\delta_{x'}f(x))^2]^{5/2}}(x'_1+\delta_{x'}f(x)\partial_1f(x))\eta_1(x',\delta_{x'}f(x))dx'\Big\|_{L^2_x}.
$$
Using the support of $\na \eta_1$, we find
\bq\label{J11b}
J_{1,1}\leq C(\|f\|_{W^{1,\infty}}+1)\|\nabla f\|_{L^2}\|\nabla\theta\|_{L^2}\leq C(\|f\|_{W^{1,\infty}})\|\Delta f\|_{L^2}.
\eq
Next, we split  $J_{1,2}$ similarly to $I_2$ by using  the cutoff function $\eta_4$ and $A(x, x')$ is given by \eqref{Adef},
$$
J_{1,2}=J_{1,2}^1+J_{1,2}^2+J_{1,2}^3,
$$
where
$$
J_{1,2}^1=\Big\|\partial_1^2f(x)\int\frac{x_2'\partial_2\theta(x-x')}{4\pi[|x'|^2+(\delta_{x'}f(x))^2]^{3/2}}(\eta_1(x',\delta_{x'}f(x))-\eta_4(x'))dx'\Big\|_{L^2_x},
$$
$$
J_{1,2}^2=\Big\|\partial_1^2f(x)\int x_2'\partial_2\theta(x-x')A(x,x')\eta_4(x')dx'\Big\|_{L^2_x},
$$
$$
J_{1,2}^3=\Big\|\partial_1^2f(x)\int\frac{x_2'\partial_2\theta(x-x')\eta_4(x')}{4\pi[|x'|^2+(\nabla f(x)\cdot x')^2]^{3/2}}dx'\Big\|_{L^2_x}.
$$
It is readily seen that
$$
J_{1,2}^1\leq C(\|f\|_{W^{1,\infty}})\|\p_2\tt\|_{L^2}\le  C(\|f\|_{W^{1,\infty}})\|\Delta f\|_{L^2}.
$$
The term  $J_{1,2}^2$ can be treated similarly to $I_{2,3}$ in \eqref{I2split}, giving
$$
J_{1,2}^2\leq \frac{\delta}{2^5}\|\nabla^3 f\|_{L^2}+ C(\|f\|_{W^{1,\infty}},\delta)\|\Delta f\|_{L^2}.
$$
On the other hand, $J^3_{1,2}$ can be treated similarly  to $I_{2,4}$ \eqref{I24sioForm}-\eqref{I24b}, so that
$$
J_{1,2}^3\leq C(\|f\|_{W^{1,\infty}})\|\nabla^3 f\|_{L^2}^{2/r}\leq \frac{\delta}{2^5}\|\nabla^3 f\|_{L^2}+ C(\|f\|_{W^{1,\infty}},\delta).
$$
Combining the last three estimates implies
\bq\label{J12b}
J_{1,2}\leq \frac{\delta}{2^4}\|\nabla^3 f\|_{L^2}+ C(\|f\|_{W^{1,\infty}},\delta)(\|\Delta f\|_{L^2}+1).
\eq
As for $J_{1,3}$ we split
$$
J_{1,3}=J_{1,3}^1+J_{1,3}^2+J_{1,3}^3+J_{1,3}^4+J_{1,3}^5,
$$
where
$$
J_{1,3}^1=3\Big\|\int\frac{x_2'\delta_{x'}\partial_1f(x)\partial_2\theta(x-x')}{4\pi[|x'|^2+(\delta_{x'}f(x))^2]^{5/2}}(x'_1+\delta_{x'}f(x)\partial_1f(x))(\eta_1(x',\delta_{x'}f(x))-\eta_4(x'))dx'\Big\|_{L^2_x},
$$
$$
J_{1,3}^2=3\Big\|\int\frac{x_2'(\delta_{x'}\partial_1f(x)-\nabla\partial_1f(x)\cdot x')\partial_2\theta(x-x')}{4\pi[|x'|^2+(\delta_{x'}f(x))^2]^{5/2}}(x'_1+\delta_{x'}f(x)\partial_1f(x))\eta_4(x')dx'\Big\|_{L^2_x},
$$
$$
J_{1,3}^3=3\Big\|\nabla\partial_1f(x)\cdot\int x'\frac{x_2' \partial_2\theta(x-x')(\delta_{x'}f(x)-\nabla f(x)\cdot x')}{4\pi[|x'|^2+(\delta_{x'}f(x))^2]^{5/2}}\eta_4(x')dx'\partial_1f(x)\Big\|_{L^2_x},
$$
$$
J_{1,3}^4=3\Big\|\nabla\partial_1f(x)\cdot\int x' x_2' \partial_2\theta(x-x')(x'_1+\nabla f(x)\cdot x'\partial_1f(x))B(x,x')\eta_4(x')dx'\Big\|_{L^2_x}
$$
with
\bq\label{Bdef}
B(x,x')=(4\pi)^{-1}[|x'|^2+(\delta_{x'}f(x))^2]^{-\frac52}-(4\pi)^{-1}[|x'|^2+(\nabla f(x)\cdot x')^2]^{-\frac52},
\eq
and
$$
J_{1,3}^5=3\Big\|\nabla\partial_1f(x)\cdot\int x'\frac{x_2' (x'_1+\nabla f(x)\cdot x'\p_1f(x)) \partial_2\theta(x-x')}{4\pi[|x'|^2+(\nabla f(x)\cdot x')^2]^{5/2}}\eta_4(x')dx'\Big\|_{L^2_x}.
$$
As before we have
$$
J_{1,3}^1\leq C(\|f\|_{W^{1,\infty}})\|\Delta f\|_{L^2}.
$$
We treat $J_{1,3}^2$ analogously to $I_{2}$ in (\ref{Isplit}, \ref{I2b}) by writing  $\partial_2\theta(x-x')=\partial_{x'_2}\delta_{x'}\theta(x)$ and integrating by parts. This yields
the desired bound
$$
J_{1,3}^2\leq \frac{\delta}{2^6}\|\nabla^3 f\|_{L^2}+ C(\|f\|_{W^{1,\infty}},\delta)(\|\Delta f\|_{L^2}+1).
$$
The terms $J_{1,3}^3$ and $J_{1,3}^4$ are similar to  $I_{2,2}$ and $I_{2,3}$ respectively, hence
$$
J_{1,3}^3+J_{1,3}^4\leq \frac{\delta}{2^5}\|\nabla^3 f\|_{L^2}+ C(\|f\|_{W^{1,\infty}},\delta)\|\Delta f\|_{L^2}.
$$
The integral in $J_{1,3}^5$ is a  singular integral operator with odd kernel, hence we obtain as in \eqref{I24sioForm}-\eqref{I24b} that
$$
J_{1,3}^5\leq C(\|f\|_{W^{1,\infty}})\|\nabla^3 f\|_{L^2}^{2/r}\leq \frac{\delta}{2^6}\|\nabla^3 f\|_{L^2}+ C(\|f\|_{W^{1,\infty}},\delta).
$$
We have proven that
\bq\label{J13b}
J_{1,3}\leq  \frac{\delta}{2^4}\|\nabla^3 f\|_{L^2}+ C(\|f\|_{W^{1,\infty}},\delta)(\|\Delta f\|_{L^2}+1).
\eq
Combining \eqref{J1split},  \eqref{J11b}, \eqref{J12b} and \eqref{J13b}, we obtain
$$
J_{1}\leq  \frac{\delta}{2^3}\|\nabla^3 f\|_{L^2}+ C(\|f\|_{W^{1,\infty}},\delta)(\|\Delta f\|_{L^2}+1).
$$
The terms $J_2$ and $J_3$ are estimated analogously, so that \eqref{Jsplit} yields
$$
J\leq  \frac\delta2\|\nabla^3 f\|_{L^2}+ C(\|f\|_{W^{1,\infty}},\delta)(\|\Delta f\|_{L^2}+1).
$$
This concludes the proof of the lemma.
 \end{proof}

\section{\texorpdfstring{$H^3$}{} and higher Sobolev estimates for \texorpdfstring{$\eps$}{}-Muskat}\label{section:H3}
Throughout this section we define $q$ by
\bq
\frac1r+\frac1q=\frac12,
\eq
where $r$ is as in \eqref{NablaThetaLr}. In particular, we have $1\ll q<\infty$.

Applying $\Delta$ to \eqref{Muskate} and multiplying the resulting equation by $-\Delta^2 f$, we find
\bq\label{evolH3}
\begin{aligned}
		\mez\frac{d}{dt} \| \na \Delta f\|_{L^2}^2&=-\int \Delta^2 f\Delta f_tdx\\
&=\int \Delta^2 f\Delta (\ka G(f)f)dx-\eps\int|\Delta^2 f|^2dx\\
	    &\leq \frac1{2\eps}\|\Delta(\ka G(f)f)\|_{L^2}^2-\frac\eps2\|\Delta^2 f\|^2_{L^2}.
\end{aligned}
\eq
As in the $H^2$ estimate in the preceding section, we split
$$
\|\Delta(\ka G(f)f)\|_{L^2}\leq K+R_3
$$
where
$$
K=\Big\|\Delta_x\int\frac{(x',\delta_{x'}f(x))\wedge w(x-x')}{4\pi[|x'|^2+(\delta_{x'}f(x))^2]^{3/2}}\eta_1(x',\delta_{x'}f(x))\cdot N(x)dx'\Big\|_{L^2_x}
$$
and $R_3$  contains non-singular terms and satisfies
\bq\label{est:R3}
R_3\leq C(\|f\|_{W^{1,\infty}})\|w\|_{L^2}\|\na^3f\|_{L^2}\leq C(\|f\|_{W^{1,\infty}})\|\na^3 f\|_{L^2}.
\eq
We have
\bq\label{Ksplit}
K\leq K_1+...+K_{12},
\eq
where in the first five terms we gather the terms with Laplacian:
$$
K_1=\Big\|\int\frac{(x',\delta_{x'}f(x))\wedge w(x-x')}{4\pi[|x'|^2+(\delta_{x'}f(x))^2]^{3/2}}\Delta_x(\eta_1(x',\delta_{x'}f(x)))\cdot N(x)dx'\Big\|_{L^2_x},
$$
$$
K_2=\Big\|\int\frac{(x',\delta_{x'}f(x))\wedge w(x-x')}{4\pi[|x'|^2+(\delta_{x'}f(x))^2]^{3/2}}\eta_1(x',\delta_{x'}f(x))\cdot \Delta N(x)dx'\Big\|_{L^2_x},
$$
$$
K_3=\Big\|\int\frac{(0,\delta_{x'}\Delta f(x))\wedge w(x-x')}{4\pi[|x'|^2+(\delta_{x'}f(x))^2]^{3/2}}\eta_1(x',\delta_{x'}f(x))\cdot N(x)dx'\Big\|_{L^2_x},
$$
$$
K_4=\Big\|\frac1{4\pi}\int (x',\delta_{x'}f(x))\wedge w(x-x')\Delta_x([|x'|^2+(\delta_{x'}f(x))^2]^{-3/2})\eta_1(x',\delta_{x'}f(x))\cdot N(x)dx'\Big\|_{L^2_x},
$$
and
$$
K_5=\Big\|\int\frac{(x',\delta_{x'} f(x))\wedge \Delta w(x-x')}{4\pi[|x'|^2+(\delta_{x'}f(x))^2]^{3/2}}\eta_1(x',\delta_{x'}f(x))\cdot N(x)dx'\Big\|_{L^2_x}.
$$
The other terms are given by
$$
K_6=\sum_{j=1}^2\Big\|\int\partial_{x_j}\Big(\frac{(x',\delta_{x'}f(x))\wedge w(x-x')}{2\pi[|x'|^2+(\delta_{x'}f(x))^2]^{3/2}}\cdot N(x)\Big)\partial_3\eta_1(x',\delta_{x'}f(x))\delta_{x'}\partial_jf(x)dx'\Big\|_{L^2_x},
$$
$$
K_7=\sum_{j=1}^2\Big\|\int\frac{(0,\delta_{x'}\partial_jf(x))\wedge w(x-x')}{2\pi[|x'|^2+(\delta_{x'}f(x))^2]^{3/2}}\eta_1(x',\delta_{x'}f(x))\cdot \partial_jN(x)dx'\Big\|_{L^2_x},
$$
$$
K_8=3\sum_{j=1}^2\Big\|\int\frac{(x',\delta_{x'}f(x))\wedge w(x-x')}{2\pi[|x'|^2+(\delta_{x'}f(x))^2]^{5/2}}\delta_{x'}f(x)\delta_{x'}\partial_jf(x)\eta_1(x',\delta_{x'}f(x))\cdot \partial_jN(x)dx'\Big\|_{L^2_x},
$$
$$
K_9=\sum_{j=1}^2\Big\|\int\frac{(x',\delta_{x'}f(x))\wedge \partial_jw(x-x')}{2\pi[|x'|^2+(\delta_{x'}f(x))^2]^{3/2}}\eta_1(x',\delta_{x'}f(x))\cdot \partial_jN(x)dx'\Big\|_{L^2_x},
$$
$$
K_{10}=3\sum_{j=1}^2\Big\|\int\frac{(0,\delta_{x'}\partial_j f(x))\wedge w(x-x')}{2\pi[|x'|^2+(\delta_{x'}f(x))^2]^{5/2}}\delta_{x'}f(x)\delta_{x'}\partial_jf(x)\eta_1(x',\delta_{x'}f(x))\cdot N(x)dx'\Big\|_{L^2_x},
$$
$$
K_{11}=\sum_{j=1}^2\Big\|\int\frac{(0,\delta_{x'}\partial_jf(x))\wedge\partial_jw(x-x')}{2\pi[|x'|^2+(\delta_{x'}f(x))^2]^{3/2}}\eta_1(x',\delta_{x'}f(x))\cdot N(x)dx'\Big\|_{L^2_x},
$$
and finally
$$
K_{12}=3\sum_{j=1}^2\Big\|\int\frac{(x',\delta_{x'}f(x))\wedge \partial_jw(x-x')}{2\pi[|x'|^2+(\delta_{x'}f(x))^2]^{5/2}}\delta_{x'}f(x)\delta_{x'}\partial_jf(x)\eta_1(x',\delta_{x'}f(x))\cdot N(x)dx'\Big\|_{L^2_x}.
$$
Since $\na \eta$ is supported sway from the origin, we have
\bq\label{K1b}
K_1 \leq C(\|f\|_{W^{1,\infty}})\|\nabla^3f\|_{L^2}.
\eq
{\bf Estimates for $K_2$.} We split further $K_2$ as
\bq\label{K2split}
K_2\leq K_{2,1}+K_{2,2}+K_{2,3}+K_{2,4},
\eq
where
$$
K_{2,1}=\Big\|\int\frac{(x',\delta_{x'}f(x))\wedge w(x-x')}{4\pi[|x'|^2+(\delta_{x'}f(x))^2]^{3/2}}(\eta_1(x',\delta_{x'}f(x))-1)\cdot \Delta N(x)dx'\Big\|_{L^2_x},
$$
$$
K_{2,2}=\Big\|\int\frac{(0,\delta_{x'}f(x)-\nabla f(x)\cdot x')\wedge w(x-x')}{4\pi[|x'|^2+(\delta_{x'}f(x))^2]^{3/2}}\cdot \Delta N(x)dx'\Big\|_{L^2_x},
$$
$$
K_{2,3}=\Big\|\int(x',\nabla f(x)\cdot x')\wedge w(x-x')A(x,x')\cdot \Delta N(x)dx'\Big\|_{L^2_x},
$$
with $A$ given in \eqref{Adef} and
$$
K_{2,4}=\Big\|\int\frac{(x',\nabla f(x)\cdot x')\wedge w(x-x')}{4\pi[|x'|^2+(\nabla f(x)\cdot x')^2]^{3/2}}\cdot \Delta N(x)dx'\Big\|_{L^2_x}.
$$
It is readily seen that
$$
K_{2,1} \leq C(\|f\|_{W^{1,\infty}})\|\nabla^3f\|_{L^2}.
$$
Next we consider $K_{2, 2}$. Since $r>2$, we can find a number $2^{-}$ smaller than $2$  such that $\frac1p+\frac1{2^-}+\frac1r=1$  for some  $4\ll p<\infty$. The  identity
\bq\label{TVMIo2}
\delta_{x'}f(x)-\nabla f(x)\cdot x'=-\int_0^1s x'\cdot\nabla^2f(x+(s-1)x')\cdot x' ds
\eq
allows us to bound
\begin{align*}
	K_{2,2}&\leq\sup_{x}\int_0^1sds \int|\nabla^2 f(x+(1-s)x')||w(x-x')|\frac{dx'}{|x'|}\|\nabla^3 f\|_{L^2}\\
	&\leq \int_0^1\frac{sds}{(1-s)^{2/p}} \|\nabla^2f\|_{L^p}\||x'|^{-1}\|_{L^{2^-}_{x'}}\|w\|_{L^r}\|\nabla^3 f\|_{L^2}.
\end{align*}
Using the $L^r$ estimate \eqref{w:Lp} and the interpolation inequality \eqref{GNL2Linf} to bound $ \|\nabla^2f\|_{L^p}$, we find
$$
K_{2,2}\leq C(\|f\|_{W^{1,\infty}})\|\nabla^3f\|_{L^2}^{2-\frac2p}.
$$
 An analogous approach for $K_{2,3}$ yields the same bound
$$
K_{2,3}\leq C(\|f\|_{W^{1,\infty}})\|\nabla^3f\|_{L^2}^{2-\frac2p}.
$$
The integral in $K_{2, 4}$ is a singular integral operator with odd kernel, so \eqref{bound:Sop} implies
\begin{align*}
K_{2,4}&\leq C(\|f\|_{W^{1,\infty}})\|\nabla^3f\|_{L^q}\|w\|_{L^r}\leq C(\|f\|_{W^{1,\infty}})\|\Delta^2f\|_{L^2}^{1-\frac2q}\|\nabla^3 f\|_{L^2}^{\frac2q},
\end{align*}
where we have applied the following Gagliardo-Nirenberg inequality
\bq\label{GNnablaL2L2}
\|u\|_{L^p}\leq C_p\|\nabla u\|^{1-\frac2p}_{L^2}\|u\|^{\frac2p}_{L^2},\quad 2\leq p<\infty.
\eq
After using Young's inequality, we obtain
$$
K_{2,4}
\leq \frac{\eps}{2^{9}}\|\Delta^2 f\|_{L^2}+C(\|f\|_{W^{1,\infty}},\eps)\|\nabla^3 f\|_{L^2}.
$$
Gathering the above estimates for $K_{2, j}$ yields
\bq\label{K2b}
K_{2}
\leq \frac{\eps}{2^{9}}\|\Delta^2 f\|_{L^2}+C(\|f\|_{W^{1,\infty}},\eps)(\|\nabla^3 f\|^2_{L^2}+1).
\eq
{\bf Estimates for $K_3$.}  We use $\eta_4$ in \eqref{eta4def} to decompose $K_3$ as
\bq\label{K3split}
K_3\leq K_{3,1}+K_{3,2}+K_{3,3}
\eq
where
$$
K_{3,1}=\Big\|\int\frac{(0,\delta_{x'}\Delta f(x)-\nabla\Delta f(x)\cdot x')\wedge w(x-x')}{4\pi[|x'|^2+(\delta_{x'}f(x))^2]^{3/2}}(\eta_1(x',\delta_{x'}f(x))-\eta_4(x'))\cdot N(x)dx'\Big\|_{L^2_x},
$$
$$
K_{3,2}=\Big\|\int\frac{(0,\delta_{x'}\Delta f(x)-\nabla\Delta f(x)\cdot x')\wedge w(x-x')}{4\pi[|x'|^2+(\delta_{x'}f(x))^2]^{3/2}}\eta_4(x')\cdot N(x)dx'\Big\|_{L^2_x},
$$
$$
K_{3,3}=\Big\|\int\frac{(0,\nabla\Delta f(x)\cdot x')\wedge w(x-x')}{4\pi[|x'|^2+(\delta_{x'}f(x))^2]^{3/2}}\eta_4(x')\cdot N(x)dx'\Big\|_{L^2_x}.
$$
It is easy to see that
\bqa
K_{3,1} \leq C(\|f\|_{W^{1,\infty}})\|\nabla^3f\|_{L^2}.
\eqa
For the next term $K_{3,2}$ we integrate by parts as in $I_{2,2}$ to get
$$
K_{3,2}\leq K_{3,2}^1+K_{3,2}^2+K_{3,2}^3,
$$
where
$$
K_{3,2}^1=\Big\|\int\delta_{x'}\theta(x)\frac{\delta_{x'}\Delta f(x)-\nabla\Delta f(x)\cdot x'}{4\pi[|x'|^2+(\delta_{x'}f(x))^2]^{3/2}}(\nabla\eta_4(x'),0)\cdot N(x)dx'\Big\|_{L^2_x},
$$
$$
K_{3,2}^2=\Big\|\int\frac{\delta_{x'}\theta(x)(\delta_{x'}\nabla\Delta f(x),0)}{4\pi[|x'|^2+(\delta_{x'}f(x))^2]^{3/2}}\eta_4(x')\cdot N(x)dx'\Big\|_{L^2_x},
$$
$$
K_{3,2}^3=3\Big\|\int\!\delta_{x'}\theta(x)(\delta_{x'}\Delta f(x)\!-\!\nabla\Delta f(x)\!\cdot\! x')\frac{(x'\!+\!\delta_{x'}f(x)\nabla f(x\!-\!x'),0)}{4\pi[|x'|^2\!+\!(\delta_{x'}f(x))^2]^{5/2}}\eta_4(x')\!\cdot\! N(x)dx'\Big\|_{L^2_x}.
$$
The support of $K_{3,2}^1$ implies
$$
K_{3,2}^1\leq C(\|f\|_{W^{1,\infty}})\|\nabla^3 f\|_{L^2}.
$$
The same argument as in \eqref{I32:small-large} yields
$$
K_{3,2}^2\leq \frac{\eps}{2^{9}}\|\Delta^2 f\|_{L^2}+C(\|f\|_{W^{1,\infty}},\eps)\|\nabla^3 f\|_{L^2},
$$
and similarly $K_{3, 2}^3$ obeys the same bound.
Therefore, we obtain
$$
K_{3,2}\leq \frac{\eps}{2^{8}}\|\Delta^2 f\|_{L^2}+C(\|f\|_{W^{1,\infty}},\eps)\|\nabla^3 f\|_{L^2}.
$$
The term $K_{3, 3}$ is similar to $K_{2}$, hence
$$
K_{3,3}\leq \frac{\eps}{2^{9}}\|\Delta^2 f\|_{L^2}+C(\|f\|_{W^{1,\infty}},\eps)(\|\nabla^3 f\|_{L^2}^2+1).
$$
Adding the $K_{3,j}$ estimates gives
\bq\label{K3b}
K_{3}\leq \frac{\eps}{2^{7}}\|\Delta^2 f\|_{L^2}+C(\|f\|_{W^{1,\infty}},\eps)(\|\nabla^3 f\|^2_{L^2}+1).
\eq
{\bf Estimates for $K_4$.}  We expand the Laplacian in $K_4$ to obtain
\bq\label{K4split}
K_4\leq K_{4,1}+K_{4,2}+K_{4,3},
\eq
where
$$
K_{4,1}=\Big\|3\int \frac{(x',\delta_{x'}f(x))\wedge w(x-x')}{4\pi[|x'|^2+(\delta_{x'}f(x))^2]^{5/2}}\delta_{x'}f(x)\delta_{x'}\Delta f(x)\eta_1(x',\delta_{x'}f(x))\cdot N(x)dx'\Big\|_{L^2_x},
$$
$$
K_{4,2}=\Big\|3\int \frac{(x',\delta_{x'}f(x))\wedge w(x-x')}{4\pi[|x'|^2+(\delta_{x'}f(x))^2]^{5/2}}|\delta_{x'}\nabla f(x)|^2\eta_1(x',\delta_{x'}f(x))\cdot N(x)dx'\Big\|_{L^2_x},
$$
$$
K_{4,3}=\Big\|15\!\int \frac{(x',\delta_{x'}f(x))\wedge w(x-x')}{4\pi[|x'|^2+(\delta_{x'}f(x))^2]^{7/2}}(\delta_{x'}f(x))^2|\delta_{x'}\nabla f(x)|^2\eta_1(x',\delta_{x'}f(x))\cdot N(x)dx'\Big\|_{L^2_x}.
$$
The term $K_{4,1}$ is analogous to $K_{3}$, whence
\bq\label{K41b}
K_{4,1}\leq \frac{\eps}{2^{7}}\|\Delta^2 f\|_{L^2}+C(\|f\|_{W^{1,\infty}},\eps)(\|\nabla^3 f\|^2_{L^2}+1).
\eq
As for $K_{4,2}$ we split
\bq\label{K42split}
K_{4,2}\les K_{4,2}^1+K_{4,2}^2+K_{4,2}^3+K_{4,2}^4+K_{4,2}^5+K_{4,2}^6,
\eq
where
$$
K_{4,2}^1=\Big\|\int \frac{(x',\delta_{x'}f(x))\wedge w(x-x')}{4\pi[|x'|^2+(\delta_{x'}f(x))^2]^{5/2}}|\delta_{x'}\nabla f(x)|^2(\eta_1(x',\delta_{x'}f(x))-1)\cdot N(x)dx'\Big\|_{L^2_x},
$$
$$
K_{4,2}^2=\Big\|\int \frac{(x',\delta_{x'}f(x))\wedge w(x-x')}{4\pi[|x'|^2+(\delta_{x'}f(x))^2]^{5/2}}(\delta_{x'}\nabla f(x)\!\cdot\!(\delta_{x'}\nabla f(x)\!-\!\nabla^2f(x)\cdot x'))\cdot N(x)dx'\Big\|_{L^2_x},
$$
$$
K_{4,2}^3=\Big\|\int \frac{(x',\delta_{x'}f(x))\wedge w(x-x')}{4\pi[|x'|^2+(\delta_{x'}f(x))^2]^{5/2}}((\delta_{x'}\nabla f(x)\!-\!\nabla^2f(x)\cdot x')\cdot \nabla^2 f(x)\!\cdot\!x')\cdot N(x)dx'\Big\|_{L^2_x},
$$
$$
K_{4,2}^4=\Big\|\int \frac{(0,\delta_{x'}f(x)-\nabla f(x)\!\cdot\! x')\wedge w(x\!-\!x')}{4\pi[|x'|^2\!+\!(\delta_{x'}f(x))^2]^{5/2}}|\nabla^2f(x)\cdot x'|^2\!\cdot\! N(x)dx'\Big\|_{L^2_x},
$$
$$
K_{4,2}^5=\Big\|\int (x',\nabla f(x)\cdot x')\wedge w(x-x')B(x,x')|\nabla^2 f(x)\cdot x'|^2\cdot N(x)dx'\Big\|_{L^2_x}
$$
with $B$ given in \eqref{Bdef}, and
$$
K_{4,2}^6=\Big\|\int \frac{(x',\nabla f(x)\!\cdot\! x')\wedge w(x\!-\!x')}{4\pi[|x'|^2\!+\!(\nabla f(x)\!\cdot\! x')^2]^{5/2}}|\nabla^2f(x)\cdot x'|^2\!\cdot\! N(x)dx'\Big\|_{L^2_x}.
$$
The first term is easily controlled by
$$
K_{4,2}^1\leq C(\|f\|_{W^{1,\infty}})\|w \|_{L^2}\le C(\|f\|_{W^{1,\infty}})\|\nabla^3 f\|_{L^2}.
$$
 Letting $2<p<(\mez-\frac{1}{q})^{-1}$, it follows from  Morrey's and Minkowski's inequalities that
\bq
\label{K422K423b}
\begin{aligned}
	K_{4,2}^2\!+\!K_{4,2}^3\leq&\|\nabla f\|_{L^\infty}\|
	\nabla^2 f\|_{\dot{C}^{1-\frac2p}}\Big\|\int_0^1\!ds\!\int\frac{|w(x\!-\!x')|(
		|\nabla^2 f(x\!-\!sx')|\!+\!|\nabla^2 f(x)|)}{|x'|^{1+\frac2p}}dx'
	\Big\|_{L^2_x}\\
	& \le C\|\nabla f\|_{L^\infty}\|
	\nabla^3 f\|_{L^p}\| w\|_{L^r}\| \na^2f\|_{L^q} \int_0^1\!ds\!\int\frac{1}{|x'|^{1+\frac2p}}dx'\\
	&\le C(\| f\|_{W^{1, \infty}})\|
	\nabla^3 f\|_{L^p}\| \na^2f\|_{L^q}.
\end{aligned}
\eq
Now we apply \eqref{GNL2Linf} with $u=\na^2f$ and apply \eqref{GNnablaL2L2} with $u=\na^3f$ to have
\[
\|\nabla^3 f\|_{L^p}\| \na^2f\|_{L^q}\le C(\| f\|_{W^{1, \infty}})\|\nabla^4 f\|_{L^2}^{1-\frac{2}{p}}\| \nabla^3 f\|_{L^2}^{1-\frac{2}{q}+\frac{2}{p}}.
\]
Using Young's inequality and noting that $\frac{p}{2}(1-\frac{2}{q}+\frac{2}{p})=p(\mez-\frac{1}{q})+1<2$, we deduce
\[
\begin{aligned}
K_{4,2}^2\!+\!K_{4,2}^3&\le \frac{\eps}{2^{8}}\|\na^4 f\|_{L^2}+C(\|f\|_{W^{1,\infty}},\eps)\|\nabla^3 f\|_{L^2}^{\frac{p}{2}(1-\frac{2}{q}+\frac{2}{p})}\\
&\le  \frac{\eps}{2^{8}}\|\Delta^2 f\|_{L^2}+C(\|f\|_{W^{1,\infty}},\eps)(\|\nabla^3 f\|_{L^2}^2+1).
\end{aligned}
\]
Regarding  $K_{4, 2}^4$  we choose  $\frac{2r}{r-2}<p<\infty$ and use \eqref{GNL2Linf} to have
\begin{align}
\begin{split}\label{K424b}
K_{4,2}^4\leq& \|\nabla f\|_{L^\infty}\|\nabla f\|_{C^{1-\frac2p}}\sup_{x}\int\frac{|w(x-x')|}{|x'|^{1+\frac2p}}dx'\|\nabla^2f\|_{L^4}^2\\
&\le \|\nabla f\|^2_{L^\infty}\|\nabla^2f\|_{L^p}\|w\|_{L^r}\||x'|^{-(1+\frac2p)}\|_{L^{\frac r{r-1}}_{x'}}\|\nabla^3f\|_{L^2}\\
&\leq C(\|f\|_{W^{1,\infty}})\|\nabla^3f\|_{L^2}^{2-\frac2p}.
\end{split}
\end{align}
The bound
$$
|B(x,x')|\leq C\frac{|\delta_{x'}f(x)-\nabla f(x)\!\cdot\! x'|}{|x'|^6}
$$
allows us to treat $K_{4,2}^5$ similarly to $K_{4,2}^4$, hence
$$
K_{4,2}^5\leq C(\|f\|_{W^{1,\infty}})\|\nabla^3f\|_{L^2}^{2-\frac2p}.
$$
In $K_{4,2}^6$ we find a singular integral operator as before, so \eqref{bound:Sop} and \eqref{GNL2Linf}  imply
\bq\label{K426b}
K_{4,2}^6\leq C\|\nabla f\|_{L^\infty}\||\nabla^2f|^2\|_{L^q}\|w\|_{L^r}\leq C(\|f\|_{W^{1,\infty}})\|\nabla^3f\|_{L^2}^{2-\frac2q}.
\eq
Combining the above estimates for $K_{4, 2}^j$ leads to
\bq\label{K42b}
K_{4,2}\leq \frac{\eps}{2^{8}}\|\Delta^2 f\|_{L^2}+C(\|f\|_{W^{1,\infty}},\eps)(\|\nabla^3f\|^2_{L^2}+1).
\eq
The term  $K_{4,3}$ can be treated in the same way as $K_{4,2}$, hence we obtain
\bq\label{K4b}
 K_{4}\leq \frac{\eps}{2^{6}}\|\Delta^2 f\|_{L^2}+C(\|f\|_{W^{1,\infty}},\eps)(\|\nabla^3f\|^2_{L^2}+1).
\eq
{\bf Estimates for $K_5$.}  We bound $K_5$ by
\bq\label{K5split}
K_5\leq K_{5,1}+K_{5,2}+K_{5,3}+K_{5,4},
\eq
where
$$
K_{5,1}=\Big\|\int\frac{(x',\delta_{x'} f(x))\wedge \Delta w(x-x')}{4\pi[|x'|^2+(\delta_{x'}f(x))^2]^{3/2}}(\eta_1(x',\delta_{x'}f(x))-\eta_4(x'))\cdot N(x)dx'\Big\|_{L^2_x},
$$
$$
K_{5,2}=\Big\|\int\frac{(0,\delta_{x'} f(x)-\nabla f(x)\cdot x')\wedge \Delta w(x-x')}{4\pi[|x'|^2+(\delta_{x'}f(x))^2]^{3/2}}\eta_4(x')\cdot N(x)dx'\Big\|_{L^2_x},
$$
$$
K_{5,3}=\Big\|\int(x',\nabla f(x)\cdot x')\wedge \Delta w(x-x')A(x,x')\eta_4(x')\cdot N(x)dx'\Big\|_{L^2_x},
$$
$$
K_{5,4}=\Big\|\int\frac{(x',\nabla f(x)\cdot x')\wedge \Delta w(x-x')}{4\pi[|x'|^2+(\nabla f(x)\cdot x')^2]^{3/2}}\eta_4(x')\cdot N(x)dx'\Big\|_{L^2_x}.
$$
In $K_{5, 1}$ we write $\Delta w(x-x')=-\nabla_{x'}\cdot(\nabla w(x-x'))$ and integrate by parts. Since $\eta_1(x',\delta_{x'}f(x))-\eta_4(x')=0$ for $x'$ close to $0$,  using \eqref{nablawL2b} to bound $\| \na w\|_{L^2}$, we obtain
\begin{align*}
K_{5,1}\leq&\sum_{j=1}^2\Big\|\int\partial_{x'_j}\Big(\frac{(x',\delta_{x'} f(x))(\eta_1(x',\delta_{x'}f(x))\!-\!\eta_4(x'))}{4\pi[|x'|^2\!+\!(\delta_{x'}f(x))^2]^{3/2}}\Big)\wedge\partial_jw(x\!-\!x')\!\cdot\! N(x)dx'\Big\|_{L^2_x}\\
\leq&C(\|f\|_{W^{1,\infty}})\|\nabla w\|_{L^2}\leq C(\|f\|_{W^{1,\infty}})(\|\nabla^3f\|_{L^2}+1).
\end{align*}
 In the term $K_{5,2}$ we perform a similar integration by part to get
\bq\label{K52split}
K_{5,2}\leq K_{5,2}^1+K_{5,2}^2+K_{5,2}^3,
\eq
where
$$
K_{5,2}^1=\sum_{j=1}^2\Big\|\int\frac{(0,\delta_{x'} f(x)\!-\!\nabla f(x)\!\cdot\! x')\wedge\partial_jw(x\!-\!x')}{4\pi[|x'|^2\!+\!(\delta_{x'}f(x))^2]^{3/2}}\partial_j\eta_4(x')\!\cdot\! N(x)dx'\Big\|_{L^2_x},
$$
$$
K_{5,2}^2=\sum_{j=1}^2\Big\|\int\frac{(0,\delta_{x'}\partial_j f(x))\wedge\partial_jw(x\!-\!x')}{4\pi[|x'|^2\!+\!(\delta_{x'}f(x))^2]^{3/2}}\eta_4(x')\!\cdot\! N(x)dx'\Big\|_{L^2_x},
$$
$$
K_{5,2}^3\!=\!3\!\sum_{j=1}^2\!\Big\|\!\int\!\frac{(0,\delta_{x'} f(x)\!-\!\nabla f(x)\!\cdot\! x')\wedge\partial_jw(x\!-\!x')}{4\pi[|x'|^2\!+\!(\delta_{x'}f(x))^2]^{5/2}}(x'_j\!+\!\delta_{x'}f(x)\partial_jf(x\!-\!x'))\eta_4(x')\!\cdot\! N(x)dx'\Big\|_{L^2_x}.
$$
By virtue of \eqref{nablawL2b} and the support of $\na \eta_4$, we have
$$
K_{5,2}^1\leq C(\|f\|_{W^{1,\infty}})\|\na w\|_{L^2}\le C(\|f\|_{W^{1,\infty}})(\|\nabla^3f\|_{L^2}+1).
$$
In order to deal with $K_{5,2}^2$ we split further
$$
K_{5,2}^2\leq K_{5,2}^{2,1}+K_{5,2}^{2,2}+K_{5,2}^{2,3},
$$
where
$$
K_{5,2}^{2,1}=\sum_{j=1}^2\Big\|\int
(0,\delta_{x'}\partial_j f(x))\wedge\partial_jw(x\!-\!x')A(x,x')\eta_4(x')\!\cdot\! N(x)dx'\Big\|_{L^2_x},
$$
$$
K_{5,2}^{2,2}=\sum_{j=1}^2\Big\|\int\frac{(0,\delta_{x'}\partial_j f(x)\!-\!\nabla \partial_jf(x)\!\cdot\!x'1_{B_\zeta(x')})\wedge\partial_jw(x\!-\!x')}{4\pi[|x'|^2\!+\!(\nabla f(x)\!\cdot\!x')^2]^{3/2}}\eta_4(x')\!\cdot\! N(x)dx'\Big\|_{L^2_x},
$$
$$
K_{5,2}^{2,3}=\sum_{j=1}^2\Big\|\int\frac{(0,\nabla \partial_jf(x)\!\cdot\!x'1_{B_\zeta}(x'))\wedge\partial_jw(x\!-\!x')}{4\pi[|x'|^2\!+\!(\nabla f(x)\!\cdot\!x')^2]^{3/2}}\eta_4(x')\!\cdot\! N(x)dx'\Big\|_{L^2_x},
$$
where $1_{B_\zeta}$ denotes the characteristic function of the ball $B_\zeta$ and $0<\zeta$ will be chosen later. The bound \eqref{Abound} implies
\begin{align*}
	K_{5,2}^{2,1}\leq&C\|\nabla f\|_{L^\infty}\Big\|\int_0^1\!\!d\sigma\int_0^1\!\!ds\int_{|x'|<\delta}|\nabla^2f(x\!-\!sx')||\nabla^2f(x\!-\!\sigma x')||\nabla w(x\!-\!x')|\frac{dx'}{|x'|}\Big\|_{L^2_x}
	\\
	&+C(\|f\|_{W^{1,\infty}})\Big\|\int_{|x'|>\delta}|\nabla w(x-x')|\frac{dx'}{|x'|^3}\Big\|_{L^2_x}
	\\
	\leq&\delta C(\|f\|_{W^{1,\infty}})\|\nabla^2f\|_{L^\infty}^2\|\nabla w\|_{L^2}+\frac1\delta C(\|f\|_{W^{1,\infty}})\|\nabla w\|_{L^2}.
\end{align*}
Then in view of the Gagliardo-Nirenberg inequality
\bq\label{GNnabla2linfty}
\|\nabla u\|_{L^\infty}\leq C\|\nabla^3u\|_{L^2}^{\frac12}\|u\|_{L^\infty}^{\frac12},
\eq
and \eqref{nablawL2b}, we deduce
\[
K_{5,2}^{2,1}\leq \delta C(\|f\|_{W^{1,\infty}})\|\nabla^4f\|_{L^2}(\|\nabla^3 f\|_{L^2}+1)+\frac1\delta C(\|f\|_{W^{1,\infty}})(\|\nabla^3 f\|_{L^2}+1).
\]
Choosing $\delta=\eps(2^9C(\|f\|_{W^{1,\infty}})(\|\nabla^3f\|_{L^2}+1))^{-1}$ yields
$$
K_{5,2}^{2,1}\leq \frac{\eps}{2^9}\|\Delta^2f\|_{L^2}+C(\|f\|_{W^{1,\infty}},\eps)(\|\nabla^3 f\|_{L^2}^2+1).
$$
For $K_{5,2}^{2,2}$ we use Morrey's inequality, \eqref{GNnablaL2L2} and \eqref{nablawL2b} to have
\begin{align*}
	K_{5,2}^{2,2}\leq&C\|\nabla f\|_{L^\infty}\|\nabla^2f\|_{\dot{C}^{\frac12}}\Big\|\int_{|x'|<\zeta}|\nabla w(x-x')|\frac{dx'}{|x'|^{\frac32}}\Big\|_{L^2_x}
	\\
	&+C(\|f\|_{W^{1,\infty}})\Big\|\int_{|x'|>\zeta}|\nabla w(x-x')|\frac{dx'}{|x'|^3}\Big\|_{L^2_x}
	\\
	\leq&\zeta^{\frac12}C(\|f\|_{W^{1,\infty}})\|\nabla^3f\|_{L^4}\|\nabla w\|_{L^2}+\frac1\zeta C(\|f\|_{W^{1,\infty}})\|\nabla w\|_{L^2}\\
	\leq&\zeta^{\frac12}C(\|f\|_{W^{1,\infty}})\|\Delta^2 f\|_{L^2}^\mez\| \na^3f\|_{L^2}^\mez\|\nabla w\|_{L^2}+\frac1\zeta C(\|f\|_{W^{1,\infty}})\|\nabla w\|_{L^2}\\	
	\leq&\zeta^{\frac12}C(\|f\|_{W^{1,\infty}})\|\Delta^2 f\|_{L^2}^\mez\| \na^3f\|_{L^2}^\mez(\|\nabla^3f\|_{L^2}+1)+\frac1\zeta C(\|f\|_{W^{1,\infty}})(\|\nabla^3f\|_{L^2}+1).
\end{align*}
Choosing $\zeta=(\|\nabla^3f\|_{L^2}+1)^{-1}$ and using Young's inequality, we find
$$
K_{5,2}^{2,2}\leq \frac{\eps}{2^9}\|\Delta^2f\|_{L^2}+C(\|f\|_{W^{1,\infty}},\eps)(\|\nabla^3 f\|_{L^2}^2+1).
$$
Using the method of rotations and the maximal Hilbert transform (see Chapter 3, Theorem 3.4 \cite{Duo}), we can bound the singular integral operator in $K_{5,2}^{2,3}$ independent of the cutoff scale $\zeta$,
$$
K_{5,2}^{2,3}\leq C(\|f\|_{W^{1,\infty}})\|\nabla^2f\|_{L^\infty}\|\nabla w\|_{L^2}\leq \frac{\eps}{2^9}\|\Delta^2f\|_{L^2}+C(\|f\|_{W^{1,\infty}},\eps)(\|\nabla^3 f\|_{L^2}^2+1),
$$
where we have used \eqref{GNnabla2linfty} and Young's inequality. Adding the above estimates for $K_{5, 2}^{2, j}$ gives
\bq\label{K522b}
K_{5,2}^{2}\leq \frac{\eps}{2^7}\|\Delta^2f\|_{L^2}+C(\|f\|_{W^{1,\infty}},\eps)(\|\nabla^3 f\|_{L^2}^2+1).
\eq
The term $K_{5,2}^3$ can be handled analogously to $K_{5,2}^2$. It follows that
$$
K_{5,2}\leq \frac{\eps}{2^6}\|\Delta^2f\|_{L^2}+C(\|f\|_{W^{1,\infty}},\eps)(\|\nabla^3 f\|_{L^2}^2+1).
$$
A similar approach for $K_{5,3}$ gives
$$
K_{5,3}\leq \frac{\eps}{2^6}\|\Delta^2f\|_{L^2}+C(\|f\|_{W^{1,\infty}},\eps)(\|\nabla^3 f\|_{L^2}^2+1).
$$
Next, we use  \eqref{bound:Sop} to estimate the singular integral operator in $K_{5, 4}$, so that
\[
K_{5,4}\leq C(\|f\|_{W^{1,\infty}})\|\Delta w\|_{L^2}.
\]
Then by virtue of \eqref{NablaThetaLr}, \eqref{GNL2Linf}, \eqref{GNnabla2linfty} and \eqref{DeltathetaL2b}, we obtain
 \begin{align*}
K_{5,4} &\leq C(\|f\|_{W^{1,\infty}})\left(\|\nabla^3\theta\|_{L^2}+\|\nabla^2f\|_{L^4}\|\nabla^2\theta\|_{L^4}+\|\nabla\theta\|_{L^r}\|\nabla^3f\|_{L^q}\right)\\
\leq&C(\|f\|_{W^{1,\infty}})\left(\|\nabla^3\theta\|_{L^2}+\|\nabla^3f\|^{\frac12}_{L^2}\|\nabla^3\theta\|^{\frac12}_{L^2}\|\Delta\theta\|_{L^2}^{\frac12}+\|\Delta^2f\|_{L^2}^{1-\frac2q}\|\nabla^3f\|_{L^2}^{\frac2q}\right)\\
\leq&C(\|f\|_{W^{1,\infty}})\|\nabla^3\theta\|_{L^2}+\frac{\eps}{2^6}\|\Delta^2f\|_{L^2}+C(\|f\|_{W^{1,\infty}},\eps)(\|\nabla^3f\|_{L^2}^2+1)
\end{align*}
 To finish with this term we claim that
\bq\label{claim:tt:H3}
\|\nabla^3\theta\|_{L^2}\leq \delta\|\Delta^2f\|_{L^2}+ C(\|f\|_{W^{1,\infty}},\delta)(\|\nabla^3 f\|^2_{L^2}+1)\quad\forall \delta>0,
\eq
whose proof is postponed to Lemma \ref{nabla3thetaL2} below. Consequently,
$$
K_{5,4}\leq \frac{\eps}{2^5}\|\Delta^2f\|_{L^2}+C(\|f\|_{W^{1,\infty}},\eps)(\|\nabla^3 f\|_{L^2}^2+1),
$$
so that gathering the estimates for $K_{5,j}$ gives
\bq\label{K5b}
K_{5}\leq \frac{\eps}{2^4}\|\Delta^2f\|_{L^2}+C(\|f\|_{W^{1,\infty}},\eps)(\|\nabla^3 f\|_{L^2}^2+1).
\eq
{\bf Estimates for $K_6$, $K_7$ and $K_8$.}  As for $K_6$ we use \eqref{nablawL2b} when needed to obtain
\bq\label{K6b}
K_6\leq C(\|f\|_{W^{1,\infty}})(\|\nabla^3 f\|_{L^2}+1).
\eq
Regarding $K_7$  we use the  splitting
$$
K_7=K_{7,1}+K_{7,2}+K_{7,3}+K_{7,4},
$$
where
$$
K_{7,1}=\sum_{j=1}^2\Big\|\int\frac{(0,\delta_{x'}\partial_jf(x))\wedge w(x-x')}{2\pi[|x'|^2+(\delta_{x'}f(x))^2]^{3/2}}(\eta_1(x',\delta_{x'}f(x))-1)\cdot \partial_jN(x)dx'\Big\|_{L^2_x},
$$
$$
K_{7,2}=\sum_{j=1}^2\Big\|\int\frac{(0,\delta_{x'}\partial_jf(x)-\nabla\partial_jf(x)\cdot x')\wedge w(x-x')}{2\pi[|x'|^2+(\delta_{x'}f(x))^2]^{3/2}}\cdot \partial_jN(x)dx'\Big\|_{L^2_x},
$$
$$
K_{7,3}=2\sum_{j=1}^2\Big\|\int(0,\nabla\partial_jf(x)\cdot x')\wedge w(x-x')A(x,x')\cdot \partial_jN(x)dx'\Big\|_{L^2_x},
$$
$$
K_{7,4}=\sum_{j=1}^2\Big\|\int\frac{(0,\nabla\partial_jf(x)\cdot x')\wedge w(x-x')}{2\pi[|x'|^2+(\nabla f(x)\cdot x')^2]^{3/2}}\cdot \partial_jN(x)dx'\Big\|_{L^2_x}.
$$
It is easy to see that
$$
K_{7,1}\leq C(\|f\|_{W^{1,\infty}})\|\nabla^3 f\|_{L^2}.
$$
As for $K_{7,2}$ we use the argument in  (\ref{K42split}, \ref{K422K423b}) for $K_{4,2}^2$ to obtain
$$
K_{7,2}\leq \frac{\eps}{2^9}\|\Delta^2f\|_{L^2}+C(\|f\|_{W^{1,\infty}},\eps)(\|\nabla^3 f\|_{L^2}^2+1).
$$
To control $K_{7,3}$ we use the first inequality in \eqref{Abound} and an approach similar to  (\ref{K42split}, \ref{K424b}) for $K_{4,2}^4$. It follows that
$$
K_{7,3}\leq C(\|f\|_{W^{1,\infty}})(\|\nabla^3 f\|_{L^2}^2+1).
$$
Comparing with $K_{4,2}^6$ (\ref{K42split},\ref{K426b}) we obtain
$$
K_{7,4}\leq C(\|f\|_{W^{1,\infty}})(\|\nabla^3 f\|_{L^2}^2+1).
$$
We have proven that
\bq\label{K7b}
K_{7}\leq \frac{\eps}{2^9}\|\Delta^2f\|_{L^2}+C(\|f\|_{W^{1,\infty}},\eps)(\|\nabla^3 f\|_{L^2}^2+1).
\eq
The term $K_8$ \eqref{Ksplit} can be  handled similarly to $K_7$, so that
\bq\label{K8b}
K_8\leq \frac{\eps}{2^9}\|\Delta^2f\|_{L^2}+C(\|f\|_{W^{1,\infty}},\eps)(\|\nabla^3 f\|_{L^2}^2+1).
\eq
{\bf Estimates for $K_9$.}  We split  $K_9$ \eqref{Ksplit} as
$$
K_9\leq K_{9,1}+K_{9,2}+K_{9,3}+K_{9,4}
$$
where
$$
K_{9,1}=\sum_{j=1}^2\Big\|\int\frac{(x',\delta_{x'}f(x))\wedge \partial_jw(x-x')}{2\pi[|x'|^2+(\delta_{x'}f(x))^2]^{3/2}}(\eta_1(x',\delta_{x'}f(x))-\eta_4(x'))\cdot \partial_jN(x)dx'\Big\|_{L^2_x},
$$
$$
K_{9,2}=\sum_{j=1}^2\Big\|\int\frac{(0,\delta_{x'}f(x)-\nabla f(x)\cdot x')\wedge \partial_jw(x-x')}{2\pi[|x'|^2+(\delta_{x'}f(x))^2]^{3/2}}\eta_4(x')\cdot \partial_jN(x)dx'\Big\|_{L^2_x},
$$
$$
K_{9,3}=\sum_{j=1}^2\Big\|\int(x',\nabla f(x)\cdot x')\wedge \partial_jw(x-x')A(x,x')\eta_4(x')\cdot \partial_jN(x)dx'\Big\|_{L^2_x}.
$$
and
$$
K_{9,4}=\sum_{j=1}^2\Big\|\int\frac{(0,\nabla f(x)\cdot x')\wedge \partial_jw(x-x')}{2\pi[|x'|^2+(\nabla f(x)\cdot x')^2]^{3/2}}\eta_4(x')\cdot \partial_jN(x)dx'\Big\|_{L^2_x}.
$$
Using \eqref{nablawL2b} we bound the non-singular term $K_{9,1}$ by
$$
K_{9,1}\leq C(\|f\|_{W^{1,\infty}})\|\nabla w\|_{L^2}\|\nabla N\|_{L^2}\leq  C(\|f\|_{W^{1,\infty}})(\|\nabla^3 f\|_{L^2}^2+1).
$$
In $K_{9,2}$ we integrate by parts in $\p_j w(x-x')=-\p_{x_j'}w(x-x')$ to obtain
$$
K_{9,2}\leq K_{9,2}^1+K_{9,2}^2+K_{9,2}^3,
$$
where
$$
K_{9,2}^1=\sum_{j=1}^2\Big\|\int\frac{(0,\delta_{x'}f(x)-\nabla f(x)\cdot x')\wedge w(x-x')}{2\pi[|x'|^2+(\delta_{x'}f(x))^2]^{3/2}}\partial_j\eta_4(x')\cdot \partial_jN(x)dx'\Big\|_{L^2_x},
$$
$$
K_{9,2}^2=\sum_{j=1}^2\Big\|\int\frac{(0,\delta_{x'}\partial_jf(x))\wedge w(x-x')}{2\pi[|x'|^2+(\delta_{x'}f(x))^2]^{3/2}}\eta_4(x)\cdot \partial_jN(x)dx'\Big\|_{L^2_x},
$$
$$
K_{9,2}^3=3\sum_{j=1}^2\Big\|\int\!\frac{(0,\delta_{x'}f(x)\!-\!\nabla f(x)\!\cdot\! x')\!\wedge\! w(x\!-\!x')}{2\pi[|x'|^2+(\delta_{x'}f(x))^2]^{5/2}}(x_j\!+\!\delta_{x'}f(x)\partial_jf(x\!-\!x'))\eta_4(x')\!\cdot\! \partial_jN(x)dx'\Big\|_{L^2_x}.
$$
The first term is easily bounded by
$$
K_{9,1}^1\leq C(\|f\|_{W^{1,\infty}})\|\nabla^2 f\|_{L^2}\le C(\|f\|_{W^{1,\infty}})\|\nabla^3 f\|_{L^2}.
$$
We deal with $K_{9,2}^2$ similarly to $K_7$ in (\ref{Ksplit}, \ref{K7b}) to obtain
$$
K_{9,2}^2\leq \frac{\eps}{2^9}\|\Delta^2f\|_{L^2}+C(\|f\|_{W^{1,\infty}},\eps)(\|\nabla^3 f\|_{L^2}^2+1).
$$
The term  $K_{9,2}^3$ can be handled using a similar approach. Then it follows that
$$
K_{9,2}\leq \frac{\eps}{2^8}\|\Delta^2f\|_{L^2}+C(\|f\|_{W^{1,\infty}},\eps)(\|\nabla^3 f\|_{L^2}^2+1).
$$
The term $K_{9,3}$ can be treated similarly, yielding
$$
K_{9,3}\leq \frac{\eps}{2^8}\|\Delta^2f\|_{L^2}+C(\|f\|_{W^{1,\infty}},\eps)(\|\nabla^3 f\|_{L^2}^2+1).
$$
The inequalities \eqref{GNnabla2linfty}, \eqref{DeltathetaL2b}, and \eqref{bound:Sop} imply
$$
K_{9,4}\leq C\|\nabla^2f\|_{L^\infty}\|\nabla w\|_{L^2}\leq \frac{\eps}{2^7}\|\Delta^2f\|_{L^2}+C(\|f\|_{W^{1,\infty}},\eps)(\|\nabla^3 f\|_{L^2}^2+1).
$$
Combining the above estimates for $K_{9,j}$ leads to
\bq\label{K9b}
K_{9}\leq  \frac{\eps}{2^6}\|\Delta^2f\|_{L^2}+C(\|f\|_{W^{1,\infty}},\eps)(\|\nabla^3 f\|_{L^2}^2+1).
\eq
{\bf Estimates for $K_{10}$, $K_{11}$ and $K_{12}$.} We handle $K_{10}$ \eqref{Ksplit} as in (\ref{K4split}, \ref{K42b}) for $K_{4,2}$ to find
\bq\label{K10b}
K_{10}\leq \frac{\eps}{2^6}\|\Delta^2f\|_{L^2}+C(\|f\|_{W^{1,\infty}},\eps)(\|\nabla^3 f\|_{L^2}^2+1).
\eq
The last two terms $K_{11}$ and $K_{12}$  can be estimated analogously to $K_{5,2}^2$ in (\ref{K52split}, \ref{K522b}), giving
\bq\label{K12b}
K_{11}+K_{12}\leq \frac{\eps}{2^7}\|\Delta^2f\|_{L^2}+C(\|f\|_{W^{1,\infty}},\eps)(\|\nabla^3 f\|_{L^2}^2+1).
\eq
At this point we gather the estimates (\ref{K1b}, \ref{K2b}, \ref{K3b}, \ref{K4b}, \ref{K5b}, \ref{K6b}, \ref{K7b}, \ref{K8b}) to conclude that
\bq\label{est:K}
K\leq \frac{\eps}{2^2}\|\Delta^2f\|_{L^2}+C(\|f\|_{W^{1,\infty}},\eps)(\|\nabla^3 f\|_{L^2}^2+1).
\eq
Combining the estimates \eqref{est:R3} and \eqref{est:K} for $R_3$ and $K$, we deduce
\bq
\frac1{2\eps}\|\Delta(\ka G(f)f)\|_{L^2}^2\le  \frac{\eps}{2^4}\|\Delta^2f\|^2_{L^2}+C(\|f\|_{W^{1,\infty}},\eps)(\|\nabla^3 f\|_{L^2}^4+1).
\eq
Then inserting this into \eqref{evolH3} yields
\bq
\frac{d}{dt} \| \na^3 f\|_{L^2}^2\le -\eps \|\Delta^2f\|^2_{L^2}+C(\|f\|_{W^{1,\infty}},\eps)(\|\nabla^3 f\|_{L^2}^4+1).
\eq
By Gr\"onwall's lemma and the $L^2_t H^3_x$ estimate from \eqref{DeltafL2b}, we arrive at the following $H^3$ estimate
\begin{align}
\begin{split}\label{H3control}
\|\nabla^3 f(t)\|^2_{L^2}+\eps\int_0^t\|\Delta^2 f(\tau)\|^2_{L^2}d\tau \le (\|\nabla^3 f_0\|^2_{L^2}+1)\exp\Big(&C(\|f_0\|_{W^{1,\infty}},\eps)\big(\|f_0\|^2_{H^3}+t\big)\Big).
\end{split}
\end{align}
Finally, we prove the claim \eqref{claim:tt:H3} in the following lemma.
\begin{lemm}\label{nabla3thetaL2}
Consider $\theta$ and $w$ given by \eqref{thetaeq} and  \eqref{Fw} respectively with $f\in H^4$. Then the following estimates hold for any $\delta >0$,
	\bq\label{nabla3thetaL2b}
	\|\nabla^3\theta\|_{L^2}\leq \delta\|\Delta^2f\|_{L^2} +C(\|f\|_{W^{1,\infty}},\delta)(\|\nabla^3 f\|^2_{L^2}+1),
	\eq
	\bq\label{DeltawL2b}
	\|\Delta w\|_{L^2}\leq \delta\|\Delta^2f\|_{L^2} +C(\|f\|_{W^{1,\infty}},\delta)(\|\nabla^3 f\|^2_{L^2}+1).
	\eq
\end{lemm}
\begin{proof}
In view of the formula \eqref{Fw} for $w$, the interpolation inequalities \eqref{GNL2Linf} and \eqref{GNnablaL2L2}, and the $W^{1, r}$ estimate \eqref{NablaThetaLr} for $\tt$, we have
\begin{align*}
	\|\Delta w\|_{L^2}\leq& \|f\|_{W^{1,\infty}}\|\nabla^3\theta\|_{L^2}+2\|\nabla^2f\|_{L^4}\|\nabla^2\theta\|_{L^4}+\|\nabla\theta\|_{L^r}\|\nabla^3f\|_{L^q})\\
	\leq& C(\|f\|_{W^{1,\infty}})(\|\nabla^3\theta\|_{L^2}+\|\nabla^3f\|^{\frac12}_{L^2}\|\nabla^3\theta\|^{\frac12}_{L^2}\|\Delta\theta\|_{L^2}^{\frac12}+\|\Delta^2f\|_{L^2}^{1-\frac2q}\|\nabla^3f\|_{L^2}^{\frac2q})\\
	\leq& C(\|f\|_{W^{1,\infty}})\|\nabla^3\theta\|_{L^2}+\frac{\delta}{2}\|\Delta^2f\|_{L^2}+C(\|f\|_{W^{1,\infty}},\delta)\|\nabla^3f\|_{L^2}(\|\nabla^2\tt\|_{L^2}+1).
\end{align*}
Then invoking the estimate \eqref{DeltathetaL2b}  for $\|\na^2 \tt\|_{L^2}$ we find that  \eqref{nabla3thetaL2b} implies \eqref{DeltawL2b}. The remainder of this proof is devoted to \eqref{nabla3thetaL2b} and  we note that it suffices to estimate $\p_j^3\tt$. We shall only consider $j=1$ as the case $j=2$ can be treated similarly. To this end we take  $\partial^2_1$ of \eqref{d1theta} and obtain
$$
\|\partial_1^3\theta\|_{L^2}\leq \Big\|\partial^2_{x_1}\Big(\int\nabla_{x,z} \Gamma(x-x',f(x)-f(x'))\wedge w(x')dx'\cdot T_1(x)\Big)\Big\|_{L^2_x}+\ka\|\partial_1^3f\|_{L^2}.
$$
To control the first term on the right-hand side, we split
$$
\Big\|\partial^2_{x_1}\Big(\int\nabla_{x,z} \Gamma(x-x',f(x)-f(x'))\wedge w(x')dx'\cdot T_1(x)\Big)\Big\|_{L^2_x}=L+R_4,
$$
where
$$
L=\Big\|\partial^2_{x_1}\Big(\int\frac{(x-x',f(x)-f(x'))\wedge w(x')}{4\pi[|x-x'|^2+(f(x)-f(x'))^2]^{3/2}}\eta_1(x-x',f(x)-f(x'))dx'\cdot T_1(x)\Big)\Big\|_{L^2_x}
$$
and  $R_4$ gathers non-singular integrals involving  $\eta_3$, $\nabla^j\eta_3$, $\eta_2$, $\nabla^j\eta_2$ and $\nabla^j\eta_1$ with $j\in \{1,2,3\}$. With the properties of the $\eta_j$ functions, the integrals to estimate are not singular, so that it can be controlled as follows
\bq\label{R4b}
R_4\leq C(\|f\|_{W^{1,\infty}})\|w\|_{L^2}\|\nabla^3 f\|_{L^2}\leq C(\|f\|_{W^{1,\infty}})\|\nabla^3 f\|_{L^2}.
\eq
Using the formula \eqref{Fw} for $w$ to calculate the triple scalar product in $L$,  we find
\bq\label{Lsplit}
L\leq L_1+L_2+L_3
\eq
where
$$
L_1=\Big\|\partial^2_{x_1}\Big(\int\frac{(x_2-x_2')(\partial_1f(x)-\partial_1f(x'))\partial_2\theta(x')}{4\pi[|x-x'|^2+(f(x)-f(x'))^2]^{3/2}}\eta_1(x-x',f(x)-f(x'))dx'\Big)\Big\|_{L^2_x},
$$
$$
L_2=\Big\|\partial^2_{x_1}\Big(\int\frac{(x_2-x_2')(\partial_2f(x)-\partial_2f(x'))\partial_1\theta(x')}{4\pi[|x-x'|^2+(f(x)-f(x'))^2]^{3/2}}\eta_1(x-x',f(x)-f(x'))dx'\Big)\Big\|_{L^2_x},
$$
$$
L_3=\Big\|\partial^2_{x_1}\Big(\int\frac{(f(x)\!-\!f(x')\!-\!\nabla f(x)\!\cdot\!(x\!-\!x'))\partial_1\theta(x')}{4\pi[|x-x'|^2+(f(x)-f(x'))^2]^{3/2}}\eta_1(x\!-\!x',f(x)\!-\!f(x'))dx'\Big)\Big\|_{L^2_x}.
$$
In $L_j$ we take  one $\partial_{x_1}$ and make the change of  variables $x'\mapsto x-x'$ to obtain
\bq\label{L1split}
L_1\leq L_{1,1}+L_{1,2}+L_{1,3}
\eq
where
$$
L_{1,1}=\Big\|\partial_{x_1}\Big(\int\frac{x_2'\delta_{x'}\partial_1f(x)\partial_2\theta(x-x')}{4\pi[|x'|^2+(\delta_{x'}f(x))^2]^{3/2}}\nabla\eta_1(x',\delta_{x'}f(x))\cdot T_1(x)dx'\Big)\Big\|_{L^2_x},
$$
$$
L_{1,2}=\Big\|\partial_{x_1}\Big(\partial_1^2f(x)\int\frac{x_2'\partial_2\theta(x-x')\eta_1(x',\delta_{x'}f(x))}{4\pi[|x'|^2+(\delta_{x'}f(x))^2]^{3/2}}dx'\Big)\Big\|_{L^2_x},
$$
$$
L_{1,3}=3\Big\|\partial_{x_1}\Big(\int\frac{x_2'\delta_{x'}\partial_1f(x)\partial_2\theta(x-x')}{4\pi[|x'|^2+(\delta_{x'}f(x))^2]^{5/2}}(x'_1+\delta_{x'}f(x)\partial_1f(x))\eta_1(x',\delta_{x'}f(x))dx'\Big)\Big\|_{L^2_x}.
$$
In view of \eqref{NablaThetaLr}, \eqref{DeltathetaL2b} and the support of $\na\eta_1$, it is easy to bound
$$
L_{1,1}\leq C(\|f\|_{W^{1,\infty}})(\|\nabla^3 f\|_{L^2}+1).
$$
Taking $\p_{x_1}$ in $L_{1,2}$ yields
$$
L_{1,2}\leq L_{1,2}^1+L_{1,2}^2+L_{1,2}^3+L_{1,2}^4,
$$
where
$$
L_{1,2}^1=\Big\|\partial_1^2f(x)\int\frac{x_2'\partial_2\theta(x-x')\partial_3\eta_1(x',\delta_{x'}f(x))}{4\pi[|x'|^2+(\delta_{x'}f(x))^2]^{3/2}}\delta_{x'}\partial_1f(x)dx'\Big\|_{L^2_x},
$$
$$
L_{1,2}^2=\Big\|\partial_1^3f(x)\int\frac{x_2'\partial_2\theta(x-x')\eta_1(x',\delta_{x'}f(x))}{4\pi[|x'|^2+(\delta_{x'}f(x))^2]^{3/2}}dx'\Big\|_{L^2_x},
$$
$$
L_{1,2}^3=3\Big\|\partial_1^2f(x)\int\frac{x_2'\partial_2\theta(x-x')\eta_1(x',\delta_{x'}f(x))}{4\pi[|x'|^2+(\delta_{x'}f(x))^2]^{5/2}}\delta_{x'}f(x)\delta_{x'}\partial_1f(x)dx'\Big\|_{L^2_x},
$$
$$
L_{1,2}^4=\Big\|\partial_1^2f(x)\int\frac{x_2'\partial^2_{1,2}\theta(x-x')\eta_1(x',\delta_{x'}f(x))}{4\pi[|x'|^2+(\delta_{x'}f(x))^2]^{3/2}}dx'\Big)\Big\|_{L^2_x}.
$$
 Thanks to the support of $\na \eta_1$, $L_{1, 2}^1$ is controlled by
$$
L_{1,2}^1\leq C(\|f\|_{W^{1,\infty}})\| \na^2 f\|_{L^2}\| \na \tt\|_{L^2}\le C(\|f\|_{W^{1,\infty}})\|\nabla^3 f\|_{L^2}.
$$
The terms $L_{1, 2}^2$, $L_{1, 2}^3$ and $L_{1, 2}^4$ can be treated similarly to $K_2$, $K_8$ and $K_9$ respectively, giving
$$
L_{1,2}\leq \frac{\delta}{2^4}\|\Delta^2f\|_{L^2}+C(\|f\|_{W^{1,\infty}},\delta)(\|\nabla^3 f\|_{L^2}^2+1).
$$
Next we expand $L_{1,3}$ to find
$$
L_{1,3}\leq L_{1,3}^1+...+L_{1,3}^6$$
where
$$
L_{1,3}^1\!=\!3\Big\|\int\!\frac{x_2'\delta_{x'}\partial_1f(x)\partial_2\theta(x\!-\!x')}{4\pi[|x'|^2\!+\!(\delta_{x'}f(x))^2]^{5/2}}(x'_1\!+\!\delta_{x'}f(x)\partial_1f(x))\partial_3\eta_1(x',\delta_{x'}f(x))\delta_{x'}f(x)\partial_1f(x)dx'\Big\|_{L^2_x},
$$
$$
L_{1,3}^2=3\Big\|\int\frac{x_2'\delta_{x'}\partial_1f(x)\partial_2\theta(x-x')}{4\pi[|x'|^2+(\delta_{x'}f(x))^2]^{5/2}}\delta_{x'}f(x)\partial^2_1f(x)\eta_1(x',\delta_{x'}f(x))dx'\Big\|_{L^2_x},
$$
$$
L_{1,3}^3=3\Big\|\int\frac{x_2'\delta_{x'}\partial_1f(x)\partial_2\theta(x-x')}{4\pi[|x'|^2+(\delta_{x'}f(x))^2]^{5/2}}\delta_{x'}\partial_1f(x)\partial_1f(x)\eta_1(x',\delta_{x'}f(x))dx'\Big\|_{L^2_x},
$$
$$
L_{1,3}^4=3\Big\|\int\frac{x_2'\delta_{x'}\partial^2_1f(x)\partial_2\theta(x-x')}{4\pi[|x'|^2+(\delta_{x'}f(x))^2]^{5/2}}(x'_1+\delta_{x'}f(x)\partial_1f(x))\eta_1(x',\delta_{x'}f(x))dx'\Big\|_{L^2_x},
$$
$$
L_{1,3}^5=15\Big\|\int\frac{x_2'\delta_{x'}\partial_1f(x)\partial_2\theta(x\!-\!x')}{4\pi[|x'|^2\!+\!(\delta_{x'}f(x))^2]^{7/2}}(x'_1\!+\!\delta_{x'}f(x)\partial_1f(x))\delta_{x'}f(x)\delta_{x'}\partial_1f(x)\eta_1(x',\delta_{x'}f(x))dx'\Big\|_{L^2_x},
$$
and
$$
L_{1,3}^6=3\Big\|\int\frac{x_2'\delta_{x'}\partial_1f(x)\partial_{1,2}^2\theta(x-x')}{4\pi[|x'|^2+(\delta_{x'}f(x))^2]^{5/2}}(x'_1+\delta_{x'}f(x)\partial_1f(x))\eta_1(x',\delta_{x'}f(x))dx'\Big\|_{L^2_x}.
$$
It is readily seen that
$$
L_{1,3}^1\leq C(\|f\|_{W^{1,\infty}})\| \na \tt\|_{L^2}\le C(\|f\|_{W^{1,\infty}})\| \na^3f\|_{L^2}.
$$
The term $L_{1, 3}^j$ with $j\in \{2, \dots, 6\}$ can be controlled analogously to the term $K_8$, $K_{10}$, $K_{4,1}$, $L_{1, 3}^3$ and $K_{12}$ respectively, giving
$$
L_{1,3}\leq \frac{\delta}{2^4}\|\Delta^2f\|_{L^2}+C(\|f\|_{W^{1,\infty}},\delta)(\|\nabla^3 f\|_{L^2}^2+1).
$$
Combining the above estimates for  $L_{1, j}$ yields
$$
L_{1}\leq \frac{\delta}{2^3}\|\Delta^2f\|_{L^2}+C(\|f\|_{W^{1,\infty}},\delta)(\|\nabla^3 f\|_{L^2}^2+1).
$$
The terms  $L_2$ and $L_3$ can be treated using a similar approach, giving
$$
L\leq \frac{\delta}{2}\|\Delta^2f\|_{L^2}+C(\|f\|_{W^{1,\infty}},\delta)(\|\nabla^3 f\|_{L^2}^2+1),
$$
which finishes the proof of the lemma.
\end{proof}
To conclude the proof of Theorem \ref{theo:Muskate} it remains to establish $H^s$ estimates with $s >3$. To this end we fix a subcritical index $s_0\in (2, 3]$ and apply Proposition \ref{prop:estDN} to have
\[
\| G(f)f\|_{H^{s-1}}\le \cF(\| f\|_{H^{s_0}})\| f\|_{H^s}.
\]
Then an $H^s$ estimate for \eqref{Muskate} yields
\bq
\begin{aligned}
\mez\frac{d}{dt} \| f(t)\|_{H^s}^2&\le \| G(f)f\|_{H^{s-1}}\| f\|_{H^{s+1}}-\eps \| \na f\|_{H^s}^2\\
&\le \cF(\| f\|_{H^{s_0}})\| f\|_{H^s}\| f\|_{H^{s+1}}-\eps \|  f\|_{H^{s+1}}^2+\eps \| f\|_{L^2}^2\\
&\le \left(\frac{1}{2\eps}\cF^2(\| f\|_{H^{s_0}})+1\right)\| f\|_{H^s}^2-\frac{\eps}{2} \|  f\|_{H^{s+1}}^2.
\end{aligned}
\eq
Using Gr\"onwall's inequality we deduce
\bq
\| f(t)\|_{H^s}^2+\eps \int_0^t\| f(\tau)\|_{H^{s+1}}^2d\tau\le \| f_0\|_{H^s}^2C\left(\| f\|_{L^\infty([0, t]; H^{s_0})}, t, \eps \right).
\eq
Finally, invoking the estimate \eqref{fL2bound} and \eqref{H3control} to control $\| f\|_{L^\infty([0, t]; H^{s_0})}\le \| f\|_{L^\infty([0, t]; H^3)}$, we obtain the global-in-time $H^s$ estimate
\bq
\| f(t)\|_{H^s}^2+\eps \int_0^t\| f(\tau)\|_{H^{s+1}}^2d\tau\le \| f_0\|_{H^s}^2C\left(\| f_0\|_{H^3}, t, \eps \right),\quad t>0.
\eq
This concludes the proof of Theorem \ref{theo:Muskate}.
\section{Proof of Theorem \ref{theo:main}}\label{section:proof}
For  $f_0\in W^{1, \infty}(\T^2)$ and $\eps\in (0, 1)$, we let $f_{\eps, 0}=f_0*\rho_\eps$, where $\rho_\eps$ denotes the standard mollifier at scale $\eps$.  By virtue of Theorem \ref{theo:Muskate}, the  $\eps$-Muskat equation
\bq\label{eq:fe0}
\p_t f_\eps=-G(f_\eps)f_\eps+\eps \Delta_x f_\eps,\quad f_\eps\vert_{t=0}=f_{\eps, 0}
\eq
has a unique global smooth solution. Note that by rescaling time we have set $\ka=1$. Using  Proposition \ref{prop:reformDN_b} we rewrite \eqref{eq:fe0} as
\bq\label{eq:fe}
\p_t f_\eps=p.v. \int_{\T^2}\left\{\nabla_{x,z} \Gamma(x-x',f_\eps(x)-f_\eps(x'))\wedge w_\eps(x')\right\}\cdot N_\eps(x)dx'+\eps \Delta_x f_\eps(x),
\eq
\bq\label{eq:tte}
\mez \tt_\eps(x)+p.v.\int_{\T^2} \na_{x, z}\Gamma(x- x', f_\eps(x)-f_\eps(x'))\cdot N_\eps(x') \tt_\eps(x')dx'=f_\eps(x),
\eq
\bq\label{eq:we}
w_\eps(x)=\partial_2\theta_\eps(x)(1,0,\partial_1f_\eps(x))-\partial_{1}\theta_\eps(x)(0,1,\partial_2f_\eps(x)),
\eq
where $N_\eps=(-\na_x f_\eps, 1)$. The maximum principle \eqref{max:Muskate} and Proposition \ref{prop:W1p} imply the  uniform bounds
\bq\label{uniformbounds:ftt}
\| f_\eps\|_{L^\infty(\Rr_+; W^{1, \infty}(\T^2))}\le C_1,\quad \| \tt_\eps\|_{L^\infty(\Rr_+; W^{1, 2+\delta_*}(\T^2))}\le C_2,
\eq
where $\delta_*>0$ depends only on $C_1$. Then using equation \eqref{eq:fe0} and Proposition \ref{prop:DNLp} we deduce the uniform bound
\bq
\| \p_tf_n\|_{L^\infty(\Rr_+; H^{-1}(\T^2))}\le C_3.
\eq
Fixing an arbitrary sequence $\eps_n\to 0$, we relabel $f_n\equiv f_{\eps_n}$, $\tt_n\equiv \tt_{\eps_n}$ and $w_n\equiv w_{\eps_n}$. Combining the above uniform bounds for $f_n$ with Arzela-Ascoli's theorem and  Aubin-Lions's lemma, we can extract a subsequence, still denoted by $n$, such that
 \bq\label{conv:fn}
  f_n \overset{\ast}{\rightharpoonup} f\quad\text{in } L^\infty(\Rr_+; W^{1, \infty}(\T^2)),\quad  f_n \to f\quad\text{in } C(\T^2\times [0, T])\quad \forall T>0
  \eq
  and
  \bq\label{conv:ttn}
\tt_n\overset{\ast}{\rightharpoonup} \tt\quad\text{in } L^\infty(\Rr_+; W^{1, 2+\delta_*}(\T^2)).
\eq
In particular, $f(0)=\lim_n f_n(0)=\lim_{n} f_{n, 0}=f_0$ in $C(\T^2)$.  With the uniform convergence $f_n\to f$ in $C(\T^2\times [0, T])$, we can argue as in Step 2 of the proof of Theorem 1.2 in \cite{DGN} to prove that $f$ is a viscosity solution and hence is unique by virtue of the comparison principle in Theorem \ref{theo:comparison:viscosity}. The remainder of this proof is devoted to show that $f$ satisfies the equation $\p_t f=-G(f)f$ in $L^\infty_tL^2_x$. This will be achieved by passing  to the limit in \eqref{eq:fe} and \eqref{eq:tte}.

We note that $w_n=(1,0,\partial_1f_n)\partial_2\theta_n-(0,1,\partial_2f_n)\partial_{1}\theta_n$ is uniformly bounded in $L^\infty(\Rr_+; L^{2+\delta_*}(\T^2))$, hence (along a subsequence)
\bq
w_n\overset{\ast}{\rightharpoonup} w\quad\text{in } L^\infty(\Rr_+; L^{2+\delta_*}(\T^2)).
\eq
We now identify the limit $w$. For any $\varphi\in C^\infty_c(\T^2\times \Rr_+)$, we can integrate by parts to have
\begin{align*}
\int_{\Rr_+}\int_{\T^2} w_n(x, t)\varphi(x, t)dxdt&=\int_{\Rr_+}\int_{\T^2} \vec{i} \partial_2\theta_n\varphi dxdt - \int_{\Rr_+}\int_{\T^2}\vec{k}(\partial_2\theta_n f_n\p_1\varphi+\p_1\partial_2\theta_n f_n\varphi) dxdt \\
&\quad- \int_{\Rr_+}\int_{\T^2} \vec{j} \partial_1\theta_n\varphi dxdt +\int_{\Rr_+}\int_{\T^2}\vec{k}(\partial_1\theta_nf_n\p_2\varphi+\p_2\partial_1\theta_nf_n\varphi)dxdt.
\end{align*}
The convergence in \eqref{conv:fn} and \eqref{conv:ttn} then yields
\begin{align*}
\lim_{n\to \infty} \int_{\Rr_+}\int_{\T^2} w_n(x, t)\varphi(x, t)dxdt&=\int_{\Rr_+}\int_{\T^2} \vec{i} \partial_2\theta\varphi dxdt - \int_{\Rr_+}\int_{\T^2}\vec{k}\partial_2\theta f\p_1\varphi dxdt \\
&\quad- \int_{\Rr_+}\int_{\T^2} \vec{j} \partial_1\theta\varphi dxdt +\int_{\Rr_+}\int_{\T^2}\vec{k}\partial_1\theta f\p_2\varphi dxdt.
\end{align*}
Therefore, in the sense of distributions we have $\lim_{n\to \infty}w_n= \vec{i} \partial_2\theta+\vec{k}\p_1(\partial_2\theta f)-\vec{j} \partial_1\theta-\vec{k}\p_2(\partial_1\theta f)$. It follows that
\bq\label{limiting:w}
\begin{aligned}
w=\vec{i} \partial_2\theta+\vec{k}\p_1(\partial_2\theta f)-\vec{j} \partial_1\theta-\vec{k}\p_2(\partial_1\theta f)&=\vec{i} \partial_2\theta+\vec{k}\partial_2\theta \p_1f-\vec{j} \partial_1\theta-\vec{k}\partial_2\theta \p_1f\\
&=(1,0,\partial_1f)\partial_2\theta-
(0,1,\partial_2f)\partial_{1}\theta,
\end{aligned}
\eq
where we have used the identity $\p_1(\partial_2\theta f)-\p_2(\partial_1\theta f)=\partial_2\theta \p_1f-\partial_2\theta \p_1f$ for $\tt\in H^1(\T^2)$ and $f\in W^{1, \infty}(\T^2)$; this  can be justified by approximating $\tt$ by smooth functions $\wt \tt_n\to \tt$ in $H^1(\T^2)$.

Applying Proposition \ref{prop:fundamentalsol} we can integrate by parts in the nonlinear term in \eqref{eq:tte} and obtain
\bq\label{ibp:tt}
\begin{aligned}
&-\int_{\T^2} \na_{x, z}\Gamma(x- x', f_n(x)-f_n(x'))\cdot N_n(x') \tt_n(x')dx'\\
&= \int_{\T^2}\left(\frac {x-x'}{4\pi|x-x'|^2} \frac {f_n(x)-f_n(x')}{(|x-x'|^2+|f_n(x)-f_n(x')|^2)^{1/2}}\right)\cdot \na \tt_n(x')dx'\\
&\quad+\int_{\T^2}\left(\frac{x-x'}{|x-x'|^2}F_1(x-x', f_n(x)-f_n(x'))\right)\cdot \na \tt_n(x')dx'\\
&\quad-\int_{\T^2}\frac{x-x'}{|x-x'|^2}\cdot F_2(x-x', f_n(x)-f_n(x'))\tt_n(x')dx'\\
&\quad+\int_{\T^2}F_3(x-x', f_n(x)-f_n(x'))\cdot \na \tt_n(x')dx'\\
&\quad -\int_{\T^2}(\dv_x F_3)(x-x', f_n(x)-f_n(x'))\tt_n(x')dx':=\sum_{j=1}^5I_n^j,
\end{aligned}
\eq
where the $F_j$'s are smooth and bounded. We also denote by $I^j$ the same integral as $I^j_n$ with $f_n$ and $\tt_n$ replaced by $f$ and $\tt$.

Let $p\in (1, 2)$ be the conjugate exponent of $2+\delta_*$. Fix $T \in (0, \infty)$ and let $\varphi$ be an arbitrary function in $L^p([0, T]; L^{2+\delta_*}(\T^2))$. We claim that
\bq\label{conv:Ijn}
\lim_{n\to \infty} \int_0^T\int_{\T^2} I_n^j(x, t)\varphi(x, t)dxdt=\int_0^T\int_{\T^2}I^j(x, t)\varphi(x, t)dxdt,\quad j\in\{1,\dots, 5\}.
\eq
Each integral $ \int_0^T\int_{\T^2} I_n^j(x, t)\varphi(x, t)dxdt$ assumes one of the forms
\[
\int_0^T\int_{\T^2}\int_{\T^2} F(x-x', f_n(x)-f_n(x'))\varphi(x, t)dx \tt_n(x', t)dx'dt
\]
and
\[
 \int_0^T\int_{\T^2}\int_{\T^2} F(x-x', f_n(x)-f_n(x'))\varphi(x, t)dx \na_x\tt_n(x', t)dx'dt.
\]
It follows from \eqref{conv:ttn} that $\tt_n\overset{\ast}{\rightharpoonup} \tt$ and $\na_x\tt_n\overset{\ast}{\rightharpoonup} \na_x\tt$ in $L^\infty([0, T]; L^{2+\delta_*}(\T^2))$. Thus \eqref{conv:Ijn} would follow if we can show that $g_n(x', t):=\int_{\T^2} F(x-x', f_n(x)-f_n(x'))\varphi(x, t)dx $  converges to $g(x', t):=\int_{\T^2} F(x-x', f(x)-f(x'))\varphi(x, t)dx$ in $L^p(\T^2\times [0, T])$. Indeed, the $I^j_n$'s with $j\in \{1, 2, 3\}$ have the most singular $F$ obeying (for $x\ne x'$)
\[
|F(x-x', f_n(x, t)-f_n(x', t))|\le \frac{1}{|x-x'|}\in L^p_{x}(\T^2)
\]
and $F(x-x', f_n(x, t)-f_n(x', t))\to F(x-x', f(x, t)-f(x', t))$ for all $(x, t)\in (\T^2\setminus\{x'\})\times [0, T]$ since $f_n\to f$ in $C(\T^2\times [0, T])$. Therefore, $g_n(x', t)\to g(x', t)$ pointwise by the dominated convergence theorem, and  we have
\[
|g_n(x', t)|\le C\| \varphi(\cdot, t)\|_{L^{2+\delta_*}(\T^2)}.
\]
Clearly $\| \varphi(\cdot, t)\|_{L^{2+\delta_*}(\T^2)}\in L^p(\T^2\times [0, T])$, so the dominated convergence theorem implies that $g_n\to g$ in  $L^p(\T^2\times [0, T])$ as claimed. This completes the proof of  \eqref{conv:Ijn}.  Consequently, we have the following weak convergence in $L^{2+\delta_*}([0, T]; L^p(\T^2))$:
\bq
\begin{aligned}
&-\int_{\T^2} \na_{x, z}\Gamma(x- x', f_n(x)-f_n(x'))\cdot N_n(x') \tt_n(x')dx'\\
&\wc\sum_{j=1}^5 I^j=-\int_{\T^2} \na_{x, z}\Gamma(x- x', f(x)-f(x'))\cdot N(x') \tt (x')dx',
\end{aligned}
\eq
where the last equality follows by reversing the integration by parts in \eqref{ibp:tt} for $I^j$. Therefore, passing to the limit in \eqref{eq:tte} (in the sense of distribution) we obtain
\bq\label{limiting:eqtt}
\mez \tt(x)+p.v.\int_{\T^2} \na_{x, z}\Gamma(x- x', f(x)-f(x'))\cdot N(x') \tt(x')dx'=f(x).
\eq
Next, we pass to the limit in \eqref{eq:fe}. For any $\varphi\in C^\infty_c(\T^2\times [0, T])$, using Proposition \ref{prop:GwN} we have

\[
\begin{aligned}
& \int_0^T\int_{\T^2}\varphi(x, t)\int_{\T^2} \left\{\nabla_{x,z}\Gamma(x-x',f_n(x)-f_n(x'))\wedge w_n(x')\right\}\cdot N_n(x)dx'dxdt\\
&=- \int_0^T\int_{\T^2}\varphi(x, t) \int_{\T^2} \left\{\nabla_{x,z}\Gamma(x-x',f_n(x)-f_n(x'))\wedge  N_n(x)\right\} \cdot w_n(x')dx'dxdt,
\end{aligned}
\]
where
\bq
\label{nonl:f}
\begin{aligned}
&\nabla_{x,z}\Gamma(x-x',f_n(x)-f_n(x'))\wedge  N_n(x)\\
&=(I_{1, n}, I_{2, n}, I_{3, n})+(J_{1, n}, J_{2, n}, J_{3, n})+H(x-x', f_n(x)-f_n(x'))\\
&\qquad+\vec{k}\left\{\dv_x(L^\perp(x-x', f_n(x)-f_n(x')))-(\dv_xL^\perp)(x-x', f_n(x)-f_n(x'))\right.\\
&\qquad\qquad\left. -(\p_{x_1}\eta_1)(x-x', f_n(x)-f_n(x'))\frac{x_2-x_2'}{|x-x'|^2}(f_n(x)-f_n(x'))A_n^{-\mez}\right.\\
&\qquad\qquad\left.+(\p_{x_2}\eta_1)(x-x', f_n(x)-f_n(x'))\frac{x_1-x_1'}{|x-x'|^2}(f_n(x)-f_n(x'))A_n^{-\mez}\right\}
\end{aligned}
\eq
and
\bq
A_n=|x-x'|^2+(f_n(x)-f_n(x'))^2,\quad L_j(x, z)=\frac{x_j}{|x|^2}\int_0^z\p_z\eta_1(x, s)s(|x|^2+s^2)^{-\mez}ds~(x\ne 0),
\eq
\bq
\begin{aligned}
&I_{1, n}=\p_{x_2}\Big(\eta_1(x-x', f_n(x)-f_n(x'))A_n^{-\mez}\Big),\quad I_{2, n}=-\p_{x_1}\Big(\eta_1(x-x', f_n(x)-f_n(x'))A_n^{-\mez}\Big),\\
& I_{3, n}=\dv_{x}\left(\eta_1(x-x', f_n(x)-f_n(x'))\frac{(x_2-x_2', -(x_1-x_1'))}{|x-x'|^2}(f_n(x)-f_n(x'))A_n^{-\mez}\right),
\end{aligned}
\eq
and
\bq
\begin{aligned}
J_{1, n}&=\p_2(K_1(x-x', f_n(x)-f_n(x'))),\quad J_{2, n}=\p_1(K_2(x-x', f_n(x)-f_n(x'))),\\
J_{3, n}&=\p_1(\wt{K}_1(x-x', f_n(x)-f_n(x')))+\p_2(\wt{K}_2(x-x', f_n(x)-f_n(x'))).
\end{aligned}
\eq
Here the functions $H$, $L_j$, $K_j$ and $\wt K_j$ are bounded and continuous. We note that in \eqref{nonl:f} the vectors $I$, $J$ and the term $\dv_x(L^\perp(x-x', f_n(x)-f_n(x')))$ are total derivatives in $x$, so that we can integrate by parts and put the $x$-derivatives on $\varphi$. For the vector $(I_{1, n}, I_{2, n}, I_{3, n})$ we obtain
\[
\int_0^T\int_{\T^2}\int_{\T^2} \wt I_{1, n}\p_{x_2}\varphi(x, t)w_{1, n}(x', t)- \wt I_{1, n}\p_{x_1}\varphi(x, t)w_{2, n}(x', t)+\wt I_{2, n}\cdot \na_x\varphi(x, t)w_{3, n}(x', t))dxdx'dt,
\]
where
\begin{align*}
&\wt I_{1, n}=\eta_1(x-x', f_n(x)-f_n(x'))A_n^{-\mez},\\
& \wt I_{2, n}=\eta_1(x-x', f_n(x)-f_n(x'))\frac{(x_2-x_2', -(x_1-x_1'))}{|x-x'|^2}(f_n(x)-f_n(x'))A_n^{-\mez}.
\end{align*}
Since for any $x\ne x'$, $|\wt I_{1, n}|\le C|x-x'|^{-1}\in L^p_x(\T^2)$  and $w_n\overset{\ast}{\rightharpoonup} w$ in $L^\infty([0, T]; L^{2+\delta_*}(\T^2))$, one can argue as in the proof of \eqref{conv:Ijn} to deduce the convergence of the preceding integral. The vector $(J_{1, n}, J_{2, n}, J_{3, n})$ can be treated similarly and in fact they are less singular since $K_j$ and $\wt{K}_j$ are bounded continuous functions. This is also the case for $H$. Next, the terms
\[
(\p_{x_j}\eta_1)(x-x', f_n(x)-f_n(x'))\frac{x_k-x_k'}{|x-x'|^2}(f_n(x)-f_n(x'))A_n^{-\mez},\quad j\ne k
\]
are bounded thanks to the support of $\na_x\eta_1$ and the uniform Lipchitz bound for $f_n$. 
 It remains to treat the two terms involving $L=(L_1, L_2)$. We note that the $L_j$'s (given by \eqref{def:AL}) are smooth on $(\Rr^2\setminus \{0\})\times \Rr$ and $L_j(x, z)=0$ for $(x, z)\in  B_{1/2}\setminus (\{0\}\times \Rr)$  or $|x|>1$ due to the presence of $\p_z\eta_1$. Then, since $\{f_n\}$ is bounded in $L^\infty_t W^{1, \infty}_x$, there exists $\delta>0$ small enough such that $L(x-x', f_n(x, t)-f_n(x', t))=0$ for all  $t\in [0, T]$ and $0<|x'-x|<\delta$ or $|x-x'|>1$. Consequently,
\begin{align*}
&\int_{\T^2} (\dv_xL^\perp)(x-x', f_n(x, t)-f_n(x', t))\varphi(x, t)dx \\
&=\int_{\delta<|x-x'|<1} (\dv_xL^\perp)(x-x', f_n(x, t)-f_n(x', t))\varphi(x, t)dx,
\end{align*}
which converges to
\[
\int_{\delta<|x-x'|<1} (\dv_xL^\perp)(x-x', f(x, t)-f(x', t))\varphi(x, t)dx
\]
 for all $(x, t)$ and is bounded by $C\| \varphi(\cdot, t)\|_{L^1(\T^2)}$. Finally, the integral involving  $\dv_x(L^\perp(x-x', f_n(x)-f_n(x')))$ can be treated analogously after integrating by parts to put the $x$-derivative on $\varphi$. We have proved that
 \[
 \int_{\T^2} \left\{\nabla_{x,z}\Gamma(x-x',f_n(x)-f_n(x'))\wedge w_n(x')\right\}\cdot N_n(x)dx'
 \]
converges in the sense of distribution to the same integral with $(f_n, w_n, N_n)$ replaced by $(f, w, N)$. Therefore, we can pass to the limit in the sense of distribution in \eqref{eq:fe} to have
\bq\label{limiting:eqf}
\p_t f=p.v. \int_{\T^2}\left\{\nabla_{x,z} \Gamma(x-x',f(x)-f(x'))\wedge w(x')\right\}\cdot N(x)dx'.
\eq
In view of \eqref{limiting:w}, \eqref{limiting:eqtt}, \eqref{limiting:eqf} and Proposition \ref{prop:DNLp}, we conclude that $f$ satisfies the equation $\p_t f=-G(f)f$ in $L^\infty_tL^{2+\delta_*}_x$. This concludes the proof of the main theorem.

\appendix

\section{Kernels of layer potentials on \texorpdfstring{$\T^2\times \Rr$}{}}\label{appendix:fundamentalsol}

Let $x=(x_1,x_2)\in \Rr^2$ and $z\in \Rr$. Then in $\Rr^3$, the fundamental solution of  Laplace's equation is given by
$$
\Gamma_{\Rr^3}(x,z)=-\frac 1 {4\pi \sqrt{|x|^2+z^2}}.
$$
\begin{lemm}
                    \label{lemA1}
In $\T^2\times\Rr$, the fundamental solution has the representation
\begin{equation}
                \label{eq2.25}
\Gamma(x,z)=-\eta_1(x,z)\frac{1}{4\pi \sqrt{|x|^2+z^2}}+\eta_2(x,z)+\eta_3(x,z)|z|
\end{equation}
for smooth functions $\eta_1$, $\eta_2$, and $\eta_3$ in $\T^2\times \Rr$, where
\begin{itemize}
\item $\eta_1$ is supported in $B_1$ and equals to $1$ in $B_{1/2}$;
\item $\eta_2$ decays exponentially along with its derivatives as $|z|\to \infty$;
\item  $\eta_3$ is supported in $\{|z|\ge 1\}$ and equals to $\frac 1 {2(2\pi)^2}$ when $|z|\ge 10$.
\end{itemize}
\end{lemm}
\begin{proof}
We follow the argument in \cite{Ammari}. By the definition of the fundamental solution, we have
\begin{equation}\label{eq1.34}
\begin{aligned}
  \Delta_{x,z} \Gamma(x,z)&=\sum_{m,n\in \Zz} \delta_0(x_1+m,x_2+n,z)\\
  &=\sum_{m,n\in \Zz} \delta_0(z)\delta_0(x_1+m)\delta_0(x_2+n)\\
&=\frac{1}{(2\pi)^2}\sum_{m,n\in \Zz}\delta_0(z)  e^{i (mx_1+nx_2)},
\end{aligned}
\end{equation}
where in the last equality we used the Poisson summation formula. Since $\Gamma$ is periodic in $x$ with period $1$, we also have
\begin{equation}
        \label{eq2.06}
\Gamma(x,z)=\sum_{m,n\in \Zz}\beta_{m,n}(z)e^{i(mx_1+nx_2)},
\end{equation}
which implies that
\begin{equation}\label{eq1.43}
  \Delta_{x,z} \Gamma(x,z)=\sum_{m,n\in \Zz}(\beta_{m,n}''(z)- (m^2+n^2))e^{i  (mx_1+nx_2)}.
\end{equation}
From \eqref{eq1.34} and \eqref{eq1.43}, we get the ODEs
$$
\beta_{m,n}''(z)- (m^2+n^2)\beta_{m,n}=\frac{1}{(2\pi)^2}\delta_0(z).
$$
Solving these ODEs and noting that $\beta_{m,n}$ is an even function by symmetry, we deduce that
\begin{equation}\label{eq1.48}
\beta_{0, 0}(z)=\frac 1 {2(2\pi)^2} |z|+c_0,\quad \beta_{m,n}(z)=-\frac 1 {2(2\pi)^2(m^2+n^2)^{1/2}} e^{(m^2+n^2)^{1/2}|z|}.
\end{equation}
Without loss of generality, we may assume that the constant $c_0=0$. It is easily seen that when $|z|\ge 1$,
\begin{equation}\label{eq1.58}
\sum_{m,n\in \Zz}\frac 1 {(m^2+n^2)^{1/2}} e^{-(m^2+n^2)^{1/2}|z|}\le Ce^{-c|z|},
\end{equation}
where $C,c>0$ are constants. Now we take two functions $\eta_1$ and $\eta_3$, which satisfies the conditions in the lemma. Let
$$
\eta_2(x,z)=\Gamma(x,z)+\eta_1(x,z)\frac{1}{4\pi \sqrt{|x|^2+z^2}}-\eta_3(x,z)|z|.
$$
Then in $B_1$, $\eta_2(x,z)=\Gamma(x,z)+\eta_1(x,z)\frac{1}{4\pi \sqrt{|x|^2+z^2}}$ and it is harmonic in $B_{1/2}$. Thus, $\eta_2$ is smooth in $B_1$. In $B_1^c$, by using \eqref{eq2.06} and \eqref{eq1.48},
$$
\eta_2(x,z)=\Gamma(x,z)-\eta_3(x,z) |z|=\left(\frac 1 {2(2\pi)^2}-\eta_3(x,z)\right)|z|+\sum_{(m,n)\neq (0,0)}\beta_{m,n}(z)e^{i 2\pi (mx_1+nx_2)},
$$
which is smooth and  decays exponentially as $|z|\to \infty$ in view of \eqref{eq1.58} and the condition of $\eta_3$. Finally, to see that all the derivatives of $\eta_2$ also decay exponentially as $|z|\to \infty$, it suffices to note that by its definition $\eta_2$ is harmonic when $|z|\ge 10$.
\end{proof}
In the following propositions, we discover delicate structures of the products
\[
\na_{x, z}\Gamma(x- x', f(x)-f(x'))\cdot N(x')\quad\text{and}\quad \nabla_{x,z} \Gamma(x-x',f(x)-f(x'))\wedge N(x),
\]
appearing in the integrals \eqref{DN:int} and \eqref{def:tt} used to express the Dirichlet-Neumann operator. These  structures  play an important role in obtaining the global solution in Theorem \ref{theo:main} as the limit of approximate solutions.
\begin{prop}\label{prop:fundamentalsol}
There exist a smooth bounded functions $F_1:  \Rr^3\to \Rr$ and  $F_2, F_3: \Rr^3\to \Rr^2$ with $F_j(\cdot, 0)=0$ such that
\bq
\begin{aligned}
&\partial_{N_{x'}}\Gamma(x-x',f(x)-f(x'))\\
&= \dv_{x'}\left(\frac {x-x'}{4\pi|x-x'|^2} \frac {f(x)-f(x')}{(|x-x'|^2+|f(x)-f(x')|^2)^{1/2}}\right)\\
&\quad+\dv_{x'}\left(\frac{x-x'}{|x-x'|^2}F_1(x-x', f(x)-f(x'))\right)+\frac{x-x'}{|x-x'|^2}\cdot F_2(x-x', f(x)-f(x'))\\
&\quad+\dv_{x'}F_3(x-x', f(x)-f(x'))+(\dv_x F_3)(x-x', f(x)-f(x'))
\end{aligned}
\eq
for $x\ne x'$, where $N_{x}=(-\nabla f(x),1)$.
\end{prop}
\begin{proof}
Denoting  by $\Gamma_{\Rr^3}(x, z)=\frac{1}{4\pi}(|x|^2+z^2)^{-\mez}$ the fundamental solution of Laplace's equation on $\Rr^3$, we have
\begin{equation}
                            \label{eq2.14}
\begin{aligned}
&\partial_{N_{x'}}\Gamma_{\Rr^3}(x-x',f(x)-f(x'))=\frac 1 {4\pi}\dv_{x'}\big(\frac {x-x'}{|x-x'|^2} \frac {f(x)-f(x')}{(|x-x'|^2+|f(x)-f(x')|^2)^{1/2}}\big),\quad x\ne x'.
\end{aligned}
\end{equation}
Indeed, the left-hand side of \eqref{eq2.14} is equal to
\begin{equation}
                \label{eq2.17}
\frac {-\nabla f(x')\cdot (x-x')+f(x)-f(x')} {4\pi (|x-x'|^2+|f(x)-f(x')|^2)^{3/2}}.
\end{equation}
On the other hand, since $\text{div}_{x'}\frac {x-x'}{|x-x'|^2}=0$ for $x\ne x'$, the right-hand side of \eqref{eq2.14} is equal to
$$
\frac 1 {4\pi}\frac {x-x'}{|x-x'|^2}
\Big(\frac {(x-x')(f(x)-f(x'))-|x-x'|^2\nabla f(x')} {(|x-x'|^2+|f(x)-f(x')|^2)^{3/2}}
\Big),
$$
which is equal to \eqref{eq2.17}.


In the case when the domain is $\T^2\times\Rr$, we can do a similar computation. By using \eqref{eq2.25},
$$
\partial_{N_{x'}}\Gamma(x-x',f(x)-f(x'))=I_1+I_2+I_3,
$$
where
\begin{align*}
I_1&=\eta_1(x-x',f(x)-f(x'))\partial_{N_{x'}}\Gamma_{\Rr^3}(x-x',f(x)-f(x'))\\
&\quad -\Gamma_{\Rr^3}(x-x',f(x)-f(x'))\nabla f(x')\cdot (\nabla_{x}\eta_1)(x-x',f(x)-f(x'))\\
&\qquad +\Gamma_{\Rr^3}(x-x',f(x)-f(x'))(\partial_{z}\eta_1)(x-x',f(x)-f(x'))\big):=I_{11}+I_{12}+I_{13},
\end{align*}
\begin{align*}
I_2&=-\nabla f(x')\cdot (\nabla_{x}\eta_2)(x-x',f(x)-f(x'))
+(\partial_{z}\eta_2)(x-x',f(x)-f(x')):=I_{21}+I_{22},
\end{align*}
and
\begin{align*}
I_3&=-\nabla f(x')\cdot (\nabla_{x}\eta_3)(x-x',f(x)-f(x'))|f(x)-f(x')|\\
&\quad +(\partial_{z}\eta_3)(x-x',f(x)-f(x'))|f(x)-f(x')|+\eta_3(x-x',f(x)-f(x'))\sgn(f(x)-f(x'))\\
&:=I_{31}+I_{32}+I_{33}.
\end{align*}
Next, we estimate each term above. First, note that $I_{13}$, $I_{22}$, and $I_{33}$ are bounded. To estimate $I_{21}$, we define $F(x,z)=\int_0^z \na_x\eta_2(x,s)\,ds$, which is a smooth bounded function because $\eta_2$ and its derivatives decay exponentially as $|z|\to \infty$.  Moreover, $F(\cdot, 0)=0$. Then we have
\begin{align*}
&\dv_{x'} F(x-x',f(x)-f(x'))\\
&=-(\dv_x F)(x-x',f(x)-f(x'))-(\partial_{z}F)(x-x',f(x)-f(x'))\cdot\nabla f(x')\\
&=-(\dv_x F)(x-x',f(x)-f(x'))-(\na_{x}\eta_2)(x-x',f(x)-f(x'))\cdot\nabla f(x'),
\end{align*}
which implies that
$$
I_{21}=\dv_{x'} F(x-x',f(x)-f(x'))+(\dv_x F)(x-x',f(x)-f(x')).
$$
Since $\eta_3(x,z)$ is a constant when $|z|\ge 10$, it is easily seen that $I_{32}$ is also bounded. The estimates of $I_{31}$ and $I_{12}$ are similar to that of $I_{21}$ by replacing $\na_{x}\eta_2(x,z)$ with the smooth functions $\na_{x}\eta_3(x,z) |z|$ and $\Gamma_{\Rr^3}(x,z)\na_x \eta_1(x,z)$, respectively.

Finally, we rewrite $I_{11}$ has
\begin{align*}
 I_{11}
&=\frac 1 {4\pi}
 \eta_1(x-x',f(x)-f(x'))\dv_{x'}\left(\frac{x-x'}{|x-x'|^2} \frac {f(x)-f(x')}{(|x-x'|^2+|f(x)-f(x')|^2)^{1/2}}\right)\\
 &=\frac 1 {4\pi}
 \dv_{x'}\left(\frac {x-x'}{|x-x'|^2} \frac {f(x)-f(x')}{(|x-x'|^2+|f(x)-f(x')|^2)^{1/2}}\right)\\
 &\quad+\frac 1 {4\pi}
 \dv_{x'}\left([\eta_1(x-x',f(x)-f(x'))-1]\frac {x-x'}{|x-x'|^2} \frac {f(x)-f(x')}{(|x-x'|^2+|f(x)-f(x')|^2)^{1/2}}\right)\\
 &\quad+\frac 1 {4\pi}\frac {x-x'}{|x-x'|^2} \frac {f(x)-f(x')}{(|x-x'|^2+|f(x)-f(x')|^2)^{1/2}}\cdot(\na_x\eta_1)(x-x', f(x)-f(x'))\\
  &\quad+\frac 1 {4\pi}\frac {x-x'}{|x-x'|^2} \frac {f(x)-f(x')}{(|x-x'|^2+|f(x)-f(x')|^2)^{1/2}}\p_z\eta_1(x-x', f(x)-f(x'))\cdot\na f(x')\\
  &=I_{110}+I_{111}+I_{112}+I_{113}.
\end{align*}
Note that $I_{111}$ is the divergence (in $x'$) of the product of $\frac {x-x'}{|x-x'|^2}$ and a smooth bounded function of $(x-x', f(x)-f(x'))$, where the latter vanishes when the second variable vanishes.  Since
$$
(\partial_z \eta_1)(x,z) \frac {z}{(|x|^2+z^2)^{1/2}}
$$
is a smooth function supported in $B_{1}\setminus B_{1/2}$, arguing as in $I_{21}$, we can rewrite $I_{113}$ as
\begin{align*}
I_{113}&=\frac {x-x'}{|x-x'|^2}(\partial_z \eta_1)(x-x',f(x)-f(x'))\cdot\nabla f(x') \frac {f(x)-f(x')}{4\pi(|x-x'|^2+|f(x)-f(x')|^2)^{1/2}}\\
&=\frac {x-x'}{|x-x'|^2}\cdot\nabla_{x'} H(x-x',f(x)-f(x'))+\frac {x-x'}{|x-x'|^2}(\na_{x}H)(x-x',f(x)-f(x'))\\
&=\dv_{x'} \left(\frac{x-x'}{|x-x'|^2}H(x-x',f(x)-f(x'))\right)+\frac {x-x'}{|x-x'|^2}\cdot(\na_{x}H)(x-x',f(x)-f(x'))
\end{align*}
for a smooth bounded function $H$ satisfying $H(\cdot, 0)=0$. We have used the fact that $\frac {x-x'}{|x-x'|^2}$ is divergence-free in $x'\ne x$. Finally, we note that the second term on the right-hand and $I_{112}$ both have the form
\[
\frac{x-x'}{|x-x'|^2}\cdot F_2(x-x', f(x)-f(x'))
\]
for some smooth $F_2$ satisfying $F_2(\cdot, 0)=0$.
\end{proof}
\begin{prop}\label{prop:GwN}
There exist bounded continuous vector field $H$ and functions $L_j$, $K_j$ and $\wt{K}_j$, $j\in\{1, 2\}$, such that for $x\ne x'$ we have
\bq
\begin{aligned}
&\na_{x, z}\Gamma(x-x', f(x)-f(x'))\wedge N(x)\\
&=(I_1, I_2, I_3)+(J_1, J_2, J_3)+H(x-x', f(x)-f(x'))\\
&\qquad+\vec{k}\left\{\dv_x(L^\perp(x-x', f(x)-f(x')))-(\dv_xL^\perp)(x-x', f(x)-f(x'))\right.\\
&\qquad\qquad\left. -(\p_{x_1}\eta_1)(x-x', f(x)-f(x'))\frac{x_2-x_2'}{|x-x'|^2}(f(x)-f(x'))A^{-\mez}\right.\\
&\qquad\qquad\left.+(\p_{x_2}\eta_1)(x-x', f(x)-f(x'))\frac{x_1-x_1'}{|x-x'|^2}(f(x)-f(x'))A^{-\mez}\right\},
\end{aligned}
\eq
where $L^\perp=(-L_2, L_1)$,
\bq\label{def:AL}
A=|x-x'|^2+(f(x)-f(x'))^2,\quad L_j(x, z)=\frac{x_j}{|x|^2}\int_0^z\p_z\eta_1(x, s)s(|x|^2+s^2)^{-\mez}ds~\,(x\ne 0),
\eq
and
\bq
\begin{aligned}
&I_1=\p_{x_2}\Big(\eta_1(x-x', f(x)-f(x'))A^{-\mez}\Big),\quad I_2=-\p_{x_1}\Big(\eta_1(x-x', f(x)-f(x'))A^{-\mez}\Big),\\
& I_3=\dv_{x}\left(\eta_1(x-x', f(x)-f(x'))\frac{(x_2-x_2', -(x_1-x_1'))}{|x-x'|^2}(f(x)-f(x'))A^{-\mez}\right),
\end{aligned}
\eq
and
\bq
\begin{aligned}
J_1&=\p_2(K_1(x-x', f(x)-f(x'))),\quad J_2=\p_1(K_2(x-x', f(x)-f(x'))),\\
J_3&=\p_1(\wt{K}_1(x-x', f(x)-f(x')))+\p_2(\wt{K}_2(x-x', f(x)-f(x'))).
\end{aligned}
\eq
\end{prop}
\begin{proof}
We write
\[
\na_{x, z}\Gamma(x, z)=\eta_1\na_{x, z}\Gamma_{\Rr^3}+\na_{x, z}\eta_1\Gamma_{\Rr^3}+\na_{x, z}\eta_2+\na_{x, z}\eta_3|z|+\na_{x, z}|z|\eta_3.
\]
For any function $\eta:\T^2\times \Rr\to \Rr$ we have
\[
\begin{aligned}
&(\na_{x, z}\eta)(x-x', f(x)-f(x'))\wedge N(x)\\
&=\vec{i}\{(\p_2\eta)+(\p_z\eta)\p_2f(x)\}-\vec{j}\{(\p_1\eta)+(\p_z\eta)\p_1f(x) \}+\vec{k}\{-(\p_1\eta)\p_2f(x)+(\p_2\eta)\p_1f(x)\},
\end{aligned}
\]
where $\eta$ and its derivatives are evaluated at $(x-x', f(x)-f(x'))$. Consequently, for any function $E:  \Rr^3\to \Rr$,
\bq\label{naetawN}
\begin{aligned}
&E(x-x', f(x)-f(x'))(\na_{x, z}\eta)(x-x', f(x)-f(x'))\wedge N(x)\\
&\quad=H^0()+\left(\p_2(K^0_z()), -\p_1(K^0_z()), \p_1(K^0_2())-\p_2(K^0_1())\right),
\end{aligned}
\eq
where $()=(x-x', f(x)-f(x'))$,
 \[
 K^0_\iota(x, z)=\int_0^z(E\p_\iota\eta)(x, s)ds,\quad\iota\in\{1, 2, z\},
 \]
 and each component of the vector field $H^0$ is a linear combination of  $(E\p_j\eta)(x-x', f(x)-f(x'))$ and $(\p_{j'}K^0_{\iota})(x-x', f(x)-f(x'))$ with $j, j'\in  \{1, 2\}$. Applying \eqref{naetawN} to
 \[
 (\eta, E)\in \{(\eta_1, \Gamma_{\Rr^3}), (\eta_2, 1), (\eta_3, |z|), (|z|, \eta_3)\}
 \]
 and  using properties of the $\eta_j$'s we find that $H^0$ and the $K^0_\iota$'s are continuous and bounded.  It remains to treat the most singular term $\eta_1\na_{x, z}\Gamma_{\Rr^3}$. It can be directly checked that
\bq\label{naGwN}
\na_{x, z}\Gamma_{\Rr^3}(x-x', f(x)-f(x'))\wedge N(x)=(I^0_1, I^0_2, I^0_3),
\eq
where, with $A=|x-x'|^2+(f(x)-f(x'))^2$, we have
$$
\begin{aligned}
&I^0_1=-A^{-\tdm}\{x_2-x_2'+\p_2f(x)(f(x)-f(x'))\}=\p_{x_2}A^{-\mez},\\
&I^0_2=A^{-\tdm}\{x_1-x_1'+\p_1f(x)(f(x)-f(x'))\}=-\p_{x_1}A^{-\mez},\\
&I^0_3=A^{-\tdm}\{\p_2f(x)(x_1-x_1')-\p_1f(x)(x_2-x_2')\}\\
&\qquad=\dv_{x}\left(\frac{(x_2-x_2', -(x_1-x_1'))}{|x-x'|^2}(f(x)-f(x'))A^{-\mez}\right).
\end{aligned}
$$
Consequently,
\[
(\eta_1\na_{x, z}\Gamma_{\Rr^3})(x-x', f(x)-f(x'))\wedge N(x)=(I'_1, I'_2, I'_3),
\]
where
$$
\begin{aligned}
&I'_1=\p_{x_2}(\eta_1()A^{-\mez})-(\p_{x_2}\eta_1)()A^{-\mez}-(\p_z\eta_1)()\p_{x_2}f(x)A^{-\mez},\\
&I'_2=-\p_{x_1}(\eta_1()A^{-\mez})+(\p_{x_1}\eta_1)() A^{-\mez}+(\p_z\eta_1)()\p_{x_1}f(x) A^{-\mez},\\
&I'_3=\dv_{x}\left(\eta_1()\frac{(x_2-x_2', -(x_1-x_1'))}{|x-x'|^2}(f(x)-f(x'))A^{-\mez}\right)\\
&\qquad-\left[(\p_{x_1}\eta_1)()+(\p_z\eta_1)()\p_{x_1}f(x)\right]\frac{x_2-x_2'}{|x-x'|^2}(f(x)-f(x'))A^{-\mez}\\
&\qquad +\left[(\p_{x_2}\eta_1)()+(\p_z\eta_1)()\p_{x_2}f(x)\right]\frac{x_1-x_1'}{|x-x'|^2}(f(x)-f(x'))A^{-\mez}.
\end{aligned}
$$
Since $\eta_1=1$ on $B_{1/2}$, $(\p_{x_j}\eta_1)()A^{-\mez}$ are smooth functions of $(x-x', f(x)-f(x'))$.

Let
\begin{align*}
&L_0(x, z)=\int_0^z\p_z\eta_1(x, s)(|x|^2+s^2)^{-\mez}ds,\quad L_j(x, z)=\frac{x_j}{|x|^2}\int_0^z\p_z\eta_1(x, s)s(|x|^2+s^2)^{-\mez}ds,
\end{align*}
where $L_0$ is smooth and bounded since $\supp \na \eta_1\subset B_1\setminus B_{1/2}$. Then, we can rewrite the $I'_j$'s as
\bq\label{etanaGwN:123}
\begin{aligned}
&I'_1=\p_{x_2}(\eta_1()A^{-\mez})-(\p_{x_2}\eta_1)()A^{-\mez}-\p_{x_2}(L_0())+(\p_{x_2}L_0)(),\\
&I'_2=-\p_{x_1}(\eta_1()A^{-\mez})+(\p_{x_1}\eta_1)() A^{-\mez}+\p_{x_1}(L_0())-(\p_{x_1}L_0)(),\\
&I'_3=\dv_{x}\left(\eta_1()\frac{(x_2-x_2', -(x_1-x_1'))}{|x-x'|^2}(f(x)-f(x'))A^{-\mez}\right)\\
&\qquad-\p_{x_1}(L_2())+(\p_{x_1}L_2)()+\p_{x_2}(L_1())-(\p_{x_2}L_1)()\\
&\qquad-(\p_{x_1}\eta_1)()\frac{x_2-x_2'}{|x-x'|^2}(f(x)-f(x'))A^{-\mez}+(\p_{x_2}\eta_1)()\frac{x_1-x_1'}{|x-x'|^2}(f(x)-f(x'))A^{-\mez}.
\end{aligned}
\eq
\end{proof}
\section{Invertibility of boundary layer potentials}\label{appendix:inverses}
We recall that the single layer potential $\cS$ is given by \eqref{singlelayer}, and the boundary single layer potential $S$ is given by \eqref{singlelayer2}, where $Sh=(\cS h)\vert_{\p\Omega_f}$.  By  \cite[Theorem 1.11]{Vorcheta}, we have
\begin{align}
&\p_N\cS h:=\lim_{X \xrightarrow{\text{i}} P}N(P)\cdot \na \cS h(X)=-(\mez I-K^*)h,\label{pNS}\\
&\p_N^c\cS h:=\lim_{X \xrightarrow{\text{e}} P}N(P)\cdot \na \cS h(X)=(\mez I+K^*)h,\label{pNcS}\\
&h=\p_N^c\cS h-\p_N\cS h. \label{jump:S}
\end{align}
On the other hand, the double layer potential $\cK$ and the boundary double layer potential $K$ are defined by \eqref{def:cK} and \eqref{def:K}, respectively. By \cite[Theorem 1.11]{Vorcheta}, the limits of $\cK h$ from the interior and the exterior of $\Omega_f$ are given by
\bq\label{lim:cK}
\lim_{(x, y)\xrightarrow{\text{i}} (x_0, f(x_0))}\phi(x, y)=(\mez I+K)h(x_0),\quad \lim_{(x, y)\xrightarrow{\text{e}} (x_0, f(x_0))}\phi(x, y)=-(\mez I-K)h(x_0).
\eq
For $g: \T^2\to \Rr$, we denote the integral mean of $g$ by 
\[
\langle g\rangle:=(2\pi)^{-2}\int_{\T^2}g.
\]
We will use the notation 
\[
L^2_0(\T^2)=\left\{h\in L^2(\T^2): \langle h\rangle=0\right\}.
\]
\begin{lemm}\label{lemm:K1}
\bq\label{K1}
 \int_{\T^2}K^*h=0\quad\forall h\in L^2(\T),\quad K1\equiv 0.
\eq
\end{lemm}
\begin{proof}
Using \eqref{pNS} and the divergence theorem, we find
\bq\label{div:Sh}
-\int_{\T^2} (\mez I-K^*)h dx=\int_{\T^2} \p_N\cS (x,f(x))h dx=\int_{\{-M<y<f(x)\}}\Delta \cS h dxdy+\int_{y=-M} \p_y \cS h dx=\int_{y=-M} \p_y \cS h dx
\eq
for any $M>\| f\|_{L^\infty(\T^2)}$.  By inspecting the formula \eqref{lemA1}  of $\Gamma$, we see that $\Gamma(x, z)\sim \frac{1}{2(2\pi)^2}|z|$ as $|z|\to \infty$, and hence
 \bq\label{behave:S} 
(\cS h)(x, y)\sim |y|\frac{1}{2(2\pi)^2}\int_{\T^2}h-(\sgn y)\frac{1}{2(2\pi)^2}\int_{\T^2}(fh)\quad  \text{as}~|y|\to \infty.
\eq
Thus, letting $M\to \infty$ in \eqref{div:Sh} yields
\bq\label{mean:mezK*} 
\int_{\T^2}(\mez I-K^*)h dx=\frac{1}{2(2\pi)^2}\int_{\T^2}hdx |\T^2|=\mez\int_{\T^2}h, 
\eq
which implies the first equality in \eqref{K1}. The second equality follows from duality.
\end{proof}
\begin{prop}\label{prop:invertL2}
The operators $\mez I\pm K^*$ are invertible on $L^2_0(\T^2)$, and the operators $\mez I\pm K$ and $\mez I\pm K^*$ are invertible on $L^2(\T^2)$. Moreover, the norm of these operators and their inverses depend only on $\| f\|_{\Lip(\T^2)}$. 
\end{prop}
\begin{proof}
Arguing as in the proof of \cite[Proposition 4.1]{DGN}, one can establish the invertibility of
$
\mez I\pm K^*
$
on \(L^2_0(\T^2)\), together with the stated quantitative bounds on the norms of their inverses. In the present setting, the argument requires the following Poincaré-type inequality for traces of the form \(H(x,f(x))\), expressed in terms of tangential derivatives:
\[
\left\|
H(\cdot,f(\cdot))
-
\left\langle H(\cdot,f(\cdot))\right\rangle
\right\|_{L^2(\T^2)}
\leq
C
\left\|
\nabla_T H(\cdot,f(\cdot))
\right\|_{L^2(\T^2)}.
\]

For any $g\in L^2(\T^2)$, we write $g=g_0+a$, $a=\langle g\rangle$, so that $g_0\in L^2_0(\T^2)$. To find an $h\in  L^2(\T^2)$ satisfying 
\begin{equation}
						\label{eq3.47}
	(\mez +K^*)h=g.
\end{equation} 
we similarly write $h=h_0+b$, where $b=\langle h\rangle$. 
Then \eqref{eq3.47} is equivalent to
\begin{equation}
	\label{eq3.48}
	(\mez +K^*)h_0+\mez b+bK^*1 =g_0+a.
\end{equation}
By \eqref{K1}, the first and the last term on the left-hand side above have zero mean, so \eqref{eq3.48} is equivalent to $b=2a$ and $h_0=(\mez +K^*)^{-1}(g_0-2aK^*1)\in L^2_0(\T^2)$.
Since $h$ is uniquely determined,  $\mez I+K^*$ is invertible on $L^2(\T^2)$, and so is $\mez I-K^*$ by an analogous argument, with the claimed quantitative bound for their norms. Finally, $\mez I\pm K$ are invertible on $L^2(\T^2)$ by duality. 
\end{proof}

Next we define $\bar S: L^2_0(\T^2)\to L^2_0(\T^2)\cap H^1(\T^2)$ by
$$
\bar S h=Sh-\langle Sh\rangle,
$$
so that $\langle \bar Sh\rangle =0$.  We also denote $\bar \cS h=\cS h-\langle Sh\rangle$ in $\T^2\times \Rr$.
\begin{prop}
				\label{propA.6}
$\bar S: L_0^2(\T^2)\to L^2_0(\T^2)\cap H^1(\T^2)$ is invertible, with the norm of its inverse depending only on $\| f\|_{\Lip(\T^2)}$.
\end{prop}

\begin{proof}
1. Injectivity. Suppose that $h\in L_0^2(\T^2)$ and $\bar Sh=0$.  As $\int_{\T^2}h=0$,  \eqref{behave:S} implies hat $(\bar \cS h)(x, y)$ tends to constants as $y\to \pm \infty$. Moreover, $\bar \cS h$ is harmonic in $\{y<f(x)\}$ and $(\bar \cS h)_{\p\Omega_f}=\bar Sh=0$. Hence, using the divergence theorem, we get
$$
\int_{{-M<y<f(x)}}|\nabla \bar \cS h|^2 dxdy
=-\int_{y=-M}\bar \cS h \partial_y (\bar \cS h) dx\to 0\quad\text{as}\,M\to \infty.
$$
Therefore, $\bar \cS h$ is constant in $\{y<f(x)\}$. Similarly, $\bar \cS h$ is constant in $\{y>f(x)\}$. It then follows from \eqref{jump:S} that $h\equiv 0$.


2. Surjectivity. Let $g\in L^2_0(\T^2)\cap H^1(\T^2)$. We shall prove that there exists $h\in L_0^2(\T^2)$ such that $\bar Sh=g$. Let $\phi=\cK(\mez  I+K)^{-1}g$ be the Poisson extension of $g$.  We have 
\begin{equation}
					\label{eq5.09}
\|\p_N\phi\|_{L^2(\Omega_f)}\le C(\| f\|_{W^{1, \infty}(\T^2)})\| g\|_{H^1(\T^2)}.
\end{equation}
Moreover, \eqref{pNS} implies that $-\p_N \cS (\mez I-K^*)^{-1}\p_N\phi=\p_N\phi$. Thus both $\tilde{\phi}:=-\cS (\mez I-K^*)^{-1}\p_N\phi$ and $\phi$ are harmonic functions in $\Omega_f$ with the same normal derivative on $\p \Omega_f$. Moreover, as $y\to -\infty$ we have
\begin{align*} 
\phi(x, y) \sim c:=-\frac{1}{2(2\pi)^2}\int_{\T^2}(\mez I+ K)^{-1}g,\\
\tilde{\phi}(x, y)\sim \tilde{c}|y|-c',\quad \tilde{c}=\frac{-1}{2(2\pi)^2}\int_{\T^2}(\mez I-K^*)^{-1}\p_N\phi,\quad c':=\frac{-1}{2(2\pi)^2}\int_{\T^2}f (\mez I-K^*)^{-1}\p_N\phi.
\end{align*}  
Since $\p_y\phi(x, y)\to 0$ as $y\to -\infty$, we have 
\[
\int_{\T^2} \p_N\phi dx=\lim_{M\to \infty}\left(\int_{\{-M<y<f(x)\}}\Delta \phi dxdy+\int_{\T^2} \p_y\phi(x, -M)dx\right)=0.
\] 
Thus restricted to $\p \Omega_f$, $\partial_N\phi, (\mez I-K^*)^{-1}\p_N\phi\in L^2_0(\T^2)$ and 
$\tilde{c}=0$  by  Proposition \ref{prop:invertL2}.  Hence, $\phi_1:=\tilde\phi -\phi$ is harmonic in $\Omega_f$, $\p_N\phi_1=0$ on $\p \Omega_f$,  and $\phi_1\sim  -(c+c')$ as $y\to -\infty$. It follows that 
\begin{multline*}
\int_{\Omega_f}|\na (\phi_1+c+c')|^2=\lim_{M\to \infty} \int_{\{-M<y<f(x)\}}|\na (\phi_1+c+c')|^2\\
=-\lim_{M\to \infty} \int_{\T^2}(\phi_1+c+c')\p_y(\phi_1+c+c')\vert_{y=-M} =0.
\end{multline*}
Thus $\phi_1\equiv -(c+c')$. Restricted to $\p \Omega_f$, since both $g$ and $-\bar S(\mez I-K^*)^{-1}\p_N\phi$ have zero mean, they must be identical. Therefore, $h=-(\mez I-K^*)^{-1}\p_N\phi$ satisfies $\bar Sh=g$.

Finally, the claimed quantitative bound follows from \eqref{eq5.09} and Proposition \ref{prop:invertL2}. The proposition is proved.
\end{proof}
For bounded Lipschitz domains of $\Rr^n$,  the Calderon identity $SK^*=KS$ holds (see the proof of Theorem 3.3 in \cite{Vorcheta}). For our unbounded domain $\Omega_f$, we have
\begin{lemm}
For $h\in L^2(\T^2)$, we have 
\bq\label{formula:KS}
(\mez I+K)Sh(x)=S(\mez I+K^*)h(x)+\frac14 f(x)\langle h\rangle-\frac14 \langle fh\rangle.
\eq
\end{lemm}
\begin{proof}
For $M>\| f\|_{L^\infty}$, we consider the truncated exterior domain 
\[
\Omega^M:=\{X=(x, y): x\in \T^2, f(x)<y<M\}
\]
and denote the exterior unit normal by $\nu$. For $v: \T^2\times \Rr$ harmonic, the Green formula for $\Omega^M$ reads 
\[
v(X)=\int_{\p \Omega^M}\p_{\nu(X')}\Gamma(X'-X)v(X')-\Gamma(X'-X)\p_{\nu(X')} v(X')dX',\quad X\in \Omega^M.
\]
Since $N=(-\na_xf, 1)$ is inward pointing for $\Omega^M$, we obtain
\bq\label{Green:v}
\begin{aligned}
	v(X)&=-\int_{\T^2}\p_{N(x')}\Gamma((x', f(x'))-X)v(x', f(x'))-\Gamma((x', f(x'))-X)\p^c_N v(x', f(x'))dx' \\
	&\quad+ \int_{\T^2}\p_y\Gamma((x', M)-X)v(x', M)-\Gamma((x', M)-X)\p_y v(x', M)dx'\\
	&=-(\cK v)(X)+(\cS\partial_N^c v)(X)\\
	&\quad+ \int_{\T^2}\p_y\Gamma(x'-x, M-y)v(x', M)-\Gamma(x'-x, M-y)\p_y v(x', M)dx'.
\end{aligned}
\eq
We set $v=\cS h$, fix $X=(x, y)\in \Omega^M$, and let $M\to \infty$. Using the asymptotic behavior of $\Gamma(x, y)$ and $(Sh)(x, y)$ (see \eqref{behave:S}), we find  
\begin{align*}
&\int_{\T^2}\p_y\Gamma(x'-x, M-y)v(x', M)-\Gamma(x'-x, M-y)\p_y v(x', M)dx'\\
&\quad\sim \int_{\T^2}  \frac{1}{2(2\pi)^2}\big(\mez \langle h\rangle M-\mez \langle fh\rangle\big)-\frac{1}{2(2\pi)^2}(M-y)\mez \langle h\rangle dx'=\frac14 y\langle h\rangle-\frac14 \langle fh\rangle. 
\end{align*}
Therefore, letting $M\to \infty$ in \eqref{Green:v} yields
\[
v(X)=-(\cK v)(X)+(\cS\partial_N^c v)(X)+\frac14 y\langle h\rangle-\frac14 \langle fh\rangle
\]
for any $X=(x, y)$, $y>f(x)$. Then, we let $X=(x, y)\to (x, f(x))$ from the exterior of $\Omega_f$ and invoke the formulas \eqref{pNcS} and  \eqref{lim:cK} to have 
\[
Sh(x)=(\mez I-K)Sh+S(\mez I+K^*)h(x)+\frac14 f(x)\langle h\rangle-\frac14 \langle fh\rangle, 
\]
which is equivalent to  \eqref{formula:KS}. 
\end{proof}

To prove that $\mez I+K$ is a bounded and invertible operator on $H^1(\T^2)$, for $\lambda\ge 0$ we define $S_\lambda: L^2(\T^2)\to H^1(\T^2)$ by
$$
S_\lambda h=Sh-\lambda \langle h \rangle.
$$
Using Lemma \ref{lemm:K1}, we find that $S_\ld$ satisfies the identity \eqref{formula:KS}: 
\begin{equation}
	\label{eq5.49}
(\mez I+K)S_\lambda h=S_\lambda (\mez I+K^*) h+\frac14 f(x)\langle h\rangle-\frac14 \langle fh\rangle.
\end{equation}
\begin{prop}
	\label{propA.7}
There is a sufficiently large $\lambda$ depending on $\|f\|_{\Lip}$ such that the operator $S_\lambda$ is invertible from $L^2(\T^2)$ to $H^1(\T^2)$ with the norm of its inverse depending only on $\| f\|_{\Lip}$.
\end{prop}
\begin{proof}
Let $g\in H^1(\T^2)$. Due to Proposition \ref{propA.6}, we can write
\begin{equation}
				\label{eq5.56}
g=\bar S \hat h+\langle g \rangle,
\end{equation}
where
$$
\hat h=\bar S^{-1}(g-\langle g \rangle)\in L^2_0(\T^2)\cap H^1(\T^2),\quad \|\hat h\|_{H^1}\le C(\|f\|_{Lip})\|g\|_{H^1(\T^2)}.
$$
To find a $h\in L^2(\T^2)$ such that $S_\lambda h=g$, we write $h=h_0+b$, where $h_0\in L^2_0(\T^2)$ and $b=\langle h \rangle$. Then using \eqref{eq5.56}, $S_\lambda h=g$ is equivalent to
$$
\bar S h_0+\langle S h_0\rangle+b S1-\lambda b
=\bar S \hat h+\langle g \rangle,
$$
which can be rewritten into
$$
\bar S h_0-\lambda b
=\bar S \hat h-b(S1-\langle S1\rangle)
+\langle g \rangle-(\langle S h_0\rangle+b \langle S1\rangle),
$$
Since the first term on the left-hand side and first two terms on the right-hand side have zero mean, the equation above is then equivalent to
$$
b=\lambda^{-1}\Big(-\langle g \rangle+\langle S h_0\rangle+b \langle S1\rangle\Big), \quad \bar Sh_0=\bar S \hat h-b(S1-\langle S1\rangle).
$$ 
By Proposition \ref{propA.6}, from the second equation above we have
\begin{equation}
	\label{eq6.11}
h_0=\hat h-b\bar S^{-1}(S1-\langle S1\rangle).
\end{equation}
Plugging this into the first equation, we get
$$
b=\lambda^{-1}\Big(-\langle g \rangle+\langle S (\hat h-b\bar S^{-1}(S1-\langle S1\rangle))\rangle+b \langle S1\rangle\Big).
$$ 
Recalling that the operator norms of $S$ and $\bar S^{-1}$ only depend on $\|f\|_{\Lip(\T^2)}$, we can uniquely solve for $b$ when $\lambda$ is sufficiently large depending on $\|f\|_{\Lip(\T^2)}$. After that, $h_0$ is determined using \eqref{eq6.11}. The proof above shows that $\|h_0\|_{L^2(\T^2)}$ and $|b|$ are bounded  by $C(\|f\|_{\Lip})\|g\|_{H^1(\T^2)}$ and such $h=h_0+b$ is unique. The proposition is proved.
\end{proof}
\begin{theo}\label{theo:Kinv}
(i) The operators $\mez I\pm K: H^1(\T^2)\to H^1(\T^2)$ are invertible, with the operator norms of the inverse depending only on $\|f\|_{\Lip}$.

(ii) There exists a small $\varepsilon\in (0,1)$ depending only on $\|f\|_{\Lip}$ such that for any $p\in (2-\varepsilon,2+\varepsilon)$, the operators $\mez I\pm K: W^{1,p}(\T^2)\to W^{1,p}(\T^2)$ are invertible, with the operator norms of the inverse depending only on $\|f\|_{\Lip}$.
\end{theo}
\begin{proof}
We only consider $\mez I+ K$ as $\mez I-K$ can be treated similarly. 

(i) By \eqref{eq5.49} and Proposition \ref{propA.7}, it suffices to show that for any $g\in H^1$, we can find a unique $h\in L^2(\T^2)$ such that 
\begin{equation}
	\label{eq5.49b}
	S_\lambda (\mez I+K^*) h+\frac14 (f\langle h\rangle- \langle f h\rangle)=g
\end{equation}
and $\|h\|_{L^2(\T^2)}\le C(\|f\|_{\Lip})\|g\|_{H^1(\T^2)}$. As in the proof of Proposition \ref{propA.7}, we write $g=g_0+\langle g\rangle$ and $h=h_0+b$, where $b=\langle h\rangle$. Then, \eqref{eq5.49b} can be rewritten to
\bq
\bar S((\mez I+K^*) h_0)+\langle S((\mez I+K^*) h_0)\rangle
+b S((\mez I+K^*) 1)-\mez\lambda b+\frac{1}{4}b(f-\langle f\rangle)-\frac14\langle fh_0\rangle=g_0+\langle g\rangle,
\eq
which is equivalent to
\begin{equation}
\left\{
\begin{aligned}
&\mez \lambda b=\langle S((\mez I+K^*) h_0)\rangle
+b \langle S((\mez I+K^*) 1)\rangle -\frac14\langle fh_0\rangle-\langle g\rangle\\
&\bar S((\mez I+K^*) h_0)=g_0-bS((\mez I+K^*) 1)+b \langle S((\mez I+K^*) 1)\rangle -\frac{1}{4}b(f-\langle f\rangle).
\end{aligned}
\right.
\end{equation}
Now from the second equation above as well as the invertibility of $\bar S$ and $\mez I+K^*$, we can express $h_0$ in terms of $b$:
\begin{equation}
					\label{eq8.26}
h_0=(\mez I+K^*)^{-1}\bar S^{-1}\Big(g_0-bS((\mez I+K^*) 1)+b \langle S((\mez I+K^*) 1)\rangle -\frac{1}{4}b(f-\langle f\rangle)\Big).
\end{equation}
Plugging this into the first equation and choosing $\ld$ sufficiently large, we obtaine a unique solution $b$, with $|b|\le C(\|f\|_{\Lip})\|g\|_{H^1(\T^2)}$. Finally, $h_0$ is determined by \eqref{eq8.26}.

(ii) We fix the constant $\lambda$ from (i). By the Shneiberg stability theorem (see \cite{Sh74} and the Appendix of \cite{ABES}),
it follows from the boundedness of $\mez I+K, \mez I+K^*$ on $L^p(\T^2)$ as well as their invertibility on $L^2(\T^2)$ shown in (i) that there exists a small $\varepsilon\in (0,1)$ depending only on $\|f\|_{\Lip}$ such that for any $p\in (2-\varepsilon,2+\varepsilon)$, $\mez I+K, \mez I+K^*$ are invertible on $L^p(\T^2)$, with the operator norms of the inverse depending only on $\|f\|_{\Lip}$. Similarly, $\bar S:L_0^p(\T^2)\to L_0^p(\T^2)\cap W^{1,p}(\T^2)$ and $S_\lambda: L^p(\T^2)\to W^{1,p}(\T^2)$ are invertible for the same range of $p$. Then we can repeat the proof in (i) with $L^p$ in place of $L^2$ to conclude the invertibility of  $\mez I+K$ on $W^{1,p}(\T^2)$.
\end{proof}
\section{The Dirichlet-Neumann operator for the sphere}\label{appendix:sphere}
The following proposition  provides an explicit integral formula for the Dirichlet-Neumann operator for the unit ball in $\Rr^3$.
\begin{prop}\label{prop:DNsphere}
Let $B_1$ be the unit ball in $\Rr^n$, and let $g\in W^{1, \infty}(\p B_1)$ be $C^{1, \alpha}$ at $x_*=(1, 0, \dots, 0)$. Let $u\in C(\overline{B_1})$ be the unique solution of the Dirichlet problem
\bq
\Delta u=0\quad\text{in } B_1,\quad u\vert_{\p B_1}=g.
\eq
Then we have
\bq\label{DN:sphere}
\lim_{r\to 1^{-}}\na u(rx_*)\cdot (rx_*)=\frac{-1}{n\alpha(n)}\int_{\p B_1}\frac{g(y)+g(\wt y)-2g(x_*)}{|y-x_*|^3}d\sigma(y),
\eq
where $\alpha(n)$ is the volume of $B_1$, and $\wt y=(y_1, -y_2, \dots, -y_n)$ if $y=(y_1, y_2, \dots, y_n)$.
\end{prop}
\begin{proof}
The following proof works in any dimension, so we present it in three dimensions for the sake of notational simplicity. By Poisson's formula, we have
\[
u(x)=\int_{\p B_1}K(x, y)g(y)d\sigma (y),\quad x\in B_1,\quad K(x, y)=c\frac{1-|x|^2}{|x-y|^3},\quad c=\frac{1}{3\alpha(3)}.
\]
If $g=1$, then $u=1$ and hence $\int_{\p B_1}\na_xK(x, y)d\sigma(y)=0$ for all $x\in B_1$. Consequently we can write
\bq
\begin{aligned}
\na u(x)\cdot x&=\int_{\p B_1}x\cdot \na_x K(x, y)[g(y)-g(x_*)]d\sigma(y)\\
&= -2c|x|^2\int_{\p B_1}\frac{1}{|x-y|^3}[g(y)-g(x_*)]d\sigma (y)\\
&\qquad-3c(1-|x|^2)\int_{\p B_1}\frac{x\cdot (x-y)}{|x-y|^5}[g(y)-g(x_*)]d\sigma (y)\\
&:=I+J.
\end{aligned}
\eq
We set $x=rx_*$ with $r\in (2/3, 1)$. Then, we can rewrite $J$ as
\bq
J=-3c(1-r^2)\int_{\p B_1}\frac{r^2-ry_1}{|rx_*-y|^5}[g(y)-g(x_*)]d\sigma (y)
\eq
and make  the change of  variables $y=(y_1, y_2, y_3)\mapsto \wt y=(y_1, -y_2, -y_3)$ to obtain
\begin{align*}
J&=-\tdm c(1+r)r\int_{\p B_1}\frac{(1-r)(r-y_1)}{|rx_*-y|^5}[g(y)+g(\wt y)-2g(x_*)]d\sigma (y)\\
&:=-\tdm c(1+r)r\int_{\p B_1}h_r(y)d\sigma(y).
\end{align*}
We shall prove that $\lim_{r\to 1}\int_{\p B_1}h_r(y)d\sigma(y)=0$, so that $\lim_{r\to 1}J=0$. To this end we split
\[
\int_{\p B_1}h_r(y)d\sigma(y)=\int_{\p B_1 \cap B_{1/5}(x_*)}h_r(y)d\sigma(y)+\int_{\p B_1 \setminus B_{1/5}(x_*)}h_r(y)d\sigma(y):=J_1+J_2.
\]
For $9/10<r<1/10$ and $y\in \p B_1 \setminus B_{1/5}(x_*)$, we have $|rx_*-y|\ge |x_*-y|-|rx_*-x_*|>1/10$, hence $J_2\to 0$ as $r\to 1$.  As for $J_1$,  if $y=(y_1, y_2, y_3)$ we let $y_*=(1, y_2, y_3)$ and write
\begin{align*}
g(y)+g(\wt y)-2g(x_*)=g(y_*)+g(\wt{y_*})-2g(x_*)+[g(y)-g(y_*)]+[g(\wt{y})-g(\wt {y_*})],
\end{align*}
where
\[
|g(y)-g(y_*)|\le (1-y_1)\|g\|_{\Lip},\quad |g(\wt{y})-g(\wt {y_*})|\le (1-y_1)\|  g\|_{\Lip}.
\]
On the other hand,  since $g$ belongs to $C(\p B)$ and  is $C^{1, \alpha}$ at $x_*$, and $y_*-x_*=-(\wt{y_*}-x_*)$, we have
\[
|g(y_*)+g(\wt{y_*})-2g(x_*)|\le C|y_*-x_*|^{1+\alpha}\le C(y_2^2+y_3^2)^{\frac{1+\alpha}{2}}.
\]
In addition, on $\p B_1 \cap B_{1/5}(x_*)$  we have $y_2^2+y_3^2=1-y_1^2=(1-y_1)(1+y_1)\ge 1-y_1$. It follows that
\bq\label{C1a:g:sphere}
|g(y)+g(\wt y)-2g(x_*)|\le C(y_2^2+y_3^2)^{\frac{1+\alpha}{2}}\quad\forall y\in \p B_1 \cap B_{1/5}(x_*).
\eq
Then, setting $\eps=1-r$, we obtain
\[
|h_r(y)|\le C\frac{\eps |1-y_1-\eps|}{[(1-y_1-\eps)^2+y_2^2+y_3^2]^\frac{5}{2}} (y_2^2+y_3^2)^{\frac{1+\alpha}{2}}\quad\forall y\in \p B_1 \cap B_{1/5}(x_*).
\]
On $\p B_1\cap B_{1/5}(x_*)$ we have that $y_1=[1-(y_2^2+y_3^2)]^\mez=y_1(y_2, y_3)$ is a graph, and
\[
[1+(\frac{\p y_1}{\p y_2})^2+(\frac{\p y_1}{\p y_3})^2]^\mez=\frac{1}{[1-(y_2^2+y_3^2)]^\mez}\le \frac{1}{(1-1/5^2)^\mez},
\]
hence
\[
J_1\le C\int_{y_2^2+y_3^2\le 1/5^2}\frac{\eps |1-y_1-\eps|}{[(1-y_1-\eps)^2+y_2^2+y_3^2]^\frac{5}{2}} (y_2^2+y_3^2)^{\frac{1+\alpha}{2}}dy_2dy_3.
\]
Switching to polar coordinates for $(y_2, y_3)$, we find
\[
J_1\le C\int_{\tt=0}^{2\pi}\int_{s=0}^{1/5}\frac{\eps (w(s)+\eps)}{[(w(s)-\eps)^2+s^2]^\frac{5}{2}} s^{1+\alpha}sdsd\tt,
\]
where  $w(s)=1-y_1(s)\in (0, 1)$ satisfies $w(s)(2-w(s))=s^2$, whence $w(s)\le s^2$. If $s\in (0, \eps)$, then $w(s)\le s^2< \eps^2$, hence $(w(s)-\eps)^2\ge \eps^2/2$ for small $\eps$.  Hence,
\begin{align*}
J_1\le C\int_{\tt=0}^{2\pi}\int_{s=0}^{\eps}\frac{\eps^2}{\eps^5} s^{2+\alpha}dsd\tt+C\int_{\tt=0}^{2\pi}\int_{s=\eps}^{1/5}\frac{\eps (s^2+\eps)}{s^5} s^{2+\alpha}dsd\tt\le C\eps^\alpha.
\end{align*}
Therefore, $\lim_{r\to 1}J_1=0$ as desired.

Regarding $I$ we symmetrize as above to have
\[
 I= -cr^2\int_{\p B_1}\frac{1}{|rx_*-y|^3}[g(y)+g(\wt y)-2g(x_*)]d\sigma (y),
 \]
 so that \eqref{DN:sphere} would follow if
 \bq\label{DN:sphere:limI}
 \lim_{r\to 1} \int_{\p B_1}\left[\frac{1}{|rx_*-y|^3}-\frac{1}{|x_*-y|^3}\right][g(y)+g(\wt y)-2g(x_*)]d\sigma (y)=0.
 \eq
 For $y\in \p B_1\setminus B_{1/5}(x_*)$, the $[\dots]$ term is bounded and tends to $0$ as $r\to 0$. Thus if suffices to consider the integral $I_0$ corresponding to  $y\in \p B_1\cap B_{1/5}(x_*)$. Since
 \[
\left||rx_*-y|^3-|x_*-y|^3\right|\le C(1-r)\left[(r-y_1)^2+(1-y_1)^2+y_2^2+y_3^2\right],
 \]
 setting $\eps=1-r$, $w=w(y_2, y_3)=1-y_1>0$ and using \eqref{C1a:g:sphere} we deduce
 \[
 I_0\le C\int_{y_2^2+y_3^2<1/5^2} \frac{\eps\left[(w-\eps)^2+w^2+y_2^2+y_3^2\right]}{[(w-\eps)^2+y_2^2+y_3^2]^\tdm[w^2+y_2^2+y_3^2]^\tdm}(y_2^2+y_3^2)^{\frac{1+\alpha}{2}}dy_2dy_3.
 \]
 Switching to polar coordinates yields
 \begin{align*}
 I_0&\le C\int_{\tt=0}^{2\pi}\int_{s=0}^{1/5} \frac{\eps\left[(w-\eps)^2+w^2+s^2\right]}{[(w-\eps)^2+s^2]^\tdm[w^2+s^2]^\tdm}s^{1+\alpha}s dsd\tt\\
 &\le C\int_{\tt=0}^{2\pi}\int_{s=0}^{1/5} \frac{\eps(\eps^2+s^2)}{[(w-\eps)^2+s^2]^\tdm[w^2+s^2]^\tdm}s^{1+\alpha}s dsd\tt,
 \end{align*}
 where we have used the fact that $w(s)\le s^2\le s$ in the domain of integration. For $s\le \eps$ we have $\eps-w(s)\ge \eps/2$, hence
 \begin{align*}
 I_0\le C\int_{s=0}^\eps \frac{\eps \eps^2}{\eps^3 s^3}s^{2+\alpha} ds+C\int_{s=\eps}^{1/5} \frac{\eps s^2}{s^6}s^{2+\alpha} ds\le C\eps^\alpha.
 \end{align*}
 This completes the proof of \eqref{DN:sphere:limI}.
\end{proof}
\begin{rema}\label{rema:rotation}
For any $x_\sharp\in \p B_1$, let $R$ be a rotation the maps $x_\sharp$ to $x_*=(1, 0, \dots, 0)$. Since $\Delta$ is rotation invariant, the function $\wt u(\wt x)=u(x)$, $x=R\wt x$, is harmonic in $B_1$, and we have $\na  u(r x_\sharp)\cdot (r x_\sharp)=\na \tilde u(r x_*)\cdot (r x_*)$, hence \eqref{DN:sphere} yields
\bq
\na  u(r x_\sharp)\cdot (r x_\sharp)=\frac{-1}{n\alpha(n)}\int_{\p B_1}\frac{g(Ry)+g(\wt{Ry})-2g(x_\sharp)}{|y-x_\sharp|^3}d\sigma(y).
\eq
\end{rema}

\vspace{.1in}
\noindent{\bf{Acknowledgment.}} H. Dong is partially supported by the NSF under agreements DMS-2055244, DMS-2350129, and the Simons Fellows Award 007638. F. Gancedo was partially supported by the ERC through the Starting Grant H2020-EU.1.1.-639227, by the MICINN (Spain) through the grants EUR2020-112271 and PID2020-114703GB-I00,
by the grant RED2022-134784-T funded by MCIN/AEI/10.13039/501100011033 and by the Junta de Andalucía through the grant P20-00566. F. Gancedo acknowledges support from IMAG, funded by MICINN through the Maria de Maeztu
Excellence Grant CEX2020-001105-M/AEI/10.13039/501100011033. H. Q. Nguyen was partially supported by NSF grants  DMS-2205734 and DMS-2205710.

\vspace{.1in}
\noindent{\bf{Data availability.}} This manuscript has no associated data

\end{document}